%% file: main3.tex
\documentclass[11pt]{amsart}
\usepackage{amssymb}
\usepackage{amsmath}
\usepackage{amsopn}
\usepackage{graphicx}
\usepackage{mathrsfs}
\usepackage{amsmath}
\usepackage{amsfonts}
\usepackage{mathtools}
\usepackage{amsthm}
\usepackage{comment}
\usepackage{mwe}
\usepackage{tikz}
\usepackage{pstool}
\usepackage{psfrag}
\usepackage{epstopdf}
\usepackage{pstricks}
\usepackage{tikz}
\usetikzlibrary{shapes,backgrounds,decorations}
\usepackage{tikz-cd}
\usetikzlibrary{matrix,arrows,decorations.pathmorphing}
\usepackage{textcomp}     
\usetikzlibrary{matrix,arrows}
\usepackage{todonotes}
\usepackage{marginnote}
\usepackage{lettrine}
\usepackage{booktabs}
\usepackage[all]{xy}
\usepackage{enumitem}
\usepackage{latexsym}
\usepackage{eepic}
\usepackage{epsfig}
\usepackage{pst-node}
\usepackage{lipsum}
\usepackage{hyperref}
\hypersetup{
  colorlinks   = true, 
  urlcolor     = black, 
  linkcolor    = blue, 
  citecolor   = red 
}
\usepackage[scale=0.75]{geometry}
\usepackage{amscd}
\usepackage{mathrsfs}
\usepackage[english]{babel}

\usepackage[utf8]{inputenc}
\usepackage[%
    style=alphabetic,sorting=nyt,
    maxnames=10,
    sortcites=true,doi=false,
    firstinits=true,hyperref,backend=bibtex]{biblatex}
    \renewbibmacro{in:}{%
      \ifentrytype{article}{}{\printtext{\bibstring{in}\intitlepunct}}}
\usepackage[babel]{csquotes}
\usepackage{guit}
\DeclareFieldFormat
[article,inbook,incollection,inproceedings,patent,thesis,unpublished]
  {title}{{#1\isdot}}
\DeclareFieldFormat{pages}{#1}

\usepackage{cancel}
\AtEveryBibitem{%
\ifentrytype{book}{
    \clearfield{url}%
    \clearfield{pages}
    \clearfield{isbn}
    \clearfield{urldate}%
    \clearfield{review}%
    \clearfield{series}
}{}
\ifentrytype{article}{
    \clearfield{url}%
    \clearfield{issn}
    \clearfield{doi}
    \clearfield{isbn}
    \clearfield{urldate}%
    \clearfield{review}%
}{}
\ifentrytype{collection}{
    \clearfield{url}%
    \clearfield{urldate}%
    \clearfield{review}%
    \clearfield{series}
}{}
\ifentrytype{incollection}{
    \clearfield{url}%
    \clearfield{urldate}%
    \clearfield{review}%
    \clearfield{series}
}{}
}

\allowdisplaybreaks
\usepackage[capitalise,noabbrev]{cleveref}

\numberwithin{equation}{section}

\newtheoremstyle{plain2}    
   {}            
   {}            
   {\itshape}    
   {}            
   {\bfseries}   
   {.}           
   {5pt plus 1pt minus 1pt}  
   {{\thmnumber{#1} \thmname{#2}{\thmnote{ (#3)}}}}          
\theoremstyle{plain2}

\newtheorem{theorem}[equation]{Theorem}

\newtheorem{cor}[equation]{Corollary}
\newtheorem{claim}[equation]{Claim}
\newtheorem{lemma}[equation]{Lemma}
\newtheorem{prop}[equation]{Proposition}

\newtheorem{defi}[equation]{Definition}

\newtheorem{prop/def}[equation]{Proposition/Definition}

\newtheoremstyle{definition2}    
   {}
   {}
   {\normalfont}
   {}
   {\bfseries}
   {.}
   {5pt plus 1pt minus 1pt}
   {{\thmnumber{#1} \thmname{#2}{\thmnote{#3}}}}
\theoremstyle{definition2}
\newtheorem{rem}[equation]{Remark}

\newtheorem{example}[equation]{Example}

\newtheorem*{theorem*}{Theorem}

\newtheoremstyle{stepstyle}
   {}     {}
   {\normalfont}
   {\parindent}
   {\itshape}
   {}
   {5pt plus 1pt minus 1pt}
   {{\thmname{#1} \thmnumber{#2}:{\thmnote{#3}}}}
\theoremstyle{stepstyle}

\newtheoremstyle{point}
   {}     {}
   {\normalfont}
   {}
   {\bfseries}
   {}
   {5pt plus 1pt minus 1pt}
   {{\thmname{#1}(\thmnumber{#2})\thmnote{ #3.}}}
\theoremstyle{point}

\input{macros.tex}
\title{Monge--Ampère measures on balanced polyhedral spaces}

\date{\today}

\thanks{The authors gratefully acknowledge the hospitality of the \emph{Mathematisches Forschungsinstitut Oberwolfach}, where part of this work was conducted. A.\,M.\,Botero was funded by the Deutsche Forschungsgemeinschaft (DFG, German Research Foundation) – Project-ID 491392403 – TRR 358. L. Pille-Schneider was supported by the SFB 1085 Higher Invariants (Universitaet Regensburg, funded by the DFG)}

\author[Botero]{Ana Mar\'ia Botero}
\address{Bielefeld University \\ Universit\"atsstra{\ss}e 25  \\ 133615 Bielefeld  \\ Germany }
\email{\href{mailto:abotero@math.uni-bielefeld.de}{abotero@math.uni-bielefeld.de}}

\author[Mazzon]{Enrica Mazzon}
\address{Centre de Mathématiques Laurent Schwartz, Ecole Polytechnique, CNRS, Institut Polytechnique de Paris, Palaiseau, France}
\email{e.mazzon15@alumni.imperial.ac.uk}

\author[Pille-Schneider]{Léonard Pille--Schneider}
\address{Fakultaet fuer Mathematik \\
Universitaet Regensburg \\Regensburg \\Germany}
\email{leonard.pille-schneider@mathematik.uni-regensburg.de}

\begin{document}
\maketitle
\begin{abstract}
We study classes of convex functions on balanced polyhedral spaces and establish various structural properties, including a compactness theorem for polyhedrally plurisubharmonic functions. Using tropical intersection theory, we construct Monge--Ampère measures, first associated with piecewise affine functions, and then we extend it to polyhedrally plurisubharmonic functions.

We investigate polyhedral Monge--Ampère equations on balanced polyhedral spaces via a variational approach, providing sufficient conditions for the existence of solutions as well as explicit counterexamples. Finally, we relate our framework to non-archimedean pluripotential theory and explore its connection with the non-archimedean Monge--Ampère equation.
\end{abstract}

\tableofcontents

\section{Introduction}

Pluripotential theory is a branch of complex analysis in several variables that studies plurisubharmonic functions and their properties. It extends classical potential theory, which deals with harmonic and subharmonic functions in the setting of one complex variable. It has proven to be a strong tool in complex geometry with several applications, notably to the study of the complex Monge--Ampère equations in relation to the Calabi--Yau problem, but also in other areas of research of complex algebraic geometry (see for instance \cite{BEGZ,BHJ17,BHJ19}), non-archimedean and Arakelov theory (see \cite{CLD,BK,yuan-zhang}), and mirror symmetry \cite{LiNA}. 

The main goal of this paper is to develop an analogous pluripotential theory on \emph{balanced polyhedral spaces} and to investigate Monge--Ampère equations in this context, building on the variational approach of Boucksom, Favre, and Jonsson in the non-archimedean setting.

\subsection{Convex functions on polyhedral spaces}
A polyhedral space $\tX \subseteq \R^n$ is a finite union of polyhedra, called faces; it is said to be balanced if one can assign a strictly positive real weight to each face of maximal dimension, in such a way that these weights satisfy a balancing condition (Definition \ref{def:balanced-poly}).
Pluripotential theory on $\tX$ studies spaces of convex-type functions on $\tX$. Since $\tX$ is generally not a convex subset of $\R^n$, several notions of convexity may be considered; see for instance \cite{BBS}. Of particular relevance for our purposes are the \emph{polyhedrally convex} functions.
A function $\phi \colon \tX \to \R$ is polyhedrally convex (in short $\PC$) if 
$
\phi(x) \leq \sum_{i \in I}\alpha_i \phi(x_i)
$
for any polyhedrally convex combination $\left(x, (x_i)_{i \in I}, (\alpha_i)_{i \in I}\right)$ in $\tX$; such a combination is a special type of convex combination $x=\sum_i \alpha_i x_i$ that respects the polyhedral structure of $\tX$. 

$\PC$ functions on $\tX$ may be viewed as a natural generalization of convex functions on $\R^n$. Indeed, polyhedral convexity is a local property, it implies continuity, and the restriction of a $\PC$ function to any face of $\tX$ is convex in the classical sense. Nevertheless, $\PC$ functions on $\tX$ do not, in general, extend to convex functions on $\R^n$, which makes this class of functions particularly interesting to study.

Within the space of $\PC$ functions, we single out those that are moreover piecewise affine (abbreviated $\PA$), namely functions that, after possibly refining the faces of $\tX$, are affine on each face. One may view $\PA$ functions as the polyhedral analogue of smooth functions in the complex setting. In particular, piecewise affine polyhedral convex (in short $\PAPC$) functions allow us to introduce the following spaces.
Fix a reference $\PA$ function $\gamma$, playing a role analogous to that of a reference line bundle in algebraic geometry. We define $\PSH(\tX,\gamma)$ and $\PCreg(\tX,\gamma)$ as the spaces of pointwise limits, respectively uniform limits, of decreasing sequences of $\PAPC$ functions which grow as $\gamma$, that is, functions $\phi$ such that $\phi=\gamma+O(1)$. We refer to such functions as plurisubharmonic and regularizable, respectively. 
These spaces satisfy the inclusions 
$$\PCreg(\tX, \gamma) \subseteq \PSH(\tX, \gamma) \subseteq \PC(\tX),$$
and are closed under taking maxima and convex combinations. Moreover, $\PSH$ functions satisfy the following compactness property.

\begin{theorem}[Theorem~\ref{thm:comp}]\label{th:comp-intro}
The space $\PSH(\tX, \gamma)/\R$ is compact for the topology of pointwise convergence.
\end{theorem}
Theorem~\ref{th:comp-intro} relies on an equicontinuity result for $\PSH(\tX,\gamma)$ functions proved in Section~\ref{sec:equi}. As we shall see later, Theorem \ref{th:comp-intro} is a key ingredient in the variational approach to solving a polyhedral analogue of the Monge--Ampère equation.

\subsection{The polyhedral Monge--Ampère operator}
Given a balanced polyhedral space $\tX$ of dimension $d$, the polyhedral Monge--Ampère measure of a $\PA$ function is a discrete measure on $\tX$, defined using tropical intersection theory. More precisely, let $\phi$ be a $\PA$ function and let $\Pi$ be a collection of faces such that $\phi$ is affine on each face of $\Pi$. The measure $\PMA(\phi)$ is supported on the zero-dimensional faces $v$ of $\Pi$, with weights $a_v$ given by the intersection product
\[
f^d \cdot [\tX] = \sum_{v \in \Pi^{(0)}} a_v v;
\]
see Section \ref{sec:intersection}.
Our second main result extends the polyhedral Monge--Ampère operator from the space of $\PA$ functions to $\PCreg(\tX,\gamma)$, for $\gamma$ positive at infinity (Definition \ref{def:posinf}) which we will assume for the rest of the introduction.

\begin{theorem}[Theorem \ref{thm:PMA PCreg}]
There exists a unique operator
    \[
    \PMA \colon (f_1, \ldots, f_d) \mapsto \PMA(f_1, \ldots, f_d) 
    \]
which assigns to any $d$-tuple in $\PCreg(\tX,\gamma)$ a positive Radon measure on $\tX$ of total mass $\deg(\gamma)$ (Definition \ref{def:degree}), extends the polyhedral Monge--Ampère operator on $\PAPC$ functions, and is continuous along decreasing sequences of regularizable $\PC$ functions. 
\end{theorem}

We establish two comparison results. First, for any $\PCreg$ function $\phi$ and for any maximal-dimensional face of $\tX$, the polyhedral Monge--Ampère measure $\PMA(\phi,\ldots,\phi)$ (which we will simply denote by $\PMA(\phi)$) coincides with the real Monge--Ampère measure of the restriction of $\phi$ to that face; see Proposition \ref{prop:comp MA}. Second, when $\tX$ has dimension one, the operator $\PMA$ on piecewise smooth $\PC$ functions agrees with the (unbounded version of the) Laplacian operator as defined by Baker and Faber in \cite{baker-faber}; see Proposition \ref{prop:laplacian}. 

Finally, we further extend the polyhedral Monge--Ampère operator to the space $\PSH(\tX,\gamma)$ (see Definition~\ref{def:MA-PSH}) and introduce an \emph{energy functional}
$$E\colon \PSH(\tX,\gamma)\to \R \cup \{-\infty\}$$
defined as an antiderivative of the Monge--Ampère operator, in analogy with the complex setting. This leads to the notion of $\gamma$-plurisubharmonic functions with finite energy, denoted by $\mathcal{E}^1(\tX,\gamma)$; see Definition~\ref{def:psh-finite-energy}.

\subsection{The polyhedral Monge--Ampère equation} 
We now turn our attention to the polyhedral Monge–Ampère equation, namely the problem of finding solutions to $\PMA(\cdot)=\mu$ for a given measure $\mu$ on $\tX$. Inspired by the variational approaches developed in \cite{BBGZ} and \cite{BoucksomFavreJonsson2015} in the complex and non-archimedean settings respectively, we introduce a functional $F_\mu$ on the space $\PSH(\tX,\gamma)$ and study conditions on the pair $(\tX,\gamma)$ under which a maximizer of $F_\mu$ yields a solution to the polyhedral Monge–Ampère equation.

The functional $F_\mu$ is defined in terms of the energy functional $E$, and the variational method relies on suitable differentiability properties of $E$, in particular of its composition with the envelope $P_\gamma$ (Definition~\ref{defi:envelope}). To this end, in Sections~\ref{sec:reg} and~\ref{sec:ortho} we introduce two structural properties of the pair $(\tX,\gamma)$, which we call \emph{regularity} and \emph{orthogonality} in analogy with the complex and non-archimedean settings, and prove the following result.
\begin{theorem}[Proposition \ref{prop:MA maxsol}] \label{theo:varintro}
Assume that $(\tX,\gamma)$ has the regularity and orthogonality properties. For any compactly supported Radon measure $\mu$ on $\tX$, the functional $F_{\mu}$ admits a maximizer $\phi \in \PSH(\tX, \gamma)$, which is a solution to $\PMA(\phi) = \mu$. Moreover if $\gamma$ is strictly convex, then $\phi \in \PCreg(\tX,\gamma)$. 
\end{theorem}
The regularity property implies in particular the equality $\PC(\tX,\gamma)=\PCreg(\tX,\gamma)$; see Theorem~\ref{th:regularisation}. This property is satisfied in dimension one (Proposition \ref{prop:env PA}), by complete polyhedral spaces, and by conical polyhedral spaces where one considers conical $\PC$ functions; see Remark \ref{rem:regularity}. The validity of the orthogonality property appears to be more delicate: we exhibit counterexamples already in dimension one, as well as in the special case of Bergman fans when $\gamma$ is not strictly convex; see Section \ref{sec:counter-example}. While such properties cannot be expected to hold for arbitrary balanced polyhedral spaces, determining the extent to which regularity and orthogonality are satisfied, and clarifying their geometric meaning, therefore appears as a natural problem that is likely to play a central role in the further development of pluripotential theory in the polyhedral setting.

In this direction, we introduce in dimension one the notion of a polyhedrally smooth space (Definition \ref{def:smooth}), show that the orthogonality property holds within this class (Proposition~\ref{prop:dim1 ortho}), thus prove the following result.
\begin{theorem}[Theorem \ref{theo:soldim1}]
Assume that $\tX$ is a connected, one-dimensional, polyhedrally smooth, balanced polyhedral space. For any positive at infinity, $\PAPC$ function $\gamma$ and any compactly supported Radon measure $\mu$ on $\tX$, there exists a unique (up to additive constant) function $\phi \in \PSH(\tX,\gamma)$ solving the Monge--Ampère equation $\PMA(\phi) = \mu.$ Moreover if $\gamma$ is strictly convex, then $\phi \in \PCreg(\tX,\gamma)$. 
\end{theorem}

This naturally leads to the question of whether the notion of polyhedrally smoothness can be extended so as to ensure both regularity and orthogonality in higher dimensions. Identifying classes of polyhedral spaces satisfying appropriate smoothness conditions thus remains an interesting problem for future investigation. 

\subsection{Relation to the non-archimedean Monge–Ampère equation} 

A non-archimedean analogue of complex pluripotential theory was developed by Boucksom, Favre, and Jonsson in \cite{BoucksomFavreJonsson}, building on earlier insights of Kontsevich and Tschinkel \cite{KT}; see also \cite{CLD, non-arch-stab, BurgosGilGublerJellKuennemann2021, BoucksomJonssonc} for further developments. This framework led to a non-archimedean version of the Monge--Ampère equation, for which existence and uniqueness of solutions were established in \cite{BoucksomFavreJonsson2015}; see also \cite{BGJM}.

The non-archimedean Monge--Ampère equation has become a fundamental tool in mirror symmetry, in particular in relation to the SYZ conjecture. The latter concerns degenerating families of Calabi--Yau varieties; we refer to \cite{StromingerYauZaslow} for its original formulation, to \cite{KontsevichSoibelman} for a non-archimedean interpretation, and to \cite{LiNA,LiSurvey} for a detailed discussion of the connections between the complex and non-archimedean settings.
In particular, let $X=(X_t)$ be a maximally degenerate family of Calabi--Yau varieties. Yang Li proved that the solution of a non-archimedean Monge--Ampère equation on the associated non-archimedean space $X^{\an}$ yields a metric version of the SYZ conjecture for $X_t$, provided that the solution is invariant under a non-archimedean retraction map \cite[Theorem 2.2]{LiNA}. In this way, a metric form of the SYZ conjecture is reduced to a non-archimedean problem. This perspective highlights the importance of the study of non-archimedean Monge--Ampère equations and motivates the search for alternative constructions of its solutions.

A first strategy in this direction was adopted in \cite{HultgrenJonssonMazzonMcCleerey2022,Li2024,AndreassonHultgren}, where it is shown that for certain classes of Calabi--Yau hypersurfaces the solution of the non-archimedean Monge--Ampère equation can be reconstructed from the solution of a Monge--Ampère-type equation on a canonical simplicial subset $\Sk(X)$ of $X^{\an}$, known as the essential skeleton. In a different direction, the approach developed in the present paper is based on the tropicalization of $X$, which carries the structure of a polyhedral space. We investigate the relation between solutions of a polyhedral Monge--Ampère equation on the tropicalization and solutions of a non-archimedean Monge--Ampère equation on $X^{\an}$. A precise formulation of this connection is given in Proposition \ref{prop:solMANA} for Calabi--Yau complete intersections in the sense of \cite{Gross2005}. Furthermore, in Example \ref{ex:elliptic} we analyze a case previously considered in \cite{Li2024} in which the author's method does not always apply, whereas the polyhedral framework we introduced allows one to establish the required invariance property.

Finally, we mention the work of Liu on totally degenerate abelian varieties. In \cite{Liu}, the non-archimedean Monge--Ampère equation is solved by reducing it to a real Monge–Ampère equation on a real torus, treated via Yau’s solution of the complex Monge--Ampère equation \cite{yau}, and then relating it to the non-archimedean setting.

\subsection{Future research directions}
From the perspective of applications to the SYZ conjecture, we plan to further investigate the relation between the polyhedral and the non-archimedean settings. In particular, we aim to characterize those polarized families of Calabi--Yau varieties whose tropicalizations satisfy the regularity and orthogonality properties introduced in this work.
In this direction, we mention the upcoming work \cite{APW1}, where the authors develop a theory of mixed differential calculus on abstract tropical varieties, combining both the Chow rings of the local fans and a continuous component. This allows them to define notions of semipositive forms and Monge--Ampère measures, which they can compare to the non-archimedean Monge--Ampère measures. In the sequel \cite{APW2} of this work, the authors study the tropical Monge--Ampère equation, and obtain Theorem \ref{th:comp-intro} and Theorem \ref{theo:varintro} independently. Moreover, they prove the regularity property under a strict positivity assumption on $\gamma$, as well as the uniqueness of the solution of the polyhedral Monge--Ampère equation, in the case where it exists.

We also note that in \cite{mihatsch} Mihatsch developed a tropical intersection theory on balanced polyhedral spaces extending both \cite{AllermannRau2010} and the theory of delta-forms in \cite{delta-forms}; see also \cite{BGK}. A systematic comparison between Mihatsch’s framework and the approach developed here could provide additional insight into the interaction between tropical intersection theory and polyhedral pluripotential theory, and will be the subject of future work.

Classes of $\PSH$ functions with growth conditions and finite energy play a central role in complex analysis, arithmetic geometry, and non-archimedean geometry. In the complex setting, such functions are crucial in the study of regularity properties for solutions of the complex Monge--Ampère equation (see \cite{DoVu,NGL}). In the non-archimedean setting, analogous classes have proved fundamental in the study of $K$-stability \cite{non-arch-stab}. In arithmetic intersection theory, $\PSH$ functions with finite energy and suitable growth conditions have been used to define heights of arithmetic varieties with respect to singular semipositive metrics (see \cite{BK, GY}). Motivated by these developments, we expect that the class $\mathcal{E}^1(\tX,\gamma)$ of polyhedral $\gamma$-$\PSH$ functions with finite energy will provide a natural framework for computing heights of arithmetic varieties endowed with singular semipositive metrics, particularly in situations with rich combinatorial structure where combinatorial techniques can be effectively exploited.

Finally, pluripotential theory on Bergman fans can be viewed as a new contribution to the geometric approach to matroid theory, which has seen substantial progress in recent years (see, e.g., \cite{AHK, matroids-lagrangian}). To the best of our knowledge, polyhedrally convex functions with growth conditions have not been systematically investigated in this context. We expect that the theory developed here will offer new tools to study refined combinatorial properties of matroids via pluripotential theory on their Bergman fans.

\subsection*{Acknowledgements} 
The authors would like to thank Omid Amini, S\'{e}bastien Boucksom, José Ignacio Burgos Gil, Walter Gubler, Mattias Jonsson, Martin Sombra, and Lyuhui Wu for inspiring discussions that contributed to the development of this work.

\section{Preliminaries on polyhedral spaces}

\subsection{Polyhedral complexes}

We recall some definitions from polyhedral and tropical geometry; see for instance \cite{MaclaganSturmfels2015,AllermannRau2010}. Let $N$ be a lattice in a finite dimensional real vector space. Denote by $M$ its dual, and $N_{R} \coloneqq N \otimes_{\Z} R$ for any ring $R$.

\begin{defi}[Rational polyhedron] \label{defi:polyhedral}
    A rational polyhedron in $N_\R$ is a subset $\sigma \subseteq N_\R$ defined by finitely many affine integral equalities and inequalities, i.e., 
    \begin{align*}
    \sigma = \big{\{} n \in N_\R \,| \; & f_1(n)+b_1=\ldots = f_r(n)+b_r=0, \\
    & f_{r+1}(n)+b_{r+1} \geq 0,\ldots, f_N(n) + b_N\geq 0 \big{\}}  
    \end{align*} 
    for some linear forms $f_1, \ldots, f_N \in M$ and $b_i \in \Z$. 
    
    We define 
    \begin{itemize}
        \item $N_{\sigma, \R}$ the linear space spanned by all $x-y$ for $x,y \in \sigma$;
        \item $N_\sigma = N_{\sigma, \R} \cap N$ the induced sublattice, and $M_\sigma$ its dual lattice;
        \item $V_{\sigma, \R}=x+N_{\sigma, \R}$, $x \in \sigma$, the smallest affine subspace of $N_\R$ containing $\sigma$;
        \item the relative interior of $\sigma$, denoted by $\relint(\sigma)$ is the interior of $\sigma$ as a subset of $V_{\sigma, \R}$;
        \item the dimension of $\sigma$ as the dimension of $N_{\sigma, \R}$.
    \end{itemize}
\end{defi}

\begin{defi} [Rational polyhedral complex] 
    A rational polyhedral complex is a finite collection $\Pi$ of rational polyhedra satisfying two conditions: 
    \begin{itemize}
        \item if $\sigma$ is in $\Pi$, then so is any face of $\sigma$; 
        \item if $\sigma$ and $\sigma'$ are in $\Pi$, then $\sigma \cap \sigma'$ is either empty or a face of both $\sigma$ and $\sigma'$. 
    \end{itemize}
     
    The polyhedra of $\Pi$ are called \emph{faces} and we write $\tau \prec \sigma$ whenever $\tau$ is a face of $\sigma$. We denote the set of all $k$-dimensional faces of $\Pi$ by $\Pi^{(k)}$. The union of all faces in $\Pi$ is denoted $|\Pi| \subseteq N_\R$ and called the \emph{support} of $\Pi$.
    
    The \emph{dimension} of $\Pi$ is defined as the maximum of the dimensions of the faces of $\Pi$, and $\Pi$ is \emph{pure dimensional} if every inclusion-maximal face in $\Pi$ has the same dimension. 
\end{defi}

\begin{defi}[Refinement] 
    A \emph{refinement} $\Pi'$ of a rational polyhedral complex $\Pi$ is a rational polyhedral complex such that 
    \begin{itemize}
        \item $|\Pi'|=|\Pi|$
        \item for every $\sigma' \in \Pi'$, there exists $\sigma \in \Pi$ such that $\sigma' \subseteq \sigma$;
    \end{itemize} 
    we write $\Pi' \geq \Pi$.
    Two rational polyhedral complexes are \emph{equivalent} if they admit a common refinement.
\end{defi}

\begin{defi}[Rational polyhedral space]
\label{def:poly} \hfill
\begin{itemize}
    \item Let $\tX \subseteq N_{\R}$ be a topological space. A rational polyhedral complex $\Pi$ is a \emph{rational polyhedral structure on} $\tX$ if $|\Pi| = \tX$. 
    \item A \emph{polyhedral space} $\tX \subseteq N_{\R}$ is a topological space together with an equivalence class of rational polyhedral structures on $\tX$.
    \item A polyhedral space $\tX$ is \emph{of dimension $d$} (respectively of pure dimension $d$) if some (and hence any) representative of the equivalence class of rational polyhedral structures on $\tX$ has dimension $d$ (respectively pure dimension $d$).
\end{itemize}    
\end{defi}
\begin{rem}
From now on, we will only consider rational polyhedral complexes, so that we will omit the word ``rational'' and simply refer to them as polyhedral complexes. Moreover, we will assume our polyhedral complexes are pure dimensional. 
\end{rem}

\begin{defi}[Restriction] 
    For any polyhedral complex $\Pi$ and any subset $S \subseteq N_\R$, the restriction of $\Pi$ to $S$ is the set of polyhedra
    $$
    \Pi|_S \coloneqq \{ \sigma \in \Pi \,|\, \sigma \cap S \neq \emptyset \}.$$
    Note that in general this is not a polyhedral complex, but it is if $S$ a polyhedron.
\end{defi}

Let $\Pi$ be a polyhedral structure on $\tX$, and denote by $\Pi_b$ the subcomplex of bounded faces of $\Pi$. By analogy with the non-archimedean case we write $\Sk(\Pi)\coloneqq \lvert \Pi_b \rvert \subseteq \tX$, and call it the combinatorial skeleton of $\Pi$ as in \cite[Appendix A]{GublerJellKuennemannEtAl}.

\begin{defi} [Simplicial polyhedral complex] 
\label{defi:regular}
    A polyhedron $\sigma \subseteq N_{\R}$ of dimension $k$ is \emph{simplicial} if
    $\sigma = C +\tau ,$
    where $C$ is an $l$-simplex and $\tau$ is a simplicial cone of dimension $(k-l)$. A polyhedral complex is simplicial if all its faces are simplicial. 
\end{defi}

This means that 
$\sigma = \conv(p_0,\ldots,p_l) + \Cone(v_{l+1},\ldots,v_k)$
while having dimension $k$, so that each $x \in \sigma$ has a \emph{unique} expression
\begin{equation}\label{eq:eqn-cone}
x = \sum_{i=0}^l \theta_i p_i + \sum_{j=l+1}^k \lambda_j v_j,
\end{equation}
where $\theta_i, \lambda_j \ge 0$ and $\sum_i \theta_i =1$.

\begin{lemma} \label{lem:regular}
Any polyhedral structure on $\tX$ admits a simplicial refinement.
\end{lemma}
\begin{proof}
    Let $\Pi$ be a polyhedral structure on $\tX \subseteq N_{\R}$, and let $\Sigma$ be the fan in $N_{\R} \times \R_{\ge 0}$ whose cones are cones over $\tau \times \{1\}$ whenever $\tau$ is a face of $\Pi$. Let $\tilde{\Sigma}$ be a simplicial subdivision of $\Sigma$ (as fans, see e.\,g.\,\cite[Theorem 11.1.7]{CoxLittleSchenck2011}), then $\Pi' \coloneqq \tilde{\Sigma} \cap (N_{\R} \times \{1 \})$ is the required simplicial refinement.
\end{proof}

\begin{defi}[Retraction]\label{def:retraction}
    A simplicial cone has a canonical retraction onto its maximal bounded face, given explicitly by $r(x) = \sum_i \theta_i p_i$ using the expression in \eqref{eq:eqn-cone}. If $\Pi$ is a simplicial polyhedral structure on $\tX$, these glue naturally into a continuous retraction
    $$r_{\Pi} \colon \tX \to \Sk(\Pi).$$
\end{defi}

\begin{defi}[Recession cone and recession pseudo-fan]\label{def:recession} \hfill
\begin{enumerate}
\item
If $\sigma \subset N_{\mathbb{R}}$ is a polyhedron, its
\emph{recession cone} is
\[
\operatorname{rec}(\sigma)
 \;\coloneqq\;
 \{\, v \in N_{\mathbb{R}}
      \mid x + t v \in \sigma \text{ for all } x\in\sigma,\ t \ge 0 \,\}.
\]

\item
The \emph{recession pseudo-fan} of $\tX$ with respect to a polyhedral structure $\Pi$, denoted $\operatorname{rec}(\Pi)$, is the
collection of recession cones of all faces of $\Pi$:
\[
\operatorname{rec}(\Pi)
 \;\coloneqq\;
 \{\, \operatorname{rec}(\sigma) \mid \sigma \in \Pi \,\}.
\]
\end{enumerate}
\end{defi}

\begin{rem}
The recession pseudo-fan $\rec(\Pi)$ is not a polyhedral complex in general: it may fail to be closed under taking faces or intersections (see \cite[Proposition 3.15]{BurgosSombra2010}). However, one can always find a refinement $\Pi' \geq \Pi$, such that $\rec(\Pi')$ is a fan. 
\end{rem}

\subsection{Compactifications of polyhedral spaces}\label{subsec:compactification}

\begin{defi}\label{def:compactification} Let $\Pi$ be a polyhedral structure on $\tX$ such that the recession pseudo-fan $\rec(\Pi)$ is a fan, and let $\Sigma$ be a fan in $N_{\R}$ containing $\on{rec}(\Pi)$. The \emph{compactification $\overline{\tX}_\Pi$ of $\tX$ with respect to $\Pi$} is the closure of $\tX$ inside the tropical toric variety $N_{\Sigma}$. 
\end{defi}

We will prove that $\overline{\tX}_\Pi$ is indeed compact and independent of the choice of $\Sigma$.
We refer to \cite[Section 4.1]{BurgosPhilipponSombra} or \cite[Section 3]{Paynea} for the definition of the topological space $N_{\Sigma}$. In particular, we recall that $N_{\Sigma}$ admits a natural stratification indexed by cones of $\Sigma$ such that, for any $\gamma \in \Sigma$, the corresponding stratum $N(\gamma)$ is identified with the quotient vector space $N_{\R}/N_{\gamma, \R}$, equipped with the quotient topology.

\begin{lemma}\label{lem:closure}
Let $\Pi$, $\Sigma$ and $\overline{\tX}_\Pi$ as in Definition \ref{def:compactification}. 
Let $\sigma \subset \Pi$ be an unbounded face.
If the closure $\overline{\sigma} \subset N_{\Sigma}$ meets the stratum $N(\gamma)$ corresponding to a cone $\gamma \in \Sigma$, then
\[
\gamma \subseteq \rec(\sigma).
\]
\end{lemma}

\begin{proof}
Write $\sigma = C + \tau$, where $C \subset N_{\mathbb{R}}$ is a bounded polytope and $\tau \coloneqq \rec(\sigma)$. By \cite[Lemma 3.9]{OssermanRabinoff2013}, $\overline{\sigma}$ intersects the stratum $N(\gamma)$ if and only if $\tau$ intersects the relative interior of $\gamma$. As $\tau$ and $\gamma$ are cones of the same fan $\Sigma$, this implies that $\gamma\subseteq \tau$.
\end{proof}

\begin{cor} \label{cor:compactification}
The space $\overline{\tX}_\Pi$ does not depend on the choice of $\Sigma$ in Definition \ref{def:compactification}. In particular, $\overline{\tX}_\Pi$ is compact.
\end{cor}

\begin{proof}
Let $\Sigma$ be a fan containing $\rec(\Pi)$. By Lemma \ref{lem:closure}, the closure $\overline{\tX}_\Pi$ of $\tX$ in $N_{\Sigma}$ is the union of the closures of faces of $\tX$ inside the strata $N(\tau)$ for $\tau \in \rec(\Pi)$. This is in particular a subset of $\bigsqcup_{\tau \in \rec(\Pi)} N(\tau)$ and is independent of the choice of $\Sigma$. Moreover, by \cite[III, Theorem 2.8]{ewald}, there exists a complete fan $\Sigma'$ containing $\rec(\Pi)$. Hence the space $\overline{\tX}_\Pi$ is the closure of $\tX$ inside the compact space $N_{\Sigma'}$, and is thus compact.
\end{proof}

\begin{prop} \label{prop:disjoint}
Let $\Pi$, $\Sigma$ and $\overline{\tX}_\Pi$ as above. 
If $\sigma$ and $\sigma'$ are two disjoint unbounded faces in $\Pi$, then their closures in $N_{\Sigma}$ are disjoint.
\end{prop}
\begin{proof}
It suffices to consider the case
$\operatorname{rec}(\sigma) = \operatorname{rec}(\sigma') = \tau.$
By the Lemma \ref{lem:closure}, $\overline{\sigma}$ and $\overline{\sigma}'$ meet only the
strata $N(\gamma)$ with $\gamma \subseteq \tau$.  It therefore suffices to
show that they do not meet in $N(\tau)$.

The stratum $N(\tau)$ is naturally identified with 
$N_{\mathbb{R}} /
N_{\tau, \R}$.  The image of
$\sigma = C + \tau$ in $N_{\mathbb{R}} / N_{\tau, \R}$ is the projection of
the bounded set $C$, hence a bounded polytope; similarly for $\sigma'$.
If the closures $\overline{\sigma}$ and $\overline{\sigma}'$ met in $N(\tau)$, then the projected images of $\sigma$ and $\sigma'$ in $N_{\mathbb{R}}/N_{\tau, \R}$ would intersect. Thus there exist points $c \in \sigma$ and $c' \in \sigma'$ such that $c - c' \in N_{\tau, \R}$.  Since $\tau - \tau = N_{\tau, \R}$, we may write $c - c' = t_1 - t_2$ for some $t_1, t_2 \in \tau$.  Rearranging, we obtain $c + t_2 = c' + t_1$, but $c+t_2 \in \sigma$ and $c'+t_1 \in \sigma'$, so $\sigma$ and $\sigma'$ intersect, which is a contradiction. Hence $\overline{\sigma}$ and $\overline{\sigma}'$ do not meet in $N(\tau)$.
\end{proof}

We denote by $\PA_b(\overline{\tX}_\Pi)$ the set of bounded piecewise affine functions on $\tX$ which extend continuously to $\overline{\tX}_\Pi$. We have the following density result.

\begin{prop}\label{prop:bounded-dense}
    Let $\Pi$ be a simplicial polyhedral structure on $\tX$, and $\overline{\tX}_\Pi$ the compactification of $\tX$ with respect to $\Pi$. The subset $\PA_b(\overline{\tX}_\Pi)$ is dense in $\mathcal{C}^0(\overline{\tX}_\Pi)$.
\end{prop}
\begin{proof}
Let $\Sigma$ be a complete simplicial fan containing $\rec(\Pi)$ as a subfan. We can reduce to the case where $\tX = N_{\R}$ and $\overline{\tX}_\Pi = N_{\Sigma}$. Indeed, let $\phi \in \mathcal{C}^0(\overline{\tX}_\Pi)$, then by the Tietze extension theorem there exists an extension $\phi' \in \mathcal{C}^0(N_{\Sigma})$ of $\phi$. If we can find a sequence $(\phi_k)_k \in \PA_b(N_{\Sigma})$ converging uniformly to $\phi'$, then restricting back to $\overline{\tX}_\Pi$ yields the claim for $\overline{\tX}_\Pi$. 

We now prove the result for $\overline{\tX}_\Pi = N_{\Sigma}$. By the max version of the Stone--Weierstrass theorem, it is enough to show that $\PA_b(N_{\Sigma})$ separates points. Let $x$ and $y$ be two distinct points of $N_\Sigma$. Assume that $x, y$ lie in the closure $\overline{\sigma}$ of the same maximal cone $\sigma$. Since the fan $\Sigma$ is simplicial, we may assume that $\sigma = \R_{\ge 0}^n$ and view $x, y$ as elements of $[0, + \infty]^n$. Since $x \neq y$, up to renumbering the coordinates and switching the roles, we may assume that $x_1 > y_1$. 

Assume first that $y_1 >0$, and let $\psi : [0, + \infty] \to \R$ be a piecewise affine function with compact support contained in $(0, +\infty)$ and $\psi(x_1) \neq \psi(y_1)$. For $t = (t_1,\ldots,t_n) \in \sigma$, we define $\phi(t) = \psi(t_1)$; since the domains of linearity of $\phi$ are of the form $[z, w] + \R_{\ge 0}^{n-1}$, with $z, w$ on the ray $\rho_1 \coloneqq \R_{\ge 0} e_1$, this extends continuously to $\overline{\sigma}$ through the same formula also when $t_i \in [0, \infty]$. Moreover, it is clear that $\phi(x) \neq \phi(y)$.
\\Let $\sigma'$ be a maximal cone of $\Sigma$. If $\sigma'$ doesn't contain $\rho_1$, we extend $\phi$ to $\sigma'$ by zero. Otherwise, $\sigma' = \R_{\ge 0} e_1 +\ldots+ \R_{\ge 0} e_n(\sigma')$ is the sum of its rays, where $e(\sigma') \coloneqq(e_1, e_2(\sigma'),\ldots, e_n(\sigma'))$ is a basis of $N$, and any element $x \in \sigma'$ has coordinates $t'= (t'_1,\ldots,t'_n)$ in this basis. We now set $\phi(x) = \psi(t'_1)$. This extends continuously to the closure of $\sigma'$, by the same argument as above. 
\\It remains to check that $\phi$ is well-defined on $N_\Sigma$, i.e. that for any pair of maximal cones $\sigma'$ and $\sigma''$, the extensions of function $\phi$ agree on the intersection $\tau \coloneqq \sigma' \cap \sigma''$. By simpliciality $\tau = \R_{\ge 0} v_1+\ldots+\R_{\ge 0} v_k$, where $\{ v_i \}$ is a subset of both $e(\sigma')$ and $e(\sigma'')$. If $\tau$ doesn't contain $e_1$ then $\phi \equiv 0$ on $\tau$ in any case, so we now assume $v_1= e_1$. Let $x \in \tau$, then we may write $x$ in coordinates as $x =t_1 e_1+\ldots+t_k v_k$. Since these $t_i$'s are also the coordinates of $x$ in the bases $e(\sigma')$, $e(\sigma'')$ respectively, we have $t'_1= t''_1$ and $\phi$ is well-defined. 

If $y_1 =0$, we reduce to the previous case as follows: let $a \in N_{\R}$ such that $a+y$ has non-zero first coordinate. The translation map $T_a : z \mapsto z+a$ clearly extends to a homeomorphism of $N_{\Sigma}$, and moreover if $\phi \in \PA_b(N_{\Sigma})$ then $(\phi \circ T_a) \in \PA_b(N_{\Sigma})$ as well. Using the case $y_1 \neq 0$, we get a function $\phi \in \PA_b(N_{\Sigma})$ with $\phi(a+y) \neq \phi(a+x)$, and then the function $\phi \circ T_a$ separates $x$ and $y$.

The case where $x$ and $y$ lie in the closures of two different cones is similar and left to the reader.
\end{proof}

\subsection{Cycles}

Let $\Pi$ be a polyhedral structure on a polyhedral space $\tX$.

\begin{defi}[Normal vectors] \label{def:normal vectors}
Let $\tau \prec \sigma$ be faces of $\Pi$ with $\dim(\tau)=\dim(\sigma)-1$. This implies that there exists $f \in M$ and $b \in \Q$ such that $f + b$ is zero on $\tau$, non-negative on $\sigma$ and not-identically zero on $\sigma$. 
A normal vector of $\sigma$ relative to $\tau$ is an element $$n_{\sigma/\tau} \in N_\sigma \subseteq N$$ such that $n_{\sigma/\tau}=v_{\sigma/\tau} - v$ is a primitive generator of $N_\sigma/N_\tau$ and $f(v_{\sigma/\tau}) > 0$, when we pick $v \in \tau$ and write $V_{\tau, \R}= v+ N_{\tau, \R}$, $V_{\sigma, \R}= v+ N_{\sigma, \R}$. 
\end{defi}

\begin{defi}[Weights]
A $k$-weight on $\Pi$ is a map 
$$
\omega \colon \Pi^{(k)} \to \Q
$$ 
and the number $\omega(\sigma)$ is called the weight of $\sigma \in \Pi^{(k)}$.
We say that 
\begin{itemize}
    \item $\omega$ is non-negative if $\omega(\sigma) \geq 0$ for all $\sigma \in \Pi^{(k)}$; write $\omega \geq 0$,
    \item $\omega$ is strictly positive if $\omega(\sigma) >0$ for all $\sigma \in \Pi^{(k)}$; write $\omega >0$. 
\end{itemize}
\end{defi}

\begin{defi}[Cycles]
A $k$-cycle on $\Pi$ is a $k$-weight $c \colon \Pi^{(k)} \to \Q$ such that for any $\tau \in \Pi^{(k-1)}$ the following \emph{balancing condition} holds:
\begin{equation} \label{eq:balancing}
    \sum_{\substack{\sigma \in \Pi^{(k)}\\ \tau \prec \sigma}} c(\sigma) n_{\sigma/\tau} = 0 \in N/N_\tau,
\end{equation}
equivalently this sum is an element in $N_\tau$.
\end{defi}

Denote by 
\begin{itemize}
    \item[-] $W_k(\Pi)$ the abelian group of $k$-weights on $\Pi$,
    \item[-] $Z_k(\Pi)$ the subgroup of $k$-cycles,
    \item[-] $W_k(\Pi)_{\geq 0}$,  $Z_k(\Pi)_{\geq 0}$,  $W_k(\Pi)_{> 0}$ and  $Z_k(\Pi)_{> 0}$ the subsets of non-negative and of strictly positive weights and cycles on $\Pi$.
\end{itemize}

\begin{defi}[Local cones]
\label{def:local cone}
Let $\tau \prec \sigma$ be faces of $\Pi$, the normal cone of $\sigma$ relative to $\tau$ is the polyhedral cone
$$C_{\sigma/\tau} \coloneqq \left\{ \lambda (x-y) \; | \;\lambda \ge 0, x \in \sigma, y \in \tau \right\} \subset N_{\sigma, \R};$$
it satisfies $N_{C_{\sigma/\tau}} = N_{\sigma}$ and hence $\dim(C_{\sigma/\tau}) = \dim (\sigma)$.

Let $x \in \tX$ and $\tau \in \Pi$ the unique face such that $x \in \relint(\tau)$. We define the \emph{star of $\Pi$ at $x$} as
$$\Star_x(\Pi) \coloneqq \{ C_{\sigma/\tau}\}_{\tau \prec \sigma }.$$ 
For a face $\delta$ of $\Pi$ we write $\Star_{\delta}(\Pi) \coloneqq \Star_x(\Pi)$ for any $x \in \relint(\delta)$. 
\end{defi}
The set $\Star_x(\Pi)$ has the structure of a polyhedral complex whose faces are (not necessarily strictly convex) cones; such structure is also known as \emph{generalized fan}; see \cite[Definition 6.2.2]{CoxLittleSchenck2011}.

\begin{defi}
Let $\tau \in \Pi$ and $x \in \relint(\tau)$. For any $k$-weight $\omega$ on $\Pi$, with $k \geq \dim(\tau)$, the $k$-weight $\omega_x$ on $\Star_x(\Pi)$ is
$$\omega_x(C_{\sigma/\tau}) \coloneqq \omega(\sigma).$$
\end{defi}

\begin{lemma} \label{lem: cycle iff local cycle}
The weight $\omega$ is a cycle if and only if $\omega_x$ is a cycle for all $x \in \tX$.
\end{lemma}
\begin{proof}
For any $\ell \geq \dim(\tau)$, $\ell$-faces of $\Pi$ containing $\tau$ are in bijection with the $\ell$-faces of $\Star_x(\Pi)$ containing $C_{\tau/\tau} = N_{\tau, \R}$, and this bijection preserves $\ell$-weights by definition. It remains to see that it preserves normal vectors, i.e. if $\tau \prec \sigma' \prec \sigma$ with $\dim(\sigma') = \dim(\sigma) -1$ and $n_{\sigma/\sigma'} \in N_{\sigma}$ is a normal vector, then $n_{\sigma/\sigma'}$ is also a normal vector of $C_{\sigma/\tau}$ relative to $C_{\sigma'/\tau}$. This follows directly from the fact that $N_{C_{\sigma/\tau}} = N_{\sigma}$ and similarly for $\sigma'$. 
\end{proof}

\begin{defi} [Pullback along refinements]
    Let $\Pi'$ be a refinement of $\Pi$ and $\omega \in W_k(\Pi)$. The pullback of $\omega$ to $\Pi'$, still denoted by $\omega$, is the $k$-weight on $\Pi'$ defined 
    by
    $$
    \omega(\sigma') \coloneqq \begin{cases}
        \omega(\sigma) & \text{if there exists $\sigma \in \Pi^{(k)}$ such that $\sigma' \subseteq \sigma$} \\
        0 & \text{otherwise}.
    \end{cases}
    $$ 
\end{defi}
Moreover, if $c$ is a $k$-cycle on $\Pi$, then the pullback of $c$ to $\Pi'$ is a $k$-cycle on $\Pi'$ (see for instance \cite[Example 2.11(iv)]{GathmannKerberMarkwig2009}).

\begin{defi}
    The space of $k$-weights (respectively $k$-cycles) on $\tX$ is the direct limit
    $$
    W_k(\tX)= \varinjlim_{\Pi} W_k(\Pi) \quad (\text{respectively} \quad Z_k(\tX)= \varinjlim_{\Pi} Z_k(\Pi) )
    $$
    over the set of polyhedral structures on $\tX$ ordered by refinements, with pullback maps $W_k(\Pi) \to W_{k}(\Pi')$ (respectively for $Z_k$) for any refinement $\Pi' \geq \Pi$. 
    
    Similarly, 
    $$
    W_k(\tX)_{\geq 0}= \varinjlim_{\Pi} W_k(\Pi)_{\geq 0} \quad (\text{respectively} \quad Z_k(\tX)_{\geq 0}= \varinjlim_{\Pi} Z_k(\Pi)_{\geq 0} )
    $$
    are the spaces of non-negative $k$-weights (respectively $k$-cycles) on $\tX$.
\end{defi}
Note that, if $\tX$ is of dimension $d$, then $W_d(\tX)_{>0}$ and $Z_d(\tX)_{>0}$ are well-defined too, while for $k<d$ the notion of strictly positive $k$-weights or cycles is not stable under pullback along refinements. 

\begin{defi}[Balanced polyhedral space]\label{def:balanced-poly} Let $\tX \subseteq N_{\R}$ be a polyhedral space of pure dimension $d$. 
\begin{itemize}
    \item A \emph{balancing condition} on $\tX$ is a strictly positive $d$-cycle $[\tX]$ on $\tX$, i.e., for some (or equivalently, any) polyhedral complex $\Pi$ on which $[\tX]$ is defined we have $[\tX](\sigma) > 0$ for any $\sigma \in \Pi^{(d)}$. 
    \item A \emph{balanced polyhedral space} is a pair $(\tX, [\tX])$, where $\tX$ is a polyhedral complex of pure dimension and $[\tX]$ is a balancing condition on $\tX$ (which will be frequently omitted from notation).
\end{itemize}
\end{defi}

\subsection{Intersection product} \label{sec:intersection}

Let $\tX$ be a polyhedral space. 

\begin{defi} [Piecewise affine functions on $\Pi$]
Let $\Pi$ be a polyhedral structure on $\tX$. 
A continuous function $\phi \colon \tX \to \R$ is \emph{piecewise affine on $\Pi$} if for any $\sigma \in \Pi$ we have 
$$\phi_{|V_{\sigma, \R}}= \phi_{\sigma} + b_\sigma$$ where $\phi_{\sigma}\in M_\sigma$ is an integral linear function, and $b_\sigma \in \Q$. In particular, if $x \in \sigma$ and if we write $V_{\sigma, \R}= x + N_{\sigma, \R}$, then $b_\sigma= \phi_{\sigma}(0)$ and $\phi(z)=\phi_{\sigma}(z-x)+ b_{\sigma}$.

We denote by $\PA(\Pi)$ the set of piecewise affine functions on $\Pi.$
\end{defi}

\begin{defi} [Piecewise affine functions on $\tX$]
A continuous function $\phi \colon \tX \to \R$ is \emph{piecewise affine}
if it is piecewise affine on some polyhedral structure on $\tX$. The set of piecewise affine functions on $\tX$ is denote by $\PA(\tX)$. In other words, we have $$\PA(\tX) = \varinjlim_{\Pi}\PA(\Pi),$$
where the direct limit is over the set of all polyhedral structures on $\tX$ ordered by refinements, and with respect to the inclusion $\PA(\Pi) \subseteq \PA(\Pi')$ for any $\Pi' \geq \Pi$. 
\end{defi}

\begin{defi}
    Let $\phi \in \PA(\Pi)$, $\tau \in \Pi$, and $x \in \relint(\tau)$.
    We denote by $\phi_x$ the (conical) piecewise affine function on $\Star_x(\Pi)$ such that for any $\sigma \succ \tau$
    $$(\phi_x)_{| C_{\sigma/ \tau}} = \phi_{\sigma} -\phi_{\tau},$$
    where we write $V_{C_{\sigma/ \tau}, \R}= x+ N_{C_{\sigma/\tau}, \R}$, $V_{\sigma, \R}= x + N_{\sigma, \R}$, and $V_{\tau, \R}= x + N_{\tau, \R}$.
    
    We have $\phi_x(\lambda(z-y)) = \lambda(\phi_{\sigma}(z-x)-\phi_{\tau}(y-x))$ whenever $z \in \sigma$ and $y \in \tau$, and in particular $\phi_x=0$ on $C_{\tau/\tau}$.
\end{defi}

\begin{defi} \label{def:intproduct}
Let $\phi \in \PA(\Pi)$ and $c \in Z_k(\Pi)$. The intersection product $\phi \cdot c$ is the $(k-1)$-cycle on $\Pi$ defined by 
\begin{equation} \label{eq:intersection product}
(\phi \cdot c) (\tau) \coloneqq 
\sum_{ \sigma \in \Pi^{(k)}\,:\,\tau \prec \sigma} c (\sigma) \phi_{\sigma}(n_{\sigma/\tau}) 
- \phi_{\tau} \left( \sum_{ \sigma \in \Pi^{(k)}\,:\,\tau \prec \sigma} c (\sigma) n_{\sigma/\tau} \right)    
\end{equation}
where $\tau \in \Pi^{(k-1)}$ and $n_{\sigma/\tau}$ are normal vectors relative to $\tau$.
\end{defi}
\begin{rem}
Note that \eqref{eq:intersection product} is well-defined. Indeed, let $n'_{\sigma/\tau}$ be another choice of normal vectors. We have
{\allowdisplaybreaks
\begin{align*}
    \sum_{ \sigma \in \Pi^{(k)}\,:\,\tau \prec \sigma}& c (\sigma) \phi_{\sigma}(n_{\sigma/\tau}) 
    - \phi_{\tau} \left( \sum_{ \sigma \in \Pi^{(k)}\,:\,\tau \prec \sigma} c (\sigma) n_{\sigma/\tau} \right) 
    \\
    & - \sum_{ \sigma \in \Pi^{(k)}\,:\,\tau \prec \sigma} c (\sigma) \phi_{\sigma}(n'_{\sigma/\tau}) 
    + \phi_{\tau} \left( \sum_{ \sigma \in \Pi^{(k)}\,:\,\tau \prec \sigma} c (\sigma) n'_{\sigma/\tau} \right)  \\
    & = \sum_{ \sigma \in \Pi^{(k)}\,:\,\tau \prec \sigma} c (\sigma) \phi_{\sigma}( \underbrace{n_{\sigma/\tau}- n'_{\sigma/\tau}}_{ \in N_\tau}) 
    - \phi_{\tau} \left( \sum_{ \sigma \in \Pi^{(k)}\,:\,\tau \prec \sigma} c (\sigma) n_{\sigma/\tau} - c(\sigma) n'_{\sigma/\tau}\right)  \\
    & = \sum_{ \sigma \in \Pi^{(k)}\,:\,\tau \prec \sigma} c (\sigma) \phi_{\tau}(n_{\sigma/\tau}- n'_{\sigma/\tau}) 
    - \phi_{\tau} \left( \sum_{ \sigma \in \Pi^{(k)}\,:\,\tau \prec \sigma} c (\sigma) n_{\sigma/\tau} - c (\sigma) n'_{\sigma/\tau}\right) =0. 
\end{align*}
}
\end{rem}

\begin{lemma} \label{lem:pairing and pullback}
Let $\phi \in \PA(\Pi)$ and $c \in Z_k(\Pi)$. Let $\Pi'$ be a refinement of $\Pi$, denote by $c'$ and $(\phi \cdot c)'$ the pullbacks to $\Pi'$ of $c$ and $\phi \cdot c$ respectively. The equality 
$$
(\phi \cdot c)'= \phi \cdot  c'
$$
holds in $Z_{k-1}(\Pi')$.
\end{lemma}
\begin{proof}
    Let $\tau' \in \Pi'^{(k-1)}$ and recall that
    $$
    \phi \cdot c' \coloneqq 
    \sum_{\sigma' \in \Pi'^{(k)}\,:\, \sigma' \succ \tau'} c'(\sigma') \phi_{\sigma'}(n_{\sigma'/\tau'}) 
    - \phi_{\tau'} \left( \sum_{\sigma' \in \Pi'^{(k)}\,;\, \sigma' \succ \tau'} c'(\sigma') n_{\sigma'/\tau'}\right).
    $$
Let $\tau$ be the minimal face of $\Pi$ such that $\tau' \subseteq \tau$. If $\dim(\tau) \geq k+1$, then for any $\sigma' \succ \tau'$, the faces $\sigma$ such that $\sigma' \subseteq \sigma$ have dimension at least $k+1$ as they contain $\tau$. This implies that $(\phi \cdot c')(\tau')=0=(\phi \cdot c)'(\tau').$

If $\dim(\tau)=k$, then there are precisely two faces $\sigma', \sigma'' \in \Pi'^{(k)}$ containing $\tau'$. We have that $\sigma' \cap \sigma''=\tau'$, $c'(\sigma')=c'(\sigma'')$, $N_{\sigma'}=N_{\sigma''}$, $n_{\sigma'/\tau'}=- n_{\sigma''/\tau'}$, and $\phi_{\sigma'}=\phi_{\sigma''}$. This implies
    $$
    (\phi \cdot c')(\tau')= c'(\sigma') 
    \left(
    \phi_{\sigma'}(n_{\sigma'/\tau'}) - \phi_{\sigma'}(n_{\sigma'/\tau'}) - \phi_{\tau'}(n_{\sigma'/\tau'}- n_{\sigma'/\tau'}
    \right) =0 = (\phi \cdot c)'(\tau').
    $$

If $\dim(\tau)=k-1$, then we denote by $\sigma$ the minimal face of $\Pi$ such that $\sigma' \subseteq \sigma$, for any $\sigma' \in \Pi'^{(k)}$ with $\sigma' \succ \tau'$. As $c'(\sigma')=0$ unless $\dim(\sigma')=\dim(\sigma)=k$, we have
    \begin{align*}
    (\phi \cdot c')(\tau')
    & = \sum_{ \substack{\sigma' \succ \tau' \\ \dim(\sigma')=\dim(\sigma)=k}}
    c'(\sigma') \phi_{\sigma'}(n_{\sigma'/\tau'}) 
    - \phi_{\tau'} \left( \sum_{\substack{\sigma' \succ \tau' \\ \dim(\sigma')=\dim(\sigma)=k}} c'(\sigma') n_{\sigma'/\tau'}\right) \\
    & = \sum_{\sigma \in \Pi^{(k)}\,:\, \sigma \succ \tau} c(\sigma) \phi_{\sigma}(n_{\sigma/\tau}) 
    - \phi_{\tau} \left( \sum_{\sigma \in \Pi^{(k)}\,;\, \sigma \succ \tau} c(\sigma) n_{\sigma/\tau}\right) \\
    & = (\phi \cdot c)(\tau) 
    = (\phi \cdot c)'(\tau'),
    \end{align*}
    which concludes the proof.
\end{proof}

\begin{lemma} \label{lem:properties pairing}
Let $c \in Z_k(\Pi)$, $\phi, \psi \in \PA(\Pi)$, and $\lambda, \mu \in \Q$. We have
\begin{itemize}
    \item[(i)] (linearity) $(\lambda\phi + \mu\psi) \cdot c = \lambda (\phi \cdot c) + \mu(\psi \cdot c)$,
    \item[(ii)] (balancing condition) $\phi \cdot c \in Z_{k-1}(\Pi)$,
    \item[(iii)] (commutativity) $\psi \cdot (\phi \cdot c) = \phi \cdot (\psi \cdot c)$. 
\end{itemize}
\end{lemma}
\begin{proof}
The linearity of the intersection pairing follows from the definition. For the other two properties, we reduce to the local fans. Indeed, given any $x \in \tX$, the equality
$(\phi \cdot c)_x = \phi_x \cdot c_x$
holds in $Z_{k-1}\left(\Star_x(\Pi)\right)$. The image of $\Star_x(\Pi)$ along the projection map $N_{\R} \to N_{\R}/(x+{N_{\tau,\R}})$ is a fan in the sense of \cite[Definition 2.2]{AllermannRau2010}. Then the statements follow from~\cite[Proposition 3.7]{AllermannRau2010} and Lemma \ref{lem: cycle iff local cycle}.
\end{proof}

It follows from Lemmas \ref{lem:pairing and pullback} and \ref{lem:properties pairing} that \eqref{eq:intersection product} induces intersection pairings 
\begin{align*}
  \PA(\tX) \times Z_k(\tX) & \to Z_{k-1}(\tX)
\end{align*}
where, for $\phi \in \PA(\Pi_1)$ and $c \in Z_k(\Pi_2)$, the pairing is given by $\phi \cdot c'$ for any common refinement $\Pi'$ of $\Pi_1$ and $\Pi_2$.
For any $\phi_1, \ldots, \phi_\ell \in \PA(\tX)$ and any cycle $c \in Z_k(\tX)$, we define inductively 
\begin{equation} \label{eq:notation phi^l}
\phi_1 \cdot \ldots \cdot \phi_\ell \cdot c \coloneqq \phi_1 \cdot \left(\ldots (\phi_\ell \cdot c)\right).   
\end{equation}
In particular, for $\phi=\phi_1= \ldots=\phi_\ell$, we simply write $\phi_1 \cdot \ldots \phi_\ell \cdot c= \phi^\ell \cdot c$.

\begin{rem}
Let $\Pi$ be a polyhedral structure on $\tX$ and $U$ an open subset of $\tX$. We can define the notions of weight and non-negative weight, cycle, pullback of cycles, piecewise affine function, and intersection product by replacing $\Pi$ with $\Pi|_U$ in the above definitions. We then denote
    $$
    W_k(U)= \varinjlim_{\Pi} W_k(\Pi|_U), \quad Z_k(U)= \varinjlim_{\Pi} Z_k(\Pi|_U), \quad
    \PA(U)= \varinjlim_{\Pi} \PA(\Pi|_U)
    $$
where again the limits are over the set of polyhedral structures on $\tX$ ordered by refinements.
\end{rem}

\begin{rem} [ (Relation to toric geometry)] \label{ex:fan}
Let $\tX = N_{\R}$, $\Pi = \Sigma$ be a complete simplicial fan with balancing condition $[\tX]$ constantly equal to $1$. Let $X_\Sigma$ be the corresponding complete toric variety. Let $\phi$ be a piecewise linear function on $\Sigma$, and $D$ the corresponding torus-invariant divisor. We claim that $$\phi^{k} \cdot [\tX] (\tau) = (D^{k}) \cdot V(\tau)$$ for $\tau \in \Sigma^{(d-k)}$, where $V(\tau)$ is the $k$-dimensional toric subvariety corresponding to $\tau$. 
     
We prove the claim by induction on $k$. The case $k=0$ is clear as $[\tX](\sigma)=1=X_\Sigma \cdot V(\sigma)$. Assume that $\phi^k \cdot [\tX] (\sigma) = (D^k) \cdot V(\sigma)$ whenever $\sigma$ has codimension $k$. Let $\tau \in \Sigma^{(d-k-1)}$, applying the formula from \cite[Section 5.1]{Fulton1993} to the $1$-cycle $(D^k) \cdot V(\tau)$, we get
\begin{equation} \label{eq:Fulton}
\sum_{v \in \Sigma^{(1)}} \big( D_v \cdot D^k \cdot V(\tau) \big) v =0,
\end{equation}
writing the $k$-codimensional cones containing $\tau$ as $\sigma_0,\ldots, \sigma_r$, with $\sigma_i = \tau + \R_{\ge 0} v_i$, we can take $n_{\sigma_i/\tau} =v_i$ so that we have
\begin{align*}
\phi^{k+1}\cdot [\tX] (\tau) 
& = \sum_{i=0}^r \big(\phi^k \cdot [\tX]\big)( \sigma_i) \phi_{\sigma}(v_i) -\phi_{\tau}\big(\sum_{i=0}^r (\phi^k \cdot [\tX])(\sigma_i) v_i\big) \\
& = \sum_{i=0}^r a_{v_i} (D^k \cdot V({\sigma_i})) - \phi_{\tau}\big(\sum_{i=0}^r (D^k \cdot V({\sigma_i})) v_i\big)
\end{align*}
by the induction hypothesis, where $D= \sum_{v \in \Sigma^{(1)}} a_v D_v$. 
Then \eqref{eq:Fulton} yields
$$\phi^{k+1}\cdot [\tX] (\tau)= \sum_{i=0}^r a_{v_i} (D^k \cdot V({\sigma_i})) + \sum_{v \in \tau \cap \Sigma^{(1)}} a_v(D_v \cdot D^k \cdot V(\tau))v,$$
and since $D^k \cdot V({\sigma_i}) = D^k \cdot D_i \cdot V(\tau)$, we get
$$\phi^{k+1}\cdot [\tX] (\tau) = \sum_{v \in \Sigma^{(1)}} a_v \big(D_v \cdot D^k \cdot V(\tau)\big) = D^{k+1} \cdot V(\tau),$$
which concludes the proof.

More generally, the Chow ring of $X_\Sigma$ is isomorphic to the ring of cycles on $\Sigma$, with product given by the so-called fan displacement rule. The intersection of powers of divisors with an irreducible toric subvariety yields precisely the statement above, see \cite[Theorem 3.1]{FS}.
\end{rem}

\section{Spaces of convex functions}
Let $\tX$ be a polyhedral space. In this section we introduce and study various notions of convex functions on $\tX$, building on \cite{BBS}.

\subsection{Polyhedral convexity}
\begin{defi} 
[PAPC functions] 
\label{def:PAPC}
A piecewise affine function $\phi$ on $\tX$ is  \emph{polyhedrally convex} (in short \emph{PAPC}) if
$$\phi \cdot c \geq 0$$
for any $c \in Z_k(U)_{\geq 0}$ an on open $U \subseteq \tX$ and any $k$.
The space of piecewise affine polyhedrally convex functions is denoted $\PAPC(\tX)$ and is endowed with the topology of pointwise convergence. 
\end{defi}
The property of being PAPC does not depend on the choice of a polyhedral structure on $\tX$, by Lemma \ref{lem:pairing and pullback}. If we want to specify the polyhedral complex $\Pi$, we write 
   \[
    \PAPC(\Pi) \coloneqq \PAPC(\tX) \cap \PA(\Pi).
  \]
  
\begin{defi}\label{def:poly-combi}
A \emph{convex combination in $\tX$} is a triple 
    $$
    ( x, (x_i)_{i \in I}, (\alpha_i)_{i \in I})
    $$
where $I$ is a finite set, $x_i \in \tX \cap N_\Q$, $\alpha_i \in [0,1]$ such that 
$$\sum_{i \in I} \alpha_i=1 \quad \text{ and }x=\sum_{i \in I}\alpha_i x_i.$$
A convex combination is \emph{polyhedral on $\Pi$} for a polyhedral structure $\Pi$ on $\tX$ if there are faces $\tau$ and $\sigma_i$ in $\Pi$ such that $x \in \tau$, $x_i \in \sigma_i$ and $\tau \preceq \sigma_i$; it is \emph{polyhedral on $\tX$} if it is polyhedral for some polyhedral structure $\Pi$ on $\tX$. We will abbreviate \emph{polyhedral convex combination} with $\PCC$.
\end{defi}

\begin{defi}[PC functions] \label{defn:PC}
    A function $\phi \colon \tX \to \R$ is \emph{polyhedrally convex} (in short \emph{PC}) if 
    $$
    \phi(x) \leq \sum_{i \in I} \alpha_i \phi(x_i)
    $$
    for any polyhedral convex combination $
    ( x, (x_i)_{i \in I}, (\alpha_i)_{i \in I})
    $ in $\tX$. The space of polyhedrally convex functions on $\tX$ is denoted $\PC(\tX)$ and is endowed with the topology of pointwise convergence. 
\end{defi}
Note that PC functions are automatically continuous on $\tX$ by \cite[Theorem~6.2]{BBS}; moreover, the restriction of any usual convex function in $N_\R$ to $\tX$ is PC.
The following lemma is analog to \cite[Proposition 5.7]{BBS}. We include a proof since in \emph{loc.\,cit.\,}the authors work with an eucledean structure, instead of an integral structure. 
\begin{lemma}
\label{lem:2 defn PAPC} 
\label{lem:pc}
The two definitions of PAPC functions are equivalent, i.e.,
    $$\PAPC(\tX) = \PA(\tX) \cap \PC(\tX).$$
\end{lemma}

\begin{proof}
We argue by double inclusion. Let $\phi \in \PAPC(\tX)$ and $
    ( x, (x_i)_{i \in I}, (\alpha_i)_{i \in I})
    $ a polyhedral convex combination in $\tX$. Let $\Pi$ be a polyhedral complex such that $\phi \in \PA(\Pi)$, and $\Pi'$ a refinement of $\Pi$ such that $x \in \Pi'^{(0)}$ and $x_i \in \rho_i\in \Pi'^{(1)}$ for some $\rho_i \succ x$. Let $n_{\rho/x}$ be normal vectors relative to $x$ for each $\rho \in \Pi'^{(1)}$ with $x \succ \rho$. Then there exist $\beta_i \in \Q_{\geq 0}$ such that $x_i - x = \beta_i n_{\rho_i/x}$. Consider the weight
    $$
    c(\rho) \coloneqq \begin{cases}
        \sum_{i \in I \,:\, x_i \in \rho} \alpha_i \beta_i & \text{ if there exists $i$ such that } x_i \in \rho\\
        0 & \text{ otherwise},
    \end{cases}
    $$ and an open $U \subseteq \tX$ such that $U \cap \Pi'^{(0)}=\{x\}$. Then $c$ is a non-negative cycle in $Z_1(\Pi'|_U)$ as 
    $$
    \sum_{\rho \in \Pi'^{(1)}|_U\,:\, x \prec \rho} c(\rho)n_{\rho/x}
    = \sum_{i \in I } \alpha_i \beta_i n_{\rho_i/x} 
    = \sum_{i \in I} \alpha_i(x_i - x)=0.
    $$
    By Definition \ref{def:PAPC} we have
    \begin{align*}
        0 \leq (\phi \cdot c)(x) 
        & = \sum_{\rho \in \Pi'^{(1)}|_U \,:\, x \succ \rho} c(\rho) \phi_\rho(n_{\rho/x})
        = \sum_{i \in I} \alpha_i \beta_i \phi_{\rho_i}(n_{\rho_i/x}) \\
        & = \sum_{i \in I} \alpha_i \phi_{\rho_i}(x_i-x) 
        = \sum_{i \in I} \alpha_i \phi(x_i)- \phi(x)
    \end{align*}  
    which implies that $\phi$ is polyhedrally convex in the sense of Definition \ref{defn:PC}.

    Let $\phi \in \PA(\tX) \cap \PC(\tX)$. Let $c \in Z_k(\Pi|_U)$ be a non-negative cycle on an open $U \subseteq \tX$. 
    Up to taking a refinement of $\Pi$ we can assume that $\phi \in \PA(\Pi)$. Let $\tau \in \Pi^{(k-1)}|_U$ and pick $x \in U \cap \relint(\tau) \cap N_\Q$. Let $n_{\sigma/\tau}$ be normal vectors relative to $\tau$ for each $\sigma \in \Pi^{(k)}|_U$ with $\tau \prec \sigma$. By \eqref{eq:balancing} we have $\sum_{\sigma \succ \tau} c(\sigma) n_{\sigma/\tau}= v_\tau \in N_\tau$. Up to rescaling $c$, we can assume that $\gcd(c(\sigma))_{\sigma \succ \tau} =1$ and write 
    $\gcd(c(\sigma))_{\sigma \succ \tau}= \sum_{\sigma \succ \tau} \alpha_\sigma c(\sigma)$ for some $\alpha_\sigma \in \Z$. We set 
    \begin{align*}
    \widetilde{n}_{\sigma/\tau} 
    & \coloneqq n_{\sigma/\tau} - \alpha_\sigma v_\tau \quad \text{ another normal vector of $\sigma$ relative to $\tau$,}\\
    x_\sigma
    & \coloneqq x+ \varepsilon \widetilde{n}_{\sigma/\tau} \in \sigma \cap U \quad \text{ for some $\varepsilon \in \Q_{>0}$ sufficiently small,}\\
    S 
    & \coloneqq \sum_{\sigma \prec \tau} c(\sigma) \geq 0.
    \end{align*}
    If $S=0$, then $(\phi \cdot c)(\tau)=0$, hence $\phi \cdot c$ has non-negative weight at $\tau$. If $S \neq 0$, then 
    $$
    \left(x, (x_\sigma)_{\sigma \succ \tau}, (\tfrac{c(\sigma)}{S})_{\sigma \succ \tau}\right)
    $$
    is a polyhedral convex combination as $\tfrac{c(\sigma)}{S} \in [0,1]$,  $\sum_{\sigma  \succ \tau} \tfrac{c(\sigma)}{S}=1$ and 
    $$
    \sum_{\sigma \succ \tau} \tfrac{c(\sigma)}{S}x_\sigma 
    = x + \tfrac{\varepsilon}{S} \sum_{\sigma \succ \tau} c(\sigma) \widetilde{n}_{\sigma/\tau}
    = x + \tfrac{\varepsilon}{S} 
    \left( 
    \sum_{\sigma \succ \tau} c(\sigma) n_{\sigma/\tau} - v_\tau \sum_{\sigma \succ \tau}\alpha_\sigma c(\sigma)
    \right) = x.
    $$
    By Definition \ref{defn:PC} we have
    \begin{align*}
        0 \leq 
        \tfrac{S}{\varepsilon} \left(
        \sum_{\sigma \succ \tau} \tfrac{c(\sigma)}{S} \phi(x_\sigma) - \phi(x)
        \right)
        = \sum_{\sigma \succ \tau} c(\sigma) \phi_\sigma( \widetilde{n}_{\sigma/\tau}) = (\phi \cdot c)(\tau)     
    \end{align*}
    which implies that $\phi \cdot c$ is a non-negative cycle in $Z_{k-1}(\Pi|_U)$, hence $\phi$ is polyhedrally convex in the sense of Definition \ref{def:PAPC}.
\end{proof}

We recall the definition of convex functions in \cite[\S 2.1]{AminiPiquerez2020}, and compare such notion with the polyhedral convexity defined above.
\begin{defi}[{\cite[§2.1]{AminiPiquerez2020}}] \label{def:convexAP}
A piecewise affine function $f$ on a polyhedral structure $\Pi$ on $\tX \subseteq N_\R$ is convex if, for any $\sigma \in \Pi$, there exists an affine function $l$ on $N_\R$ such that
\begin{equation} \label{eq:def convex AP}
\begin{cases}
    f=l & \text{ on } \sigma, \\
    f \geq l & \text{ on } U \setminus \sigma
\end{cases}
\end{equation}
for an open neighborhood $U$ of $\relint(\sigma)$ in $\tX$. 
\end{defi}
\begin{prop}\label{prop:convex-ap}
    Let $\Pi$ be a simplicial polyhedral structure on $\tX$. The space of functions in $\PA(\Pi)$ which are convex in the sense of Definition \ref{def:convexAP} coincides with the space $\PAPC(\Pi)$. 
\end{prop}

\begin{proof}
Write $\mathcal{A}$ for the set of convex functions in the sense of Definition \ref{def:convexAP}, defined on $\Pi$. For any $\delta \in \Pi$, set
$$\mathcal{A}_{\delta} \coloneqq \{ \phi \in \PA(\Pi) \; | \;\exists \ell \in \Aff(N_\R), (\phi+\ell)_{| \delta} =0, (\phi+\ell)\ge 0 \; \text{near} \; \delta \}$$
$$\PAPC_{\delta} \coloneqq \left\{ \phi \in \PA(\Pi)\;  | 
\substack{\textit{\normalsize
$\;\forall \; \PCC (x, (x_i)_i, (\alpha_i)_i),\, x \in \relint(\delta), x_i \in \sigma_i, \delta \preceq \sigma_i \in \Pi$} 
\\ \textit{\normalsize{$\phi(x) \le \sum_i\alpha_i\phi(x_i)$}}
}
\right\}.$$ 
By definition $\mathcal{A} = \bigcap_{\delta} \mathcal{A}_{\delta}$ and $\PAPC(\Pi) = \bigcap_{\delta} \PAPC_{\delta}$, hence it is enough to prove $\mathcal{A}_{\delta} = \PAPC_{\delta}$. Since the conditions to belong to $\mathcal{A}_{\delta}$, or $\PAPC_{\delta}$ respectively, only depend on the restriction of a function to $\Star_{\delta}(\Pi)$, we may work on the corresponding (generalized) fan; after removing an affine function, we reduce to the case of $\PA$ functions that are identically zero on $\delta$; quotienting by $N_{\delta, \R}$, we reduce to the case of $\Pi = \Sigma$ a simplicial fan and $\delta = \{0 \}$ the vertex.

To prove the inclusion $\mathcal{A}_{\{0 \}}\subseteq \PAPC_{\{0 \}}$, let $(0,(x_i)_i,(\alpha_i)_i)$ be a PCC in $\Sigma$ and let $\phi \in \mathcal{A}_{\{0 \}}$. Then there exists a global affine function $\ell$ such that $0=(\phi+\ell)(0) \leq \sum_i \alpha_i (\phi+\ell)(x_i)$, which implies that $\phi \in  \PAPC_{\{0 \}}$.

Let $g\colon M_{\R} \to \R^{\Sigma(1)}$ be the map $\ell \mapsto (\ell(v_i))_{i \in \Sigma(1)}$, where $v_i$ is the primitive generator of the ray. Note that the dual map $p \colon (\R^{\Sigma(1)})^{\vee} \to N_{\R}$ is the map sending the $i$-th basis vector to $v_i$. We have that $\mathcal{A}_{\{0\}} = (\R_{\ge 0})^{\Sigma(1)} + g(M_{\R})$
and writing $\sigma= (\R_{\ge 0})^{\Sigma(1)}$ we get that $\mathcal{A}_{\{0\}}^{\vee} = \sigma^{\vee} \cap \Ker(p)$. Thus, an element of $\mathcal{A}_{\{0\}}^{\vee}$ is a collection of non-negative numbers attached to the rays of $\Sigma$ which satisfy the balancing condition (because it lies in $\Ker(p)$), i.e. a non-negative $1$-cycle on $\Sigma$. It follows that an element of $\mathcal{A}_{\{0\}}^{\vee}$ pairs in a non-negative way with $\PAPC(\Sigma)$, whence $\mathcal{A}_{\{0\}}^{\vee} \subseteq \PAPC(\Sigma)^{\vee}$. By the bipolar theorem (see \cite{Rockafellar}) we infer
$$\PAPC(\Sigma) \subset \overline{\mathcal{A}_{\{0\}}},$$
where $\overline{\mathcal{A}_{\{0\}}}$ is the topological closure of $\mathcal{A}_{\{0\}}$ in $\R^{\Sigma(1)}$. Since $\mathcal{A}_{\{0\}}$ is closed (it is the closure of the open convex cone appearing in \cite[Remark 4.8]{AminiPiquerez2020}), this concludes the proof.
\end{proof}

\subsection{Strict convexity}

\begin{defi}[PAPC$^+$ functions] \label{def:strict conv}
Let $\Pi$ be a polyhedral structure on $\tX$. A piecewise affine function $\phi$ on $\Pi$ is \emph{strictly convex with respect to $\Pi$} (in short \emph{$\PAPC^+(\Pi)$}) if we have the strict inequality
$$\phi(x) < \sum_{i \in I} \alpha_i \phi(x_i)$$
for any PCC $(x, (x_i)_{i \in I}, (\alpha_i)_{i \in I})$ on $\Pi$ with $x \in \relint(\tau)$, $x_i \in \relint(\sigma_i)$, and $\tau \prec \sigma_i$ a strict subface for at least two distinct indices $i$.

The space of strictly convex $\PAPC$ functions on $\Pi$ is denoted $\PAPC^+(\Pi)$ and is endowed with the topology of pointwise convergence. 
\end{defi}

\begin{rem} \label{rem:PAPC+}
Equivalently, as it can be seen from the proof of Lemma \ref{lem:2 defn PAPC}, a function $\phi$ is in $\PAPC^+(\Pi)$ if and only if, for any open subset $U \subset \tX$ and any strictly positive cycle $c$ defined on $\Pi_{|U}$, the intersection $(\phi \cdot c)$ is strictly positive.    
\end{rem}

We recall the following definition from \cite[Section 2.1]{AminiPiquerez2020}.
\begin{defi} A polyhedral complex is called \emph{quasi-projective} if it admits a $\PAPC^+$ function.
\end{defi}
Not every polyhedral complex $\Pi$ is quasi-projective, not even a complete fan (see Example in \cite[Section 3.4]{Fulton1993}). On the other hand, given a polyhedral space $\tX$, one can always find a quasi-projective polyhedral structure $\Pi$ on $\tX$. More precisely,
\begin{prop}\label{prop:quasi-proj}
    For any polyhedral structure $\Pi$ on $\tX$, there exists a simplicial quasi-projective refinement $\Pi' \geq \Pi$.
\end{prop}
\begin{proof}
This follows from \cite[Theorem 4.1]{AminiPiquerez2020}.
\end{proof}

In \cite[\S 2.1]{AminiPiquerez2020}, a piecewise affine convex function on $\Pi$ is said strictly convex if the inequalities in Definition \ref{def:convexAP} are strict.

\begin{cor} \label{cor:strictconvex-ap}
Let $\Pi$ be a simplicial polyhedral complex. The set of functions in $\PA(\Pi)$ which are strictly convex in the sense of \cite{AminiPiquerez2020} coincides with $\PAPC^+(\Pi)$.
\end{cor}
\begin{proof}
We denote by $\ca A$ the set of convex functions in the sense of Definition \ref{def:convexAP}. The set $\PAPC^+(\Pi)$, and respectively the set of strictly convex functions in the sense of \cite{AminiPiquerez2020}, are the interiors of the closed convex cones $\PAPC(\Pi)$, respectively $\mathcal{A}$, inside $\PA(\Pi)$. Then the claim follows from Proposition \ref{prop:convex-ap}.
\end{proof}

\begin{lemma} \label{lem:PA plus PAPC+}
Let $\Pi$ be a quasi-projective polyhedral structure on $\tX$, and let $\phi \in \PAPC^+(\Pi)$. For any $f \in \PA(\Pi)$, there exists $m>0$ such that $m\,\phi + f \in \PAPC(\tX)$.
\end{lemma}
\begin{proof}
Let $\delta$ be a face of $\Pi$. By Corollary \ref{cor:strictconvex-ap}, there exists an affine function $l \in \Aff(N_\R)$ such that $\phi=l$ on $\delta$ and $\phi > l$ on $U \setminus \delta$, for an open neighborhood $U$ of $\relint(\delta)$ in $\tX$. This implies that there exists $c>0$ such that, for any $x$ in $\tX$, we have
$$
\phi(x)-l(x) \geq c\, \dist(x, \delta).
$$
Let $l'$ be an affine function on $N_\R$ such that $f_{|\delta}=l'$. Then $f=l'$ on $\delta$ and there exists $c'>0$ such that 
$$
f(x)-l'(x) \geq -c'\, \dist(x, \delta)
$$
for any $x \in U \setminus \delta$.
It follows that $m\,\phi+f-(m\,l +l') \geq (mc-c')\dist(x,\delta)$, thus for $m>\tfrac{c'}{c}$ the function $m\,\phi+f$ is polyhedrally convex by Proposition \ref{prop:convex-ap}.
\end{proof}

\subsection{Convex functions with growth condition}
Let $\gamma$ be a piecewise affine function on $\tX$.
\begin{defi}\label{def:poly-growth}
The space of \emph{$\gamma$-bounded PA/PC/PAPC} functions are defined as
\begin{align*}
&\PA(\tX, \gamma) \coloneqq \{ \phi \in \PA(\tX) | \,\phi = \gamma + O(1)\}\\
&\PC(\tX, \gamma) \coloneqq \{ \phi \in \PC(\tX) | \,\phi = \gamma + O(1)\}\\
&\PAPC(\tX, \gamma) \coloneqq \PA(\tX, \gamma) \cap \PC(\tX, \gamma)
\end{align*}
\end{defi}
The non-emptiness of the space $\PAPC(\tX,\gamma)$ does not imply that the function $\gamma$ itself is polyhedrally convex. Nevertheless, whenever this space is non-empty, there exists a polyhedrally convex function $\gamma' \in \PAPC(\tX)$ such that $\PAPC(\tX,\gamma)=\PAPC(\tX,\gamma')$.

\begin{lemma} \label{lem:decr}
Let $\gamma \in \PA(\Pi)$, $\phi \in \PC(\tX, \gamma)$, and write $\psi = \phi - \gamma$. The inequality
$$\psi \le \psi \circ r_{\Pi'}$$
holds for any simplicial refinement $\Pi'\geq \Pi$. Here, the map $r_{\Pi'} \colon \tX \to \on{Sk}(\Pi')$ denotes the retraction map from Definition \ref{def:retraction}.
\end{lemma}
\begin{proof} 
Let $\Pi'$ be a simplicial refinement $\Pi' \geq \Pi$. Let $\sigma \in \Pi'$ be an unbounded face of dimension $d$, which we write as
$$\sigma =\conv(p_0,\ldots,p_l) + \Cone(v_{l+1},\ldots,v_d)$$
since $\Pi'$ is simplicial. Since $\gamma$ is affine on $\sigma$, the function $\psi$ is convex on $\sigma$ in the usual sense, and bounded from above since $\phi \in \PC(\tX,\gamma)$. As a result, if $y \in \conv(p_0,\ldots,p_l)$, the restriction of $\psi$ to $y + \Cone(v_{l+1},\ldots,v_d)$ is a convex function on a polyhedral cone which is bounded from above, hence attains its maximum at the vertex $y$. This yields precisely $\psi(x) \le \psi(r_{\Pi'}(x))$ for any $x \in \sigma$ such that $y= r_{\Pi'}(x)$, and varying $y$ concludes the proof.
\end{proof}

\begin{lemma} 
\label{lem:ext} 
Let $\Pi$ be a simplicial polyhedral structure on $\tX$ such that $\gamma \in \PA(\Pi)$, and let $\overline{\tX}_\Pi$ be the compactification of $\tX$ with respect to $\Pi$.
Let $\phi \in \PC(\tX)$ and write $\psi \coloneqq \phi - \gamma$. Then $\phi \in \PC(\tX, \gamma)$ if and only if $\psi$ extends continuously to $\overline{\tX}_\Pi$. 

Moreover, if $\tau \in \Pi$ is an unbounded face, $\psi$ extends to the closure of $\tau$ in $\overline{\tX}_\Pi$ by 
$$\psi(t) = \inf \{ \psi(t') |\, t' \in \tau, t' \le t \}$$ 
where $\le$ is the coordinate-wise partial order on unbounded coordinates of $\tau$.
\end{lemma}

\begin{proof} 
The implication $\Leftarrow$ is clear as the space $\overline{\tX}_\Pi$ is compact, so we prove the reverse implication.

Let $\tau \in \Pi$ be an unbounded face of $\tX$, then $\psi \colon \tau \to \R$ is a bounded convex function. Since by Proposition \ref{prop:disjoint} the closures of disjoint unbounded faces of $\Pi$ are still disjoint in $\overline{\tX}_\Pi$, it is enough to prove that $\psi$ extends continuously to the closure of $\tau$ in $\overline{\tX}_\Pi$.
By simpliciality, we have $\tau =\delta + \sigma$ for a compact polyhedron $\delta \subset N_{\R}$, and 
$$\sigma \coloneqq \rec(\tau)=\{ \sum_{i =1}^\ell t_i v_i | t_i \ge 0 \}$$ for $\ell$ linearly independent vectors $v_i \in N_{\R}$. The Minskowski sum $\delta + \sigma$ is homeomorphic to the product of the two polyhedra, thus writing $\overline{\sigma} = \sigma \otimes_{\R_{\ge0}} [0, + \infty]$, the closure $\overline{\tau}$ of $\tau$ in $\overline{\tX}_{\Pi}$ is homeomorphic to $\delta \times \overline{\sigma}$, and $\overline{\sigma} \simeq [0,+\infty]^\ell$ via the variables $(t_i)$. We define a partial order on $\overline{\tau}$ by saying that $t' \le t$ if $t_i' \le t_i$ for all $i =1,\ldots,\ell$.

Since a bounded convex function on $\R_{\ge0}$ is decreasing, $\psi$ is decreasing which respect to each variable $t_i$ on $\tau$. Then for $t \in \overline{\tau}$, we define the extension of $\psi$ using the following formula:
\begin{equation} \label{equ: extension}
    \psi(t) = \inf \{ \psi(t') | t' \in \tau, t' \le t \}.
\end{equation}
We prove by induction on $d \leq \ell$ that, for every $0 \leq k \leq d$ and every continuous convex bounded function on $\delta \times \R_{\ge 0}^d$, its extension to $\delta \times [0, + \infty]^d$ defined as in \eqref{equ: extension} is continuous on $ \delta \times [0, + \infty]^k \times \R_{\ge 0}^{d-k}$. The case $k=d=\ell$ applied to the function $\psi$ proves the continuity of its extension to $\overline{\tau}$.

If $d=0$, then $k=0$ as well and every continuous convex function on $\delta$ is clearly continuous on $\delta$.
Assume that the statement holds for $d-1$; we prove by induction on $k$ that it holds for $d$. If $k=0$, every continuous convex function on $\delta \times \R^d_{\geq 0}$ is clearly continuous on $\delta \times \R^d_{\geq 0}$. Assume by induction hypothesis that the statement holds for $k < d$, we are going to prove it for $k+1$. Let $\varphi$ be a continuous convex bounded function on $\delta \times \R_{\geq0}^d$, and denote by $\varphi$ also its extension to $\delta \times [0,+\infty]^d$ as in \eqref{equ: extension}. Define the function 
\begin{align*}
    \varphi' \colon \delta \times \R_{\ge 0}^k \times \R_{\ge 0}^{d-k-1} & \to \R
    \\(x, (t_1,\ldots, t_k),  (t_{k+2},\ldots,t_d)) & \mapsto \varphi(x,t_1,\ldots,+ \infty,t_{k+2},\ldots t_d). 
\end{align*}
This is clearly a continuous convex bounded function on $\R_{\ge 0}^{d-1}$ (as pointwise limit of such functions), and its extension (still denoted by $\varphi'$) to $\delta \times [0, + \infty]^{d-1}$ via the formula \eqref{equ: extension} is then continuous by induction hypothesis. Moreover the extension of $\varphi'$ coincides with the extension of $\varphi$ restricted to $ \{ t_{k+1}= \infty \}$. For $y \in \R_{\geq 0}$, define 
$$\varphi_y (x, t_1,\ldots, t_k, t_{k+2},\ldots, t_d) = \varphi(x, t_1,\ldots,y,\ldots,t_d)$$ 
for $x \in \delta$, $(t_1,\ldots,t_k) \in [0, +\infty]^k$ and $(t_{k+2},\ldots,t_d) \in \R_{\ge 0}^{d-k-1}$. The functions $\varphi_y$ are continuous by induction hypothesis, and moreover decrease pointwise to $\varphi'$ as $y \rightarrow \infty$. Since $\varphi'$ is also continuous on $\delta \times[0, + \infty]^{d-1} $, Dini's lemma shows that convergence is uniform, so that the $\varphi_y$ and $\varphi'$ glue together to a continuous function on $ \delta \times [0, + \infty]^{k+1} \times \R_{\ge 0}^{d-k-1}$, which is none other than $\varphi$. This concludes the induction step and the proof of the lemma.
\end{proof}

\begin{lemma}
\label{lem:decreasing cv uniform cv}
    If a sequence $(\phi_j)_j$ of $\PC(\tX,\gamma)$ functions decreases pointwise to a function $\phi$ in $\PC(\tX,\gamma)$, then it converges uniformly.
\end{lemma}
\begin{proof}
Let $\Pi$ be a simplicial polyhedral structure on $\tX$ such that $\gamma \in \PA(\Pi)$. By Lemma \ref{lem:ext}, $(\phi-\gamma)$ and the $(\phi_j-\gamma)$ admit a continuous extension to $\overline{\tX}_{\Pi}$, which we denote by $\psi$ and $\psi_j$ respectively. By compactness of $\overline{\tX}_{\Pi}$ and Dini's lemma, it is enough to show that the $\psi_j$ decrease to $\psi$ pointwise on $\overline{\tX}_{\Pi}$. Let $\overline{x} \in \overline{\tX}_\Pi$, and $(x_m)_m$ a sequence in $\tX$ converging to $\overline{x}$, contained in a single cone and whose coordinates are all increasing. We have
$$
\psi(\overline{x})= \inf_m \psi(x_m)= \inf_m \inf_j \psi_j (x_m) =  \inf_j \inf_m \psi_j (x_m) = \inf_j \psi_j (\overline{x}) 
$$
by the explicit description of the extension provided in Lemma \ref{lem:ext}, and the proof is complete.
\end{proof}

\begin{lemma} 
\label{lem:sup}
    Let $(\phi_{\alpha})_{\alpha \in A}$ be a non-empty family in $\PC(\tX, \gamma)$, such that $(\phi_{\alpha} - \gamma)_{\alpha}$ is uniformly bounded from above. Then $\phi(x) \coloneqq \sup_{\alpha \in A} \phi_{\alpha}(x)$ defines an element $\phi \in \PC(\tX, \gamma)$.
\end{lemma}
\begin{proof}
    As in the proof of \cite[Theorem 7.11]{BoucksomFavreJonsson}, considering the net of finite subsets of $A$ ordered by inclusion, we may assume that $A$ is a directed set and $(\phi_{\alpha})_{\alpha \in A}$ is an increasing net, so that by \cite[Theorem 6.24]{BBS} the function $\phi$ is PC. Moreover $\phi - \gamma$ is bounded from above by assumption, and $\phi-\gamma \geq \phi_{\alpha}-\gamma \geq -C$ for any choice of $\alpha$, hence $\phi \in \PC(\tX, \gamma)$.
    \end{proof}

\subsection{PSH functions}
Let $\gamma$ be a piecewise affine function on $\tX$. 

\begin{defi}[PSH functions]\label{def:psh}
A function $\phi \colon \tX \to \R$ is $\gamma$-PSH functions if there exists a decreasing sequence $(\phi_j)_j$ in $\PAPC(\tX, \gamma)$ converging pointwise to $\phi$. 

The space of $\gamma$-PSH functions is denoted $\PSH(\tX,\gamma)$ and is endowed with the topology of pointwise convergence, which agrees with the topology of uniform convergence on compact sets by \cite[Theorem 6.24]{BBS}.
\end{defi}

\begin{prop}\label{prop:gamma-psh} Let $\phi \in \PSH(\tX, \gamma)$. The following properties hold:
\begin{enumerate}[label=(\roman*)]
    \item $\phi$ is continuous and polyhedrally convex,
    \item $\phi \leq \gamma + O(1)$,
    \item there exists a decreasing sequence $(\phi_j)_j$ in $\PAPC(\tX, \gamma)$ converging locally uniformly to $\phi$,
    \item for any compact $K \subset \tX$, $\phi|_K$ is the uniform limit on $K$ of (the restriction of) a sequence of functions in $\PAPC(\tX, \gamma)$. 
\end{enumerate}
\end{prop}

\begin{proof}
Let $\gamma \in \PA(\Pi)$, $\phi \in \PSH(\tX, \gamma)$, and  $(\phi_j)_j$ a decreasing sequence in $\PAPC(\tX, \gamma)$ converging pointwise to $\phi$. It follows immediately that $\phi$ is continuous and $\PC$. Moreover, for any simplicial refinement $\Pi' \geq \Pi$, we have
$$
\phi-\gamma \leq \phi_j - \gamma \leq C
$$
by definition of $\PAPC(\tX,\gamma)$, which implies $(ii)$.

Let $x \in \tX$ and $U$ a neighborhood of $x$ contained in a compact subset $K$ of $\tX$; respectively, let $K'$ be any compact subset of $\tX$. Applying Dini's lemma to the restriction of $(\phi_j)_j$ to $K$, respectively $K'$, we conclude $(iii)$ and $(iv)$. 
\end{proof}

\begin{prop}\label{def:PC-reg} 
Let $\phi$ be a continuous function on $\tX$. Then the following conditions are equivalent: 
\begin{enumerate}
        \item $\phi \in \PC(\tX, \gamma) \cap \PSH(\tX,\gamma)$;
        \item $\phi$ is the uniform limit on $\tX$ of a sequence of functions in $\PAPC(\tX, \gamma)$. One can moreover take the sequence to be decreasing. 
    \end{enumerate}
\end{prop}
\begin{defi}[PC-regularizable functions]
  The space of $\PC$-regularizable functions on $\tX$ is the set of continuous functions on $\tX$ satisfying one of the equivalent conditions (1) or (2) above. We denote it by $\PCreg(\tX, \gamma)$, and endow this space with the topology of uniform convergence.
 \end{defi}

\begin{proof}
    $(1) \implies (2)$. Assume $\phi \in \PC(\tX, \gamma) \cap \PSH(\tX,\gamma)$ and let $(\phi_j)_j$ be a decreasing sequence in $\PAPC(\tX, \gamma)$ converging pointwise to $\phi$. By Lemma \ref{lem:ext}, both $\psi \coloneqq \phi-\gamma$ and $\psi_j \coloneqq \phi_j - \gamma$ extend continuously to the closure of $\tX$ in $N_\Sigma$, for a fan $\Sigma$ containing $\rec(\Pi)$. We denote by $\overline{\psi}$, respectively $\overline{\psi}_j$ those continuous extensions. We claim that if $x$ lies on the boundary on $\tX$, then $(\overline{\psi}_j(x))_j$ decreases to $\overline{\psi}(x)$. Indeed, it follows from the proof of Lemma \ref{lem:ext} that there exists a sequence $(x_l)_l$ in $\tX$ converging to $x$, and such that $\overline{\psi}(x) = \inf_l \psi(x_l)$ and similarly for $\overline{\psi}_j$. Since $\psi(x_l) = \inf_j \psi_j(x_l)$, we infer $\overline{\psi}(x) = \inf_{(l, j)} \psi_j(x_l) = \inf_j \overline{\psi}_j(x)$ and the claim follows.
    \\As a result, by Dini's lemma $(\overline{\psi}_j)_j$ converges uniformly to $\overline{\psi}$ on the compactification of $\tX$, thus $(\phi_j)_j$ converges uniformly to $\phi$ on $\tX$.

    $(2) \implies (1)$. Let $(\phi_j)_j$ be a sequence in $\PAPC(\tX, \gamma)$ converging uniformly to $\phi$ on $\tX$. For an index $j$ sufficiently large, we have
    $$
    |\phi - \gamma| \leq |\phi-\phi_j| + |\phi_j - \gamma| \leq \varepsilon + C \leq C',
    $$
    which implies that $\phi \in \PC(\tX, \gamma)$.
\end{proof}

\subsection{Uniform bound on $C^{0,1}$-norm} \label{sec:equi} 
The goal of this section is to prove the following proposition, which will be crucial for establishing the compactness of $\PSH(\tX,\gamma)/\R$ in Theorem \ref{thm:comp}. Our argument is inspired by the strategy developed in \cite[\S 6]{BoucksomFavreJonsson}.

\begin{prop} \label{prop:equi PAPC}
Let $\Pi$ be a simplicial quasi-projective polyhedral structure on $\tX$ such that $\gamma \in \PA(\Pi)$.
There exists a constant $C>0$ such that for any $\phi \in \PAPC(\tX, \gamma)$ and any face $\sigma$ of $\Sk(\Pi)$, the $C^{0,1}$-norm of $\psi- \sup_{\tX}\psi$ on $\sigma$ is bounded by $C$, where $\psi \coloneqq \phi-\gamma$.
\end{prop}

We begin by recalling the definition of the $C^{0,1}$-norm, as well as the notion of directional derivative that will be used later. Fix a norm $\lVert \cdot \rVert$ on $N_{\mathbb{R}}$. Let $\varphi$ be a function on $\tX$ that is convex on a face $\sigma \in \Pi$. For any $v,w \in \sigma$, we denote
$$
D_v \varphi(w)= \frac{d}{dt} \biggr\rvert_{t = 0^+}  \varphi(v+t(w-v))
$$
the directional derivative of $\varphi$ at $v$ along $w-v$, which exists by convexity of $\varphi$. The Lipschitz constant of $\varphi$ on $\sigma$ is defined by
$$
\Lip_\sigma(\varphi) \coloneqq \sup_{v \neq v' \in \sigma} \frac{|\varphi(v)- \varphi(v')|}{\lVert v-v'\rVert} \in [0, +\infty]
$$
if $\dim(\sigma)>0$, otherwise by $\Lip_\sigma(\varphi)=0$. The $C^{0,1}$-norm of $\varphi$ on $\sigma$ is then 
$$
\lVert \varphi \rVert_{C^{0,1}(\sigma)} \coloneqq \sup_{\sigma} |\varphi| + \Lip_\sigma(\varphi).
$$

\begin{lemma} \label{lem:equ_positive}
There exists a strictly positive $k$-cycle on $\Pi$, for all $k \le d$.
\end{lemma}
\begin{proof}
By assumption $\Pi$ is quasi-projective, hence there exists $\alpha \in \PAPC^+(\Pi)$. Then the $k$-cycle $c =\alpha^{d-k} \cdot [\tX] $ is strictly positive.
\end{proof}

\begin{proof}[Proof of Proposition \ref{prop:equi PAPC}]
Let $\phi \in \PAPC(\tX, \gamma ) \cap \PA(\Pi')$ for some simplicial refinement $\Pi'$ of $\Pi$, and write $\psi=\phi-\gamma$. Let $\sigma \in \Sk(\Pi)$. We prove the proposition by induction on the dimension of $\sigma$. 

\subsubsection*{Base of induction} Assume $\dim(\sigma)=0$.

\begin{lemma} \label{lem: equi bound vertices}
There exists $C>0$ independent of $\psi$ such that
        $$\sup_{v \in \Pi^{(0)}} |\psi(v)-\sup_{\tX}\psi| \le C.$$
\end{lemma}

\begin{proof}
Fix a strictly positive $1$-cycle $c$ on $\Pi$, whose existence is provided by Lemma \ref{lem:equ_positive}. Since $\psi-\sup_{\tX}\psi$ is invariant by adding a constant to $\psi$, we assume that $\sup_{\tX}\psi=0$. By Lemma \ref{lem:decr} we have $\sup_{\tX}\psi=\sup_{\Sk(\Pi)}\psi$. Moreover, since $\psi$ is continuous and convex on each face on $\Sk(\Pi)$, we even have $\sup_{\tX}\psi= \sup_{v \in \Pi^{(0)}} \psi(v)$.

We denote by $r : \Pi' \to \Pi$ the refinement, and define $\psi_{\Pi} \in \PA_b(\Pi)$ to be the unique bounded piecewise affine function on $\Pi$ whose values at the vertices of $\Pi$ agree with those of $\psi$.
For any $v \in \Pi^{(0)} \subset \Pi'^{(0)}$ we have
        \begin{align*}
            r^* ( (\gamma + \psi_{\Pi}) \cdot c) (v)
            & = ( (\gamma + \psi_{\Pi}) \cdot c)  (v) 
            = \sum_{\substack{e \succ v\\ e \in \Pi^{(1)}}} c(e) (\gamma + \psi_{\Pi})_e (n_{e/v}),\\
            \phi \cdot (r^*c) (v) 
            & = \sum_{\substack{\tau' \succ v\\ \tau' \in \Pi'^{(1)}}} (r^*c) (\tau') \phi_{\tau'} (n_{\tau'/v})
            = \sum_{\substack{
            e \succ v\\ e \in \Pi^{(1)}, \tau' \subseteq e
            }} c(e) \phi_{\tau'} (n_{\tau'/v}).
        \end{align*}
        As $\phi$ is convex on the edges of $\Pi$, we have $\phi_{\tau'} (n_{\tau'/v}) \leq (\gamma + \psi_{\Pi})_e (n_{e/v})$ for $\tau' \subseteq e$, hence
        $$
        0
        \leq \phi \cdot (r^*c) 
        \leq (\gamma + \psi_{\Pi}) \cdot c 
        $$
        in $Z_0(\tX)$, where the first inequality holds since $ \phi \in \PAPC(\tX)$. 
        For any $v \in \Pi^{(0)}$, we infer that
        \begin{equation}
        \label{equ:bound on vertices}
        -C
        \leq -(\gamma \cdot c) (v) 
        \leq (\psi_{\Pi} \cdot c) (v)
        = \sum_{\substack{e \succ v\\e = [v, w] \in \Pi^{(1)}}} \frac{c(e)}{\ell(e)}(\psi(w)-\psi(v))
        \end{equation}
        for some constant $C>0$ independent of $\psi$, where $\ell(e)$ is the length of $e$, i.e. the unique positive rational such that $(w-v) = \ell(e) n_{e/v}$. Note that the sum in \eqref{equ:bound on vertices} is over the bounded edges adjacent to $v$, as $\psi$ is constant along unbounded edges of $\Pi$.
        
    Let $M\coloneqq |\Pi^{(0)}|$, $\Pi^{(1)}_b \coloneqq \Pi^{(1)} \cap \Sk(\Pi)$, and $M_1\coloneqq |\Pi^{(1)}_b|$. We label the vertices of $\Sk(\Pi)$ by $v_1, \ldots, v_M$ such that $\sup_{j} \psi(v_j) = \psi(v_1) = 0$, and for any $j \geq 2$ there exists $i < j $ with $[v_i, v_j] \in \Pi^{(1)}_b$. We prove by induction on $j \le M$ that there exists a constant $C_j >0$ independent of $\psi$ such that $\psi(v_i)\geq -C_j$ for all $i \leq j$.
    
    If $j=1$, we have $\psi(v_1)=0$ by our choice of labelling. Now assume that there exists $C_j >0$ independent of $\psi$ such that $\psi(v_1), \ldots, \psi(v_j) \geq -C_j$. Consider $v_{j+1}$ and $i<j+1$ such that $[v_i, v_{j+1}] \in \Pi^{(1)}_b$. 
    We have
    \begin{alignat*}{2}
        -C 
        & \leq \sum_{\substack{e \succ v_i\\e = [v_i, w]}} \frac{c(e)}{\ell(e)}(\psi(w)-\psi(v_i))
        && \text{by \eqref{equ:bound on vertices} for $v=v_i$}\\
        & \leq \sum_{\substack{e \succ v_i\\e = [v_i, w]}} \frac{c(e)}{\ell(e)}(\psi(w)+C_j)
        && \text{by induction hypothesis}\\
        & \leq \big(\min_{e \in \Pi^{(1)}_b} \tfrac{c(e)}{\ell(e)} \big)\,
        \psi(v_{j+1}) +  M_1 \big(\max_{e \in \Pi^{(1)}_b} \tfrac{c(e)}{\ell(e)}\big) \, C_j.
        && 
    \end{alignat*}
   This implies that $\psi(v_{j+1}) \geq -C_{j+1}$ for some constant $C_{j+1} \geq C_j > 0$ independent of $\psi$, and concludes the proof of the lemma.
\end{proof}

\subsubsection*{Induction step} 
Assume that Proposition \ref{prop:equi PAPC} holds for all faces of $\Sk(\Pi)$ of dimension $k$, for some $k\geq 0$. Let $\sigma \in \Sk(\Pi)$ be a face of dimension $k+1$. 

\begin{lemma} 
\label{lem: equi bound derivatives}
There exists $C>0$ independent of $\psi$ such that
        $$
        |D_v \psi(w)| \leq C
        $$
        for any vertex $w \in \sigma$ and any point $v \in \relint(\tau) \cap N_\Q$, where $\tau \in \Pi$ and $\sigma= \conv (\tau, w )$, such that $\psi|_\tau$ is affine near $v$.  
\end{lemma}
\begin{proof}
    Let $\tau$ be a $k$-dimensional face of $\sigma$ in $\Pi$, let $w$ be the unique vertex of $\sigma \setminus \tau$, and let $v \in \relint(\tau) \cap N_\Q$ such that $\psi|_\tau$ is affine near $v$.
    The convexity of $\psi$ on $\sigma$ and the induction assumption on the boundedness of $\sup_{\partial \sigma}|\psi- \sup_{\tX}\psi|$ imply that
    \begin{equation} \label{equ:equic_upper}
      D_v \psi(w) \leq \psi(w) - \psi(v) \leq C_\sigma  
    \end{equation}
    for some constant $C_\sigma>0$ independent of $\phi$. We now prove a lower bound for $D_v \psi(w)$; we can assume that $D_v \psi(w)\leq 0$.
    
    Let $\{v_j\}_{j \in J}$ be the set of vertices of $\tau$, and $\{v_l\}_{l \in L}$ the set of vertices of $\Sk(\Pi)$ such that $ \sigma_l \coloneqq \conv( \tau, v_l ) \in \Pi^{(k+1)} \cap \Sk(\Pi)$. Given $\varepsilon \in (0,1) \cap \Q$ and any $i \in J \cup L$, set $v_i^\varepsilon= \varepsilon v_i + (1-\varepsilon)v$. By \cite[\S 6.2]{BoucksomFavreJonsson} there exists a refinement $\Pi^\varepsilon$ of $\Pi$ such that
    \begin{itemize}
        \item the simplex $\sigma_l^\varepsilon \coloneqq \conv(v_l^\varepsilon, \{v_j^\varepsilon\}_{j \in J} ) \in \Pi^\varepsilon$ for any $l \in L$ (namely $\sigma_l^\varepsilon$ is obtained by scaling $\sigma_l$ by a factor $\varepsilon$), and these are precisely the bounded faces in $\Pi^\varepsilon$ of dimension $k+1$ containing $\tau^\varepsilon \coloneqq \langle \{v_j^\varepsilon\}_{j \in J} \rangle$; in particular, $\tau^\varepsilon \subset \tau$ contains $v$ in its relative interior,
        \item $\Pi^\varepsilon$ is a simplicial polyhedral structure on $\tX$.
    \end{itemize}
    Moreover, for $\varepsilon$ sufficiently small, we can assume that $\psi$ is affine on $\tau^\varepsilon$ and on the segments $[v,v_l^\varepsilon]$ for any $l \in L$. We furthermore denote by $\psi_{\Pi^\varepsilon} \in \PA_b(\Pi^\varepsilon)$ the unique bounded piecewise affine function on $\Pi^\varepsilon$ whose values at the vertices of $\Pi^\varepsilon$ agree with those of $\psi$. By possibly going to a higher refinement of $\Pi'$, we can assume that $\Pi' \geq \Pi^\varepsilon$, and denote $r \colon \Pi' \xrightarrow{t'} \Pi^\varepsilon \xrightarrow{t} \Pi$.
    
    \begin{claim} \label{claim equic} 
    Let $c\in Z_{k+1}(\Pi)_{> 0}$ as in Lemma \ref{lem:equ_positive}. For any $\tau' \in \Pi'^{(k)}$ with $\tau' \subseteq \tau^\varepsilon$, we have  
    \begin{equation*} 
    \psi \cdot (r^* c) \,(\tau')
    \leq \psi_{\Pi^\varepsilon} \cdot (t^* c) \,(\tau^\varepsilon)
    \end{equation*}
    \end{claim}
    
    \begin{proof}
    Recall that $\sigma_l^\varepsilon$ for $l \in L$ are the bounded faces in $\Pi^\varepsilon$ of dimension $k+1$ containing $\tau^\varepsilon$, and denote by $\sigma_h$ for $h \in H$ the unbounded faces of  $\Pi^\varepsilon$ of dimension $k+1$ containing $\tau^\varepsilon$. 
    As in the proof of Lemma \ref{lem:2 defn PAPC}, we can choose normal vectors relative to $\tau^\varepsilon$ such that 
    \begin{equation}
    \label{eq:claim equi}
    \sum_{l \in L} (t^*c)(\sigma_l^\varepsilon)n_{\sigma_l^\varepsilon/\tau^\varepsilon} + \sum_{h \in H} (t^*c)(\sigma_h)n_{\sigma_h/\tau^\varepsilon}=0.
    \end{equation} 
    It follows that
    $$
    \big( \psi_{\Pi^\varepsilon} \cdot t^*c \big) \,(\tau^\varepsilon) 
    = \sum_{l \in L} (t^*c)(\sigma_l^\varepsilon) (\psi_{\Pi^\varepsilon})_{\sigma_l^\varepsilon}(n_{\sigma_l^\varepsilon/\tau^\varepsilon}) + \sum_{h \in H} (t^*c)(\sigma_h) (\psi_{\Pi^\varepsilon})_{\sigma_h}(n_{\sigma_h/\tau^\varepsilon}).
    $$

    Let $\tau' \subseteq \tau^\varepsilon$ be a face of $\Pi'$ of dimension $k$. 
    We have that
    \begin{itemize}
        \item for any $l \in L$, there is a unique face $\sigma' \succ \tau'$ of $\Pi'$ of dimension $k+1$ such that $\sigma' \subseteq \sigma_l^\varepsilon$; in particular, we have $(r^*c)(\sigma')=(t^*c)(\sigma_l^\varepsilon)$ and we can take $n_{\sigma'/\tau'}=n_{\sigma_l^\varepsilon/\tau^\varepsilon}$;
        \item for any $h \in H$, there is a unique face $\sigma' \succ \tau'$ of $\Pi'$ of dimension $k+1$ such that $\sigma' \subseteq \sigma_h$; in particular, we have $(r^*c)(\sigma')=(t^*c)(\sigma_h)$ and we can take $n_{\sigma'/\tau'}=n_{\sigma_h/\tau^\varepsilon}$;
        \item for any other face $\sigma' \succ \tau'$ of $\Pi'$ of dimension $k+1$, we have that $(r^*c)(\sigma')=0$ as $\sigma'$ is not contained in a $(k+1)$-dimensional face of $\Pi^\varepsilon$; we choose any $n_{\sigma'/\tau'}$.
    \end{itemize}
    Together with \eqref{eq:claim equi}, this implies
    $$
    \psi \cdot (r^*c)\, (\tau')=
    \sum_{l \in L} (t^*c)(\sigma_l^\varepsilon)\psi_{\sigma'}(n_{\sigma_l^\varepsilon/\tau^\varepsilon})
    + \sum_{h \in H} (t^*c)(\sigma_h) \psi_{\sigma'}(n_{\sigma_h/\tau^\varepsilon})
    $$ where $\sigma' \subseteq \sigma_l^\varepsilon$ (resp. $\sigma_h$) is a uniquely determined $(k+1)$-dimensional face of $\Pi'$ containing $\tau'$. We now compare $\psi_{\sigma'}(n_{\sigma_i^\varepsilon/\tau^\varepsilon})$ with $(\psi_{\Pi^\varepsilon})_{\sigma_i^\varepsilon}(n_{\sigma_i^\varepsilon/\tau^\varepsilon})$ for any $i \in L \cup H$.
    Let $\lambda \in \Q_{>0}$ be sufficiently small that $\overline{v}_l \coloneqq v+ \lambda n_{\sigma_l^\varepsilon/\tau^\varepsilon} \in \sigma_l^\varepsilon$ for any $l \in L$ (here we are using that $v \in \relint(\tau^\varepsilon)$), and $\overline{v}_h \coloneqq v+ \lambda n_{\sigma_h/\tau^\varepsilon} \in \sigma_h$ for any $h \in H$.
    \\For $l \in L$ 
    \begin{alignat*}{2}
    \psi_{\sigma'}(\lambda n_{\sigma_l^\varepsilon/\tau^\varepsilon})
    & \leq \psi(\overline{v}_l) - \psi(v) 
    \leq \psi_{\Pi^\varepsilon}(\overline{v}_l) - \psi(v)
    &&  \quad \text{ by convexity of $\psi$ on $\sigma_l^\varepsilon$}\\
    & = \psi_{\Pi^\varepsilon}(\overline{v}_l) - \psi_{\Pi^\varepsilon}(v)
    && \quad \text{ as $\psi_{|\tau^\varepsilon}$ is affine}\\
    & = (\psi_{\Pi^\varepsilon})_{\sigma_l^\varepsilon}(\lambda n_{\sigma_l^\varepsilon/\tau^\varepsilon})
    && \quad \text{ as $\psi_{\Pi^\varepsilon}$ is affine on $\Pi^\varepsilon$.}
    \end{alignat*}
    For $h \in H$ 
    \begin{alignat*}{2}
    \psi_{\sigma'}(\lambda n_{\sigma_h/\tau^\varepsilon})
    & \leq \psi(\overline{v}_h) - \psi(v) 
    && \quad \text{ by convexity of $\psi$ on $\sigma_h$
    }\\
    & \leq (\psi \circ r_{\Pi^\varepsilon})(\overline{v}_h) - \psi(v) 
    && \quad \text{ by Lemma \ref{lem:decr}}\\
    & = \psi_{\Pi^\varepsilon} (r_{\Pi^\varepsilon}(\overline{v}_h))- \psi_{\Pi^\varepsilon}(v) 
    && \quad \text{ as $r_{\Pi^\varepsilon}(\overline{v}_h) \in \tau^\varepsilon$ and $\psi_{|\tau^\varepsilon}$ is affine}\\ 
    & = \psi_{\Pi^\varepsilon} (\overline{v}_h)- \psi_{\Pi^\varepsilon}(v) 
    && \quad \text{ by construction of $\psi_{\Pi^\varepsilon}$}\\
    & = (\psi_{\Pi^\varepsilon})_{\sigma_h}(\lambda n_{\sigma_h/\tau^\varepsilon})
    && \quad \text{ as $\psi_{\Pi^\varepsilon}$ is affine on $\Pi^\varepsilon$.}
    \end{alignat*}
    We conclude that $\psi_{\sigma'}(n_{\sigma_i^\varepsilon/\tau^\varepsilon})\leq (\psi_{\Pi^\varepsilon})_{\sigma_i^\varepsilon}(n_{\sigma_i^\varepsilon/\tau^\varepsilon})$ for any $i \in L\cup H$, hence the proof of the claim.
    \end{proof}
    
    As $c$ is a non-negative cycle and $\phi \in \PAPC(\tX)$, Claim \ref{claim equic} implies that
    \begin{equation} 
    \label{eq: lem equicont 0}
    0 
    \leq \phi \cdot (r^* c) \,(\tau')
    \leq (\gamma + \psi_{\Pi^\varepsilon}) \cdot (t^* c) \,(\tau^\varepsilon)
    \end{equation}
    for any $\tau' \in \Pi'^{(k)}$ with $\tau' \subseteq \tau^\varepsilon$. 
    In order to compute $\psi_{\Pi^\varepsilon} \cdot (t^* c) \,(\tau^\varepsilon)$, we choose normal vectors $\overline{n}_{\overline{\sigma}/\tau^\varepsilon}$ for any $(k+1)$-dimensional face $\overline{\sigma}$ of $\Pi^\varepsilon$ containing $\tau^\varepsilon$ such that, if $\overline{\sigma}=\sigma_l^\varepsilon$, we have
    $$
    v_l - v = \mu_l \,\overline{n}_{\sigma_l^\varepsilon/\tau^\varepsilon} = \tfrac{1}{\varepsilon}(v_l^\varepsilon - v)
    $$
    for some $\mu_l \in \Q_{>0}$. 
    By the balancing condition for the cycle $t^* c$ in $\Pi^\varepsilon$, we have $\sum_{\overline{\sigma}} (t^* c) (\overline{\sigma})\overline{n}_{\overline{\sigma}/\tau^\varepsilon} \in N_\tau$ which we can write in a unique way as
    $$
    \sum_{\overline{\sigma}} (t^* c)(\overline{\sigma})\,\overline{n}_{\overline{\sigma}/\tau^\varepsilon}= \sum_{j \in J'} \beta_j (v_j^\varepsilon - v) 
    $$ for some $\beta_j \in \Q_{> 0}$ and $J'\subsetneq J$. We compute the following intersection product at $\tau^\varepsilon$:
    \begin{align*} 
    \psi_{\Pi^\varepsilon} \cdot (t^* c)(\overline{\sigma}) \,(\tau^\varepsilon) 
    & = \sum_{l \in L} c(\sigma_l) (\psi_{\Pi^\varepsilon})_{\sigma^\varepsilon_l}(\overline{n}_{\sigma^\varepsilon_l/\tau^\varepsilon}) - (\psi_{\Pi^\varepsilon})_{\tau^\varepsilon}
    \big( \sum_{\overline{\sigma}} (t^* c)(\overline{\sigma}) \,\overline{n}_{\overline{\sigma}/\tau^\varepsilon} \big) 
    \\
    & = \sum_{l \in L} c(\sigma_l) \frac{\psi_{\Pi^\varepsilon}(v_l^\varepsilon) - \psi_{\Pi^\varepsilon}(v)}{\varepsilon \mu_l} 
    - (\psi_{\Pi^\varepsilon})_{\tau^\varepsilon}
    \big( \sum_{j \in J'} \beta_j (v_j^\varepsilon - v) \big) \\
    & = \sum_{l \in L} \frac{c(\sigma_l)}{\varepsilon \mu_l} D_v \psi(v_l)
    - \sum_{j \in J'} \beta_j \big( \psi_{\Pi^\varepsilon} (v_j^\varepsilon)- \psi_{\Pi^\varepsilon}(v) \big) \\
    & = \sum_{l \in L} \frac{c(\sigma_l)}{\varepsilon \mu_l} D_v \psi(v_l)
    - \sum_{j \in J'} \beta_j D_v \psi(v_j) 
    \end{align*}
    where the last equalities hold as $\psi$ is affine on the segments $[v,v_l^\varepsilon]$ and on $\tau^\varepsilon$.
    We obtain
    \begin{align} 
    \label{eq: lem equicont 2}
    \psi_{\Pi^\varepsilon} \cdot (t^* c) \,(\tau^\varepsilon) 
    \leq \big(\min_{l \in L}\tfrac{c(\sigma_l)}{\varepsilon \mu_l}\big)\, D_v\psi(w) + |L| \big(\max_{l \in L}\tfrac{c(\sigma_l)}{\varepsilon \mu_l}\big)\, C_\sigma - \sum_{j \in J'} \beta_j D_v \psi(v_j)
    \end{align}
    where $C_\sigma>0$ is the upper bound on $D_v \psi(v_l)$ for $v$ and $v_l \in \sigma_l$ in \eqref{equ:equic_upper}.
    
    Combining \eqref{eq: lem equicont 0} and \eqref{eq: lem equicont 2}, we obtain
    \begin{equation} \label{eq: lem equicont 3}
    \big(\min_{l \in L}\tfrac{c(\sigma_l)}{\varepsilon \mu_l}\big)\, D_v\psi(w) 
    \geq 
    - |L|\big(\max_{l \in L}\tfrac{c(\sigma_l)}{\varepsilon \mu_l}\big)\, C_\sigma
    + \sum_{j \in J'} \beta_j D_v \psi(v_j) 
    - \gamma \cdot (t^*c) \,(\tau^\varepsilon).
    \end{equation} 
    The coefficients $\beta_j$ depend on $v$, but they can be extended to a continuous function $\tau \to \R_{\geq 0}^{|J|}$ which then attains a finite maximum. By induction hypothesis we have $|D_v \psi(v_j)| \leq \diam(\tau)\Lip_\tau(\psi) \leq C_\tau$ for some constant $C_\tau>0$ independent of $\psi$, and by Lemma \ref{lem:equ_positive} $c$ is a strictly positive cycle, thus \eqref{eq: lem equicont 3} implies that 
    $$
    D_v \psi(w) \geq -C
    $$ for some constant $C>0$ independent of $\psi$, and concludes the proof of Lemma \ref{lem: equi bound derivatives}.
\end{proof}

We conclude the induction step, hence the proof of Proposition \ref{prop:equi PAPC}, using \cite[Proposition A.1]{BoucksomFavreJonsson} and the argument after the statement of Proposition 6.3 in \emph{loc. cit.}
\end{proof}

\subsection{Compactness}

\begin{prop} \label{prop:equi PSH}
Let $K \subset \tX$ be a compact set.
        \begin{enumerate}
        \item There exists $C>0$ such that for any $\psi = \phi-\gamma$ with $\phi \in \PSH(\tX, \gamma)$ we have $$\sup_K |\psi - \sup_\tX \psi| \leq C.$$
        
        \item The family $(\psi|_K)$ ranging over $\psi = \phi-\gamma$ with $\phi \in \PSH(\tX, \gamma)$ is uniformly Lipschitz on $K$.

        \item The family $(\phi|_K)$ ranging over $\phi \in \PSH(\tX, \gamma)$ is (uniformly) equicontinuous on $K$.
    \end{enumerate}
\end{prop}

\begin{proof}
Let $K \subset \tX$ be a compact subset of $\tX$. Let $\Pi$ be a polyhedral structure on $\tX$ such that its bounded part contains $K$, and $\gamma$ is defined on $\Pi$. By Proposition \ref{prop:quasi-proj}, there exists a simplicial quasi-projective refinement $\Pi'$ of $\Pi$; in particular, we have $K \subseteq \Sk(\Pi')$. The first two items now follow from Proposition \ref{prop:equi PAPC} applied to $\Sk(\Pi')$ and to a sequence in $\PAPC(\tX,\gamma)$ decreasing pointwise to $\phi$, hence uniformly on $K$. The third item follows from the first two.
\end{proof}

\begin{theorem} \label{thm:comp}
Endow $\PSH(\tX,\gamma)$ with the topology of uniform convergence on compact subsets. Then the map $\phi \mapsto \sup_{\tX} (\phi-\gamma)$ is continuous and proper on $\PSH(\tX,\gamma)$.
    \\ As a consequence, the space $\PSH(\tX, \gamma)/\R$ is compact.
\end{theorem}

\begin{proof}
Let $(\phi_j)_j$ be a sequence in $\PSH(\tX,\gamma)$, converging to $\phi$ on every compact subset of $\tX$. Let $\Pi'$ be a simplicial polyhedral complex on $\tX$, and write $K\coloneqq \lvert \Sk(\Pi') \rvert$. By the inequality $(\phi_j-\gamma) \le (\phi_j-\gamma) \circ r_{\Pi'}$, the supremum of $(\phi_j-\gamma)$ is attained on $K $ for all $j$, and by uniform convergence on $K$ they converge to $\sup_K (\phi-\gamma) = \sup_{\tX} (\phi-\gamma)$. This proves continuity.
\\Now let $(\phi_j)_j$ be a sequence in $\PSH(\tX,\gamma)$ with $(\sup_{\tX} (\phi_j - \gamma))_j$ bounded. By considering a increasing sequence $(K_l)_{l \ge 0}$ of compact subsets such that $\cup_{l \ge 0} K_l = \tX$ and a standard diagonal extraction argument, it is enough to prove that given a compact subset $K \subset \tX$, the sequence $(\phi_j)_j$ has a subsequence converging uniformly on $K$. This is now a direct consequence of Proposition \ref{prop:equi PSH} and Arzela-Ascoli's theorem, which concludes the proof.
\end{proof}

\section{The polyhedral Monge--Ampère operator}
In this section, we first introduce \emph{polyhedral Monge--Amp\`ere measures} associated with piecewise affine functions on a balanced polyhedral space $\tX$. We then extend the Monge--Amp\`ere operator to functions in $\PCreg(\tX,\gamma)$, study its main properties in this setting, and finally further extend it to the space $\PSH(\tX,\gamma)$. We begin by recalling the construction of the Monge--Amp\`ere measure in the real setting.

\subsection{The real Monge--Ampère measure}
Let $\phi \colon U \to \R$ be a convex function defined on an open subset $U \subset N_{\R}$, and let $d\ell$ be the Lebesgue measure on the dual real vector space $M_{\R} \coloneqq \Hom(N, \R)$.
For $x_0 \in U$, the \emph{subgradient of $\phi$ at} $x_0$ is defined as
$$ \nabla \phi (x_0) \coloneqq \left\{ p \in M_{\R} \;| \; \phi(x_0) + \langle p , x-x_0 \rangle \le \phi(x) \; \forall x \in U \right\}.$$
Geometrically, this is the set of covectors cutting out affine hyperplanes in $N_{\R}$ that meet the graph $\Gamma$ of $\phi$ at $x_0$, and are below $\Gamma$ on all of $U$. More generally, if $E \subset U$ is a Borel set, we set $$\nabla \phi(E) = \cup_{x_0 \in E} \nabla \phi(x_0).$$

\begin{defi}[{\cite{gutierrez}}]
Let $\phi \colon U \to \R$ be a convex function defined on an open subset $U \subset N_{\R}$.
The real Monge--Ampère measure of $\phi$ is defined as
$$ \RMA (\phi) (E) = d\ell(\nabla \phi(E)),$$
for any Borel set $E \subset U$.
\end{defi}
If $\phi$ has $\mathcal{C}^2$-regularity, then this is nothing but the density measure $\det(\nabla^2 \phi) d\ell^{\vee}$.

We will now give an explicit formula for the real Monge--Ampère measure of a convex piecewise affine function $\phi$ on $N_{\R}$; it is a standard fact \cite[Corollary 19.1.2]{Rockafellar} that there exists a finite family of affine forms $(f_i)_{i \in I}$ on $N_{\R}$ such that $\phi = \max_{i \in I} f_i$.

\begin{defi}
Let $(f_i)_{i \in I}$ be a finite family of affine forms on $N_{\R}$, and write $f_i = \ell_i +c_i$, with $\ell_i \in M_{\R}$ and $c_i \in \R$. We say $(f_i)_{i \in I}$ has \emph{positive volume} if for any $i \in I$, the family $(\ell_j -\ell_i)_{j \neq i}$ contains a basis of $M_{\R}$.
\\Equivalently, this means that the convex polytope $P= \conv(\ell_i | i \in I) \subset M_{\R}$ has positive volume.
\end{defi}
The condition of having positive volume is called \emph{very separating} in \cite[Définition 1.7.4]{CLD}. 
\begin{prop} \label{prop: real MA}
Let $(f_i)_{i \in I}$ be a finite family of affine forms on $N_{\R}$, and set $\phi = \max (f_i)_{i \in I}$. For $x \in N_{\R}$, set $I(x) = \left\{ i \;  | \;\phi(x) = f_i(x) \right\}$, and 
$$S \coloneqq \left\{ x \in N_{\R} \;| \;(f_i)_{i \in I(x)} \; \text{has positive volume} \right\}.$$
Then $S$ is finite, and the real Monge-Ampère measure is given by
$$\RMA(\phi) = \sum_{x \in S} \lambda_x \delta_x,$$
where $\lambda_x = \Vol (P_x)$, with $P_x = \conv( \ell_i | i \in I(x)) \subset M_{\R}$ and $\Vol$ is the Lebesgue measure on $M_{\R}$ normalized by the integral lattice $M$.
\end{prop}

\begin{proof} The first part of the statement is proved in \cite[Proposition 1.7.8]{CLD}. The second part follows from the equality $\left(\nabla \phi\right) (x) = P_x$ for $x \in S$, which follows easily from the definition.
\end{proof} 

\subsection{Polyhedral Monge--Ampère for $\PA(\tX)$}
Let $\tX \subseteq N_{\R}$ be a balanced polyhedral space of dimension $d$. The polyhedral Monge--Ampère measure of a piecewise affine function on $\tX$ is a discrete measure supported on the vertices of a polyhedral structure, and is defined as follows.
\begin{defi}\label{def:MA-PA}
    Let $\phi \in \PA(\tX)$, and let $\Pi$ be a polyhedral structure on $\tX$ on which $\phi$ and $[\tX]$ are defined. We write
    $$\lambda_i \coloneqq (\phi^d \cdot [\tX]) (v_i)$$
    for $v_i \in \Pi^{(0)}$. The \emph{polyhedral Monge-Ampère measure of $\phi$} is the finite atomic measure
    $$\MA_{\poly}(\phi) \coloneqq \sum_{v_i \in \Pi^{(0)}} \lambda_i \delta_{v_i}.$$
\end{defi}
It follows from Lemma \ref{lem:pairing and pullback} that this is indeed independent of $\Pi$.
\begin{defi} \label{def:degree}
Let $c \in Z_0(\tX)$ be a $0$-cycle. 
The \emph{degree} of $c$ is the number
$$\deg(c) = \sum_{v \in \Pi^{(0)}} c(v)$$
whenever $c \in Z_0(\Pi)$. For $\phi \in \PA(\tX)$, the \emph{degree} of $\phi$ is the rational number
    $$\deg(\phi) \coloneqq \int_{\tX} \MA_{\poly}(\phi) = \deg(\phi^d \cdot [\tX]).$$ 

\end{defi}
\begin{lemma}\label{lem:degree}
    For any $\phi, \psi \in \PA(\tX)$ such that $(\phi-\psi)$ is a bounded function on $\tX$, we have $\deg(\phi) = \deg(\psi).$ In particular, if $\phi \in \PA(\tX, \gamma)$ for some $\gamma \in \PA(\tX)$, then $\deg(\phi) = \deg(\gamma)$.
\end{lemma}
\begin{proof}
Let $\Pi$ be a polyhedral complex on $\tX$ on which $\phi$ and $\psi$ are defined. Let 
$$c = (\sum_{j=0}^{d-1} \phi^j \cdot \psi^{d-j-1} \cdot [\tX] ) \in Z_1(\Pi),$$ then by multilinearity and commutativity we have the following equality of $0$-cycles:
$$(\phi- \psi) \cdot c = (\phi^d \cdot [\tX]) - (\psi^d \cdot [\tX]).$$
We are thus left to prove the following: if $\eta \in \PA_b(\Pi)$ and $c \in Z_1(\Pi)$, then $\deg(\eta \cdot c) =0$, for which we argue as in \cite[Lemma 8.3]{AllermannRau2010}.
The quantity $\deg(\eta \cdot c)$ is the sum over vertices $v$ of $\Pi$ of the sum of outgoing slopes of $\eta$ at $v$, weighted by $c$; each bounded edge appears twice, once for each of its endpoints which have opposite normal vectors, so that the two corresponding slopes cancel out. Since $\eta$ has slope zero along each unbounded edge - being a bounded affine function - we get $\deg(\eta \cdot c) =0$. 
\end{proof}

One can define the mixed Monge--Ampère operator of tuples of piecewise affine functions on $\tX$.

\begin{defi} \label{def:PMA PA}
Let $\phi_1,\ldots,\phi_d$ be a $d$-tuple in $\PA(\tX)$. The \emph{mixed polyhedral Monge--Ampère measure of $(\phi_1,\ldots,\phi_d)$} is the finite atomic measure on $\tX$
$$
\PMA(\phi_1,\ldots,\phi_d) \coloneqq \sum_{v_i \in \Pi^{(0)}} \phi_1 \cdot (\phi_2 \cdot ( \ldots ( \phi_d \cdot [\tX]))) \, \delta_{v_i},
$$
for any polyhedral structure $\Pi$ where the functions $\phi_i$ and $[\tX]$ are defined.
The \emph{mixed degree} of $(\phi_1, \dotsc, \phi_d)$ is the rational number 
 \[
    \deg\left(\phi_1, \ldots, \phi_d\right) \coloneqq \int_{\tX}\MA_{\poly}(\phi_1,\ldots, \phi_d).
    \]
\end{defi}

\begin{prop}\label{prop/def}
The map
$$
(\phi_1,\ldots,\phi_d) \to \PMA(\phi_1,\ldots,\phi_d) 
$$
from the space of $d$-tuples of piecewise affine functions on $\tX$ to the space of finite atomic measures on $\tX$, satisfies the following properties:
\begin{enumerate}
    \item symmetry:
    $$
    \PMA(\phi_1, \ldots, \phi_i, \ldots,\phi_k, \ldots,\phi_d)=\PMA(\phi_1, \ldots, \phi_k, \ldots,\phi_i, \ldots,\phi_d)
    $$
    \item invariance with respect to additive constants $c_i \in \Q$:
    $$
    \PMA(\phi_1 + c_1, \ldots,\phi_d+c_d)= \PMA(\phi_1 ,\ldots,\phi_d) 
    $$
    \item linearity in each argument: for $\lambda,\mu \in \Q$,
    \begin{align*}
    \PMA(\lambda \phi_0 + & \mu \phi_1, \phi_2, \ldots, \phi_d)
    \\ & = \lambda\,\PMA(\phi_0 , \phi_2, \ldots,\phi_d) + \mu \,\PMA( \phi_1, \phi_2, \ldots,\phi_d)
    \end{align*}
    \item for $k \in \N$ and any $\lambda_1, \ldots, \lambda_k \in \Q$,
    \[
    \MA_{\poly}\left(\lambda_1\phi_1+ \ldots + \lambda_k\phi_k\right) = \sum_{i_1, \ldots, i_d = 1}^k\lambda_{i_1}\ldots \lambda_{i_d}\MA_{\poly}\left(\phi_{i_1},\ldots,\phi_{i_d} \right) 
    \]
    where the left-hand side is defined as in Definition \ref{def:MA-PA}.
    \end{enumerate}
\end{prop}
\begin{proof}
$(1)$ follows from the commutativity of the intersection pairing, while $(3)$ and $(4)$ follow from its linearity, see Lemma \ref{lem:properties pairing}; $(2)$ follows from the definition of intersection pairing, see Equation \ref{eq:intersection product}.
\end{proof}

\begin{lemma}\label{lem:integral-formulae}
   Let $\phi_0, \ldots, \phi_d$ be a collection of functions in $\PA(\tX)$. Let $i, k \in \{0, \ldots, d\}$ such that $\phi_i$ and $\phi_k$ are bounded. We have the following equality: 
    \[
    \int_{\tX} \phi_i\MA_{\poly}(\phi_0,\ldots, \phi_{i-1},\phi_{i+1},\ldots, \phi_d) = \int_{\tX} \phi_k\MA_{\poly}(\phi_0,\ldots, \phi_{k-1},\phi_{k+1},\ldots, \phi_d).
    \]
As a consequence, if $\psi \in \PA(\tX,\gamma)$ and $\phi_1,\ldots,\phi_d \in \PA(\tX,\gamma)$, then the following \emph{integration by parts formula} holds
\begin{align*}
   \int_{\tX} (\psi-\gamma) & \MA_{\poly} (\phi_1,\phi_2,\ldots,\phi_d) 
   \\ & = \int_{\tX} (\phi_1-\gamma) \MA_{\poly}(\psi,\phi_2,\ldots,\phi_d) + \int_{\tX} (\psi-\phi_1) \MA_{\poly} (\gamma,\phi_2,\ldots,\phi_d), 
\end{align*}
and similar formulae for other indices $i\neq 1$.
\end{lemma}

\begin{proof}
Let $\Pi$ be a polyhedral structure where the functions $\phi_i$ and $[\tX]$ are defined. By the symmetry of $\PMA$, we can assume $i=0$ and $k=1$. For simplicity, we write $\phi_0=f$, $\phi_1=g$, and $\phi_2 \cdot (\ldots(\phi_d \cdot [\tX])\ldots)=\sum_{\rho \in \Pi^{(1)}} a_\rho \rho$. If $\rho$ is an unbounded face, then $f_\rho=g_\rho=0$ since by assumption $f$ and $g$ are bounded functions. If instead $\rho$ is bounded, we write $\rho=\langle v, w \rangle$ and $w-v= \lambda_\rho n_{\rho/v}$ for some $\lambda_\rho \in \Q$. As $f$ and $g$ are defined on $\Pi$, we have $$g_\rho(n_{\rho/v})=\frac{g(w)-g(v)}{\lambda_{\rho}}$$ and similarly for $f$.
By the definition on $\PMA$, we have
\begin{align}\label{eq:MAgraph}
\begin{split}
\int_{\tX} f \,\PMA(g,\phi_2,\ldots,\phi_d) & = \sum_{v \in \Pi^{(0)}} f(v) \sum_{\rho \succ v} a_\rho g_\rho(n_{\rho/v})
\\ & = \sum_{v \in \Pi^{(0)}} f(v) \sum_{w \in \Pi^{(0)} \,: \langle v,w \rangle \in \Pi^{(1)}} \frac{a_{\langle v,w \rangle}}{\lambda_{\langle v,w \rangle}}(g(w)-g(v))
\\ & = \sum_{\rho = \langle v,w \rangle \in \Pi^{(1)}} \frac{a_{\rho}}{\lambda_{\rho}}(f(v)-f(w))(g(w)-g(v)).
\end{split}
\end{align}
As this is symmetric with respect to the exchange of $f$ and $g$, the required equality holds.

By linearity of $\PMA$ (Proposition \ref{prop/def}) and applying the above equality, we obtain the integration by parts formula:
{\allowdisplaybreaks
\begin{align*}
   \int_{\tX} & (\psi-\gamma) \MA_{\poly} (\phi_1,\phi_2,\ldots,\phi_d) 
   \\ & = \int_{\tX} (\psi-\gamma) \MA_{\poly} (\phi_1-\gamma,\phi_2,\ldots,\phi_d)  + \int_{\tX} (\psi-\gamma) \MA_{\poly} (\gamma,\phi_2,\ldots,\phi_d) 
   \\ & = \int_{\tX} (\phi_1-\gamma)  \MA_{\poly} (\psi-\gamma,\phi_2,\ldots,\phi_d)  + \int_{\tX} (\psi-\gamma) \MA_{\poly} (\gamma,\phi_2,\ldots,\phi_d) 
   \\ & = \int_{\tX} (\phi_1-\gamma)  \MA_{\poly} (\psi,\phi_2,\ldots,\phi_d) - \int_{\tX} (\phi_1-\gamma) \MA_{\poly} (\gamma,\phi_2,\ldots,\phi_d)  
   \\ & \hspace{180pt} + \int_{\tX} (\psi-\gamma)  \MA_{\poly} (\gamma,\phi_2,\ldots,\phi_d) 
   \\ & = \int_{\tX} (\phi_1-\gamma)  \MA_{\poly} (\psi,\phi_2,\ldots,\phi_d)  + \int_{\tX} (\psi-\phi_1)  \MA_{\poly} (\gamma,\phi_2,\ldots,\phi_d).
\end{align*}
}
\end{proof}
\begin{lemma} \label{lem:nega} 
   Let $\phi \in \PAPC(\tX)$. The bilinear form
   $$ (f, g) \mapsto \int_{\tX} f\, \PMA(g, \phi,\ldots, \phi)$$
   is symmetric and semi-negative on $\PA_b(\tX)$.
\end{lemma}
\begin{proof}
The symmetry is a direct consequence of Lemma \ref{lem:integral-formulae}. To prove semi-negativity, let $f$ be bounded and let $\Pi$ be a polyhedral structure on which both $f$ and $\phi$ are defined, and write $(\phi^{d-1} \cdot [\tX]) = \sum_{\rho \in \Pi^{(1)}} a_{\rho} \rho$. Note that each $a_{\rho} \ge 0$ since $\phi$ is polyhedrally convex.
\\As in the previous proof, write bounded edges $\rho$ of $\Pi$ as $\rho = \langle v, w \rangle$ and $w-v =\lambda_{\rho} n_{\rho/v}$. By Equation \ref{eq:MAgraph}, we then have 
$$ \int_{\tX} f\, \PMA(f, \phi,\ldots, \phi) = \sum_{\rho = \langle v,w \rangle \in \Pi^{(1)}} - \frac{a_{\rho}}{\lambda_{\rho}}(f(v)-f(w))^2,$$
which proves non-positivity. 
\end{proof}

\begin{prop}\label{prop:MA-comp} 
Let $\Pi$ be a polyhedral structure on $\tX$. Let $\phi \in \PAPC(\Pi')$ be a piecewise affine, polyhedrally convex function, defined on a refinement $\Pi'$ of $\Pi$. For any top-dimensional face $\sigma \in \Pi^{(d)}$, we have the equality of measures
$$ \mathbf{1}_{\relint(\sigma)} \MA_{\poly}(\phi) = c \;  d! \; \RMA(\phi_{| \relint(\sigma)})$$
where $c=[\tX](\sigma)$.
\end{prop}
Note that the function $\phi_{| \relint(\sigma)}$ is convex in the usual sense by \cite[Proposition 5.10]{BBS}.
\begin{proof}
By definition of $\MA_{\poly}$ and by Proposition \ref{prop: real MA} both measures are atomic, supported on the finite set $\Pi'_0 \cap \relint(\sigma)$. As a result we choose $v \in (\Pi')^{(0)} \cap \relint(\sigma)$, we may work with the local fan of $\Pi'$ at $v,$ which is the germ of a complete fan $\Sigma$ in $N_{\sigma} \otimes \R$, and $\phi$ corresponds to a toric Cartier divisor $D$ on the corresponding toric variety $Y_{\Sigma}$.

We infer from Example \ref{ex:fan} that $(\phi^d \cdot [\tX]) (v) =c \,(D^d)$ or equivalently, $\MA_{\poly}(\phi)(v) =c \,(D^d) \delta_v$.
Since $\phi$ is convex by assumption, the divisor $D$ is nef on $Y_{\Sigma}$. If $P$ is the Newton polytope of $\phi$, we have $\MA_{\R}(\phi) = \Vol(P) \delta_v$ near $v$ by Proposition \ref{prop: real MA}. The integral points of $kP$ parametrize monomial sections of $\mathcal{O}(kD)$ for all $k$, hence 
$$h^0(kD) = \#(kP \cap M) = \Vol(kP) + O(k^{d-1}),$$ while by asymptotic Riemann--Roch we have also
$$h^0(kD) = \frac{(D^d)}{d!} k^d + O(k^{d-1}).$$ These yield $\Vol(P) = \frac{(D^d)}{d!}$, hence $\MA_{\poly}(\phi)= c\; d! \,\RMA(\phi)$ in a neighbourhood of $v$, as desired.
\end{proof}

\subsection{Polyhedral Monge--Ampère for $\PCreg(\tX, \gamma)$}
In this section we extend the Monge--Ampère operator from piecewise affine functions to $\PCreg(\tX, \gamma)$ following the strategy of \cite{BedfordTaylor1982}, also adopted in \cite[Section 3]{BoucksomFavreJonsson2015} in the non-archimedean setting.

Let $\tX \subset N_{\R}$ be a balanced polyhedral space of dimension $d$. Let $\gamma$ be a piecewise affine function on a simplicial polyhedral structure $\Pi$, and let $\overline{\tX}_\Pi$ be the compactification of $\tX$ with respect to $\Pi$. Recall that a \emph{Radon measure on $\overline{\tX}_\Pi$} is a continuous linear functional on the space $\mathcal{C}^0(\overline{\tX}_\Pi)$ of continuous, real-valued functions on $\overline{\tX}_\Pi$ with respect to the sup norm. The set of Radon measures on $\overline{\tX}_\Pi$ is denoted by $\mathcal{M}\left(\overline{\tX}_\Pi\right)$, and is endowed with the topology of weak convergence of Radon measures. 

\begin{theorem} \label{thm:PMA PCreg}
    The operator
    $$(\phi_1,\ldots,\phi_d) \longmapsto \MA_{\poly}(\phi_1,\ldots,\phi_d),$$
    defined on $\PAPC(\tX, \gamma)^d$ in Definition \ref{def:PMA PA}, and taking values in 
    \[
    \left\{ \mu \in \mathcal{M}\left(\overline{\tX}_\Pi\right) \; |  \;\int_{\overline{\tX}_\Pi} \mu = \deg(\gamma) \right \},
    \]
    admits a unique continuous extension to $\PCreg(\tX,\gamma)^d$, satisfying the properties of
\begin{enumerate}
    \item symmetry:
    $$
    \PMA(\phi_1, \ldots, \phi_i, \ldots,\phi_k, \ldots,\phi_d)=\PMA(\phi_1, \ldots, \phi_k, \ldots,\phi_i, \ldots,\phi_d)
    $$
    \item invariance with respect to additive constants $c_i \in \R$:
    $$
    \PMA(\phi_1 + c_1, \ldots,\phi_d+c_d)= \PMA(\phi_1 ,\ldots,\phi_d) 
    $$
    \item linearity in each argument: for $t \in [0,1]$
    \begin{align*}
    \PMA(t\phi_0 + & (1-t) \phi_1, \phi_2, \ldots, \phi_d)
    \\ & = t\, \PMA(\phi_0 , \phi_2, \ldots,\phi_d) + (1-t) \PMA( \phi_1, \phi_2, \ldots,\phi_d)
    \end{align*}
    \item positivity: if $\psi \in \mathcal{C}^0(\overline{\tX}_\Pi)$ is non-negative, then $$\int_{\overline{\tX}_\Pi} \psi\, \PMA(\phi_1,\ldots,\phi_d) \geq 0$$
    \item the integration by parts formula: if $\psi, \phi_1,\ldots, \phi_d \in \PCreg(\tX,\gamma)$, then
    \begin{align*}
   \int_{\overline{\tX}_\Pi} & (\psi-\gamma) \MA_{\poly} (\phi_1,\phi_2,\ldots,\phi_d) 
   \\ & = \int_{\overline{\tX}_\Pi} (\phi_1-\gamma) \MA_{\poly}(\psi,\phi_2,\ldots,\phi_d) + \int_{\overline{\tX}_\Pi} (\psi-\phi_1) \MA_{\poly} (\gamma,\phi_2,\ldots,\phi_d)
   \end{align*}

    \end{enumerate}
\end{theorem}

Before moving to the proof, let us mention that it will be shown in Proposition \ref{prop:suppMA} that if $\gamma$ is positive at infinity (Definition \ref{def:posinf}), then $\MA_{\poly}(\phi_1,\ldots,\phi_d)$ is always supported on $\tX$, despite being defined \emph{a priori} as a measure on $\overline{\tX}_\Pi$; as a result, in this case the Monge--Ampère measure defined above will not depend on the choice of compactification.
\begin{proof}
Fix $0 \leq k \leq d$ and $\phi'_{k+1}, \ldots, \phi'_d$ functions in $\PAPC(\tX, \gamma)$. Consider the following assertion:  any $k$-tuple $\phi_1,\ldots, \phi_k$ in $\PCreg(\tX, \gamma)$ defines a Radon measure $\PMA(\phi_1,\ldots,\phi_k,\phi'_{k+1},\ldots,\phi'_d)$ of mass $\deg(\gamma)$ such that
\begin{enumerate}[label=(\roman*)]
    \item if $\phi_1,\ldots,\phi_k$ are in $\PAPC(\tX,\gamma)$, then $\PMA(\phi_1,\ldots,\phi_k,\phi'_{k+1},\ldots,\phi'_d)$ coincides with the mixed polyhedral Monge--Ampère operator in Definition \ref{def:PMA PA};
    \item the map
    $$ (\psi, \phi_1, \ldots, \phi_k) \to \int_{\overline{\tX}_\Pi} (\psi-\gamma)\PMA(\phi_1,\ldots,\phi_k, \phi'_{k+1},\ldots,\phi'_d)$$
    is continuous along decreasing sequences in $\PCreg(\tX,\gamma).$ 
\end{enumerate}

For $k=0$, the assertion holds as $\PMA(\phi'_1,\ldots,\phi'_d)$ is the finite atomic measure as defined in Definition \ref{def:PMA PA}. Assuming the assertion holds for $k$, we now prove it for $k+1$. Given $(\phi_1,\ldots,\phi_{k+1}) \in \PCreg(\tX,\gamma)^{k+1},$ we define $\PMA(\phi_1,\ldots,\phi_{k+1},\phi'_{k+2},\ldots,\phi'_d)$ by forcing the integration by parts formula, for every $\psi \in \PA_b(\tX)$:
\begin{align} \label{equ:formula PMA}
\begin{split}
    \int_{\overline{\tX}_\Pi} \psi \,\PMA & (\phi_1,\ldots,\phi_{k+1},\phi'_{k+2},\ldots,\phi'_d) 
    \\ & \coloneqq 
    \int_{\overline{\tX}_\Pi} (\phi_{k+1}-\gamma) \PMA(\phi_1,\ldots,\phi_k,\psi + \gamma,\phi'_{k+2},\ldots,\phi'_d) 
    \\ & +
    \int_{\overline{\tX}_\Pi} (\psi + \gamma - \phi_{k+1}) \PMA(\phi_1,\ldots,\phi_k,\gamma,\phi'_{k+2},\ldots,\phi'_d).
\end{split}
\end{align}
Note that the right-hand side is well-defined and continuous along decreasing sequences as a function of $(\phi_1, \ldots,\phi_{k+1})$, by the hypothesis that the assertion holds for $k$. By Lemma \ref{lem:integral-formulae}, Equation \ref{equ:formula PMA} coincides with Definition \ref{def:PMA PA} when $\phi_1,\ldots,\phi_{k+1}$ are in $\PAPC(\tX,\gamma)$.

For the total mass, let $\psi \equiv 1$ and compute
{\allowdisplaybreaks
\begin{align*}
   \int_{\overline{\tX}_\Pi} 1 \PMA & (\phi_1,\ldots,\phi_{k+1},\phi'_{k+2},\ldots,\phi'_d) 
    \\ & =
    \int_{\overline{\tX}_\Pi} (\phi_{k+1}-\gamma) \PMA(\phi_1,\ldots,\phi_k,1 + \gamma,\phi'_{k+2},\ldots,\phi'_d) 
    \\ & +
    \int_{\overline{\tX}_\Pi} (1 + \gamma - \phi_{k+1}) \PMA(\phi_1,\ldots,\phi_k,\gamma,\phi'_{k+2},\ldots,\phi'_d)
    \\ & =
    \int_{\overline{\tX}_\Pi} (\phi_{k+1}-\gamma) \PMA(\phi_1,\ldots,\phi_k, \gamma,\phi'_{k+2},\ldots,\phi'_d) 
    \\ & +
    \int_{\overline{\tX}_\Pi} \PMA(\phi_1,\ldots,\phi_k,\gamma,\phi'_{k+2},\ldots,\phi'_d) 
    \\ & +
    \int_{\overline{\tX}_\Pi} (\gamma - \phi_{k+1}) \PMA(\phi_1,\ldots,\phi_k,\gamma,\phi'_{k+2},\ldots,\phi'_d)
    \\ & = \deg(\gamma)
\end{align*}
}
since $\PMA(\phi_1,\ldots,\phi_k,\gamma,\phi'_{k+2},\ldots,\phi'_d) $ is a Radon measure of mass $\deg(\gamma)$ by hypothesis.

We now check that the map 
$$ \psi \to \int_{\overline{\tX}_\Pi} \psi \,\PMA (\phi_1,\ldots,\phi_{k+1},\phi'_{k+2},\ldots,\phi'_d) $$ is a linear and positive functional on $\PA_b(\tX)$. Let $\psi_1$ and $\psi_2$ in $\PA_b(\tX)$. We have
{\allowdisplaybreaks
\begin{align*}
   \int_{\overline{\tX}_\Pi} (\psi_1 + \psi_2) \PMA & (\phi_1,\ldots,\phi_{k+1},\phi'_{k+2},\ldots,\phi'_d) 
    \\ & =
    \int_{\overline{\tX}_\Pi} (\phi_{k+1}-\gamma) \PMA(\phi_1,\ldots,\phi_k,\psi_1 + \psi_2 + \gamma,\phi'_{k+2},\ldots,\phi'_d) 
    \\ & +
    \int_{\overline{\tX}_\Pi} (\psi_1 + \psi_2 + \gamma - \phi_{k+1}) \PMA(\phi_1,\ldots,\phi_k,\gamma,\phi'_{k+2},\ldots,\phi'_d)
    \\ & + \int_{\overline{\tX}_\Pi} (\phi_{k+1}-\gamma) \PMA(\phi_1,\ldots,\phi_k,\gamma,\phi'_{k+2},\ldots,\phi'_d) 
    \\ & + \int_{\overline{\tX}_\Pi} (\gamma - \phi_{k+1}) \PMA(\phi_1,\ldots,\phi_k,\gamma,\phi'_{k+2},\ldots,\phi'_d) 
    \\ & = \int_{\overline{\tX}_\Pi} (\phi_{k+1}-\gamma) \PMA(\phi_1,\ldots,\phi_k,\psi_1 + \gamma,\phi'_{k+2},\ldots,\phi'_d) 
    \\& + \int_{\overline{\tX}_\Pi} (\phi_{k+1}-\gamma) \PMA(\phi_1,\ldots,\phi_k,\psi_2 + \gamma,\phi'_{k+2},\ldots,\phi'_d) 
    \\ & + \int_{\overline{\tX}_\Pi} (\psi_1 + \gamma - \phi_{k+1}) \PMA(\phi_1,\ldots,\phi_k,\gamma,\phi'_{k+2},\ldots,\phi'_d)
    \\ & + \int_{\overline{\tX}_\Pi} (\psi_2 + \gamma - \phi_{k+1}) \PMA(\phi_1,\ldots,\phi_k,\gamma,\phi'_{k+2},\ldots,\phi'_d)
    \\ & = \int_{\overline{\tX}_\Pi} \psi_1\, \PMA  (\phi_1,\ldots,\phi_{k+1},\phi'_{k+2},\ldots,\phi'_d) 
    \\& + \int_{\overline{\tX}_\Pi} \psi_2 \,\PMA  (\phi_1,\ldots,\phi_{k+1},\phi'_{k+2},\ldots,\phi'_d) 
\end{align*}
}
which shows the linearity. If $\psi \geq 0$ and $\phi_1,\ldots,\phi_{k+1}$ are in $\PAPC(\tX,\gamma)$, then $\int \psi \,\PMA (\phi_1,\ldots,\phi_{k+1},\phi'_{k+2},\ldots,\phi'_d)$ is non-negative by Proposition \ref{prop/def} and Definition~\ref{def:PAPC}. Since $\int \psi \,\PMA (\phi_1,\ldots,\phi_{k+1},\phi'_{k+2},\ldots,\phi'_d)$ is continuous along decreasing nets as a function of $(\phi_1,\ldots,\phi_{k+1})$, and since for any function in $\PCreg(\tX, \gamma)$ there is a decreasing sequence of functions in $\PAPC(\tX,\gamma)$ uniform converging to it, it follows that the functional is positive on $\PA_b(\tX)$.
\\By Proposition \ref{prop:bounded-dense}, the bounded $\PA$ functions are dense in $\mathcal{C}^0(\overline{\tX}_\Pi)$, hence we may define the positive measure $\PMA (\phi_1,\ldots,\phi_{k+1},\phi'_{k+2},\ldots,\phi'_d)$ on $\overline{\tX}_\Pi$; it has total mass $\deg(\gamma)$ and is continuous along decreasing sequences as a function of $(\phi_1,\ldots,\phi_{k+1})$.

It remains to show that the map
$$
(\psi,\phi_1,\ldots,\phi_{k+1}) \to \int_{\overline{\tX}_\Pi}  (\psi-\gamma) \,\PMA (\phi_1,\ldots,\phi_{k+1},\phi'_{k+2},\ldots,\phi'_d)
$$
is continuous along decreasing sequences in $\PCreg(\tX,\gamma)$. Let $(\psi^j)_j$ and $(\phi^j_i)_j$, for $i=1,\ldots,k+1$, be decreasing sequences of functions in $\PCreg(\tX,\gamma)$ converging respectively to $\psi$ and $\phi_i$. For simplicity, we write
$$
\mu^j \coloneqq \PMA(\phi^j_1,\ldots,\phi^j_{k+1},\phi'_{k+2},\ldots,\phi'_d),
$$
which we know that converges weakly to $\mu \coloneqq \PMA(\phi_1,\ldots,\phi_{k+1},\phi'_{k+2},\ldots,\phi'_d)$. To prove the required continuity, we show the following two inequalities 
$$
\limsup_j \int_{\overline{\tX}_\Pi}  (\psi_j-\gamma) \mu^j \leq \int_{\overline{\tX}_\Pi}  (\psi-\gamma)\mu \leq \liminf_j \int_{\overline{\tX}_\Pi}  (\psi_j-\gamma) \mu^j.
$$
On one side, for any $j$, we have 
$$
\int_{\overline{\tX}_\Pi}  (\psi_j - \gamma) \mu^j \geq \int_{\overline{\tX}_\Pi}  (\psi - \gamma) \mu^j,
$$
hence, passing to the $\liminf$, we get
$$
\liminf_{j} \int_{\overline{\tX}_\Pi}  (\psi_j - \gamma) \mu^j \geq \liminf_j \int_{\overline{\tX}_\Pi}  (\psi - \gamma) \mu^j = \int_{\overline{\tX}_\Pi}  (\psi - \gamma) \mu,
$$
where the last equality holds by the weak convergence of measures. 
For the other inequality, we extend $\psi^j - \gamma$ to $\overline{\tX}_\Pi$ using Lemma \ref{lem:ext} and similarly for $\psi-\gamma$, we still use the same notation for the extensions. 
By Lemma \ref{lem:decreasing cv uniform cv}, the decreasing convergence of $(\psi^j - \gamma)_j$ to $\psi - \gamma$ on $\overline{\tX}_\Pi$ implies its uniform convergence, thus for any $\varepsilon >0$ there exists $j_\varepsilon$ such that we have
$$
|(\psi^j - \gamma) - (\psi-\gamma)| < \varepsilon
$$
for any $j \geq j_\varepsilon$. This implies that
$$
\limsup_j \int_{\overline{\tX}_\Pi}  (\psi^j-\gamma)\mu^j 
\leq \limsup_j \int_{\overline{\tX}_\Pi}  (\psi-\gamma)\mu^j + \varepsilon \limsup_j \int_{\overline{\tX}_\Pi}  \mu^j 
= \int_{\overline{\tX}_\Pi}  (\psi-\gamma)\mu + \varepsilon \int_{\overline{\tX}_\Pi}  \mu
$$
where the last equality holds by the weak convergence of measures. 
When $\varepsilon$ tends to $0$, we obtain 
$$
\limsup_j \int_{\overline{\tX}_\Pi}  (\psi^j-\gamma)\mu^j 
\leq  \int_{\overline{\tX}_\Pi}  (\psi-\gamma)\mu
$$
which concludes the proof of the assertion for $k+1$. 

By induction, the assertion now holds for any $k$. In particular, for $k=d$, it shows that there exists an extension of $\PMA$ to functions in $\PCreg(\tX, \gamma)$. By construction and by density of bounded piecewise affine functions in $\mathcal{C}^0(\overline{\tX}_\Pi)$, property $(5)$ (integration by parts formula) automatically holds true.

The continuity of the operator $\PMA$ from $\PCreg(\tX, \gamma)^d$ to the space of Radon measures follows from the continuity along decreasing sequences, using the fact that uniform convergence of a sequence of functions on a compact space can be made decreasing up to extraction and adding a sequence of constants. The uniqueness of the extension follows from its continuity and from Definition \ref{def:PC-reg}, stating that every function in $\PCreg(\tX, \gamma)$ is a decreasing uniform limit of functions in $\PAPC(\tX,\gamma)$.

Finally, properties $(1)$-$(4)$ hold for functions $(\phi_1,\ldots,\phi_d)$, in $\PAPC(\tX,\gamma)$, for constants $t,c_i \in \Q$, and for piecewise affine functions $\psi$, by Proposition \ref{prop/def} and Lemma \ref{lem:integral-formulae}. As any function in $\PCreg(\tX,\gamma)$ and any real constant is respectively the decreasing limit of functions in $\PAPC(\tX,\gamma)$ and of rational constants, and by the density of bounded piecewise affine functions in $\mathcal{C}^0(\overline{\tX}_{\Pi})$, the properties $(1)$-$(4)$ hold true in general.
\end{proof}

For a single function $\phi \in \PCreg(\tX,\gamma)$, we set $\MA_{\poly}(\phi) = \PMA(\phi, \dots, \phi)$.

\subsection{Properties of $\PMA$}\label{sec:propr MA} 

\begin{defi}\label{def:posinf}
Let $\Pi$ be a polyhedral structure on $\tX$ and $\gamma \in \PA(\Pi)$. We say that $\gamma$ is \emph{positive at infinity} if for any face $\sigma \in \Pi$ and non-zero $v \in \rec(\sigma)$, the slope of $\gamma_{\sigma}$ in the direction $v$ is strictly positive.
\end{defi}
This condition is equivalent to $\lvert \gamma \rvert$ going to $+ \infty$ at infinity on $\tX$, so that it is independent of the choice of $\Pi$ where $\gamma$ is defined.
\begin{prop} \label{prop:suppMA}
    For any $\phi \in \PCreg(\tX,\gamma)$ and $\psi \in \PSH(\tX,\gamma)$, the function $(\psi-\gamma)$ is integrable with respect to $\MA_{\poly}(\phi)$. 
    \\As a consequence, if $\gamma$ is positive at infinity, then $\MA_{\poly}(\phi)$ does not charge mass on the boundary of $\overline{\tX}_\Pi$.
\end{prop}
\begin{proof}
Set $\mu \coloneqq \MA_{\poly}(\phi)$. We normalize the functions by a constant so that 
$$\sup_{\tX} (\phi-\gamma)=\sup_{\tX} (\psi-\gamma) =0.$$ First assume that $\psi$ lies in $\PAPC(\tX,\gamma)$. We claim that $ \int  (\gamma-\psi) \mu \le M$, where $M$ does not depend on $\psi$. Indeed, integration by parts (item (5) in Theorem \ref{thm:PMA PCreg}) yields 
$$ \int_{\overline{\tX}_\Pi} (\gamma-\psi) \MA_{\poly}(\phi,\ldots,\phi)=  \int_{\overline{\tX}_\Pi} (\gamma-\phi) \MA_{\poly}(\psi, \phi,\ldots,\phi) + \int_{\overline{\tX}_\Pi} (\phi-\psi) \MA_{\poly}(\gamma,\phi,\ldots,\phi).$$ 
The first term is bounded by $ \sup \lvert \phi-\gamma \rvert \cdot \deg(\gamma)$, while we handle the second one by writing $(\phi-\psi) = (\phi-\gamma)+(\gamma-\psi)$. The integral of $(\phi-\gamma)$ is again bounded by  $\sup \lvert \phi-\gamma \rvert \cdot \deg(\gamma)$, while to the integral of $(\gamma-\psi)$ we apply again the integration by parts. Iterating this, the final outcome is
$$ \int_{\overline{\tX}_\Pi} (\gamma-\psi) \PMA(\phi) \leq 2d\,\sup \lvert \phi-\gamma \rvert \cdot \deg(\gamma)+\int_{\overline{\tX}_\Pi} (\gamma-\psi) \PMA(\gamma).
$$
Now $\MA_{\poly}(\gamma)$ is supported on a compact set $K$, and by Theorem \ref{thm:comp} there exists a constant $C=C(K)$ such that any $\psi \in \PAPC(\tX,\gamma)$ satisfies $\sup_K \lvert \psi-\gamma \rvert \le C$. This yields $$\int_{\overline{\tX}_\Pi} (\gamma-\psi) \MA_{\poly}(\gamma) \le C \deg(\gamma)$$ and proves the claim. In the general case, let $(\psi_j)_j$ be a sequence in $\PAPC(\tX,\gamma)$ decreasing pointwise to $\psi$, and satisfying $\sup_{\tX} (\psi_j - \gamma) = 0$. By monotone convergence,
$\int_{\overline{\tX}_\Pi} (\gamma-\psi_j) \mu$ increases to $\int_{\overline{\tX}_\Pi} (\gamma-\psi) \mu$, so that $\int_{\overline{\tX}_\Pi} (\gamma-\psi) \mu \le M$ as well, and is thus finite.

The second claim comes from the fact that the zero function lies in $\PSH(\tX,\gamma)$ when $\gamma$ is positive at infinity. Indeed, $0 = \lim_j \max \{ \gamma-j, 0 \}$ pointwise, and $\max \{ \gamma-j, 0 \} \in \PAPC(\tX,\gamma)$ since $\gamma$ goes to $+ \infty$ at infinity. This implies that $- \gamma$ is $\mu$-integrable, whence the locus $\overline{\tX}_\Pi \setminus \tX = \{ \gamma = + \infty \}$ has measure zero. 
\end{proof}

\begin{prop}[Locality] \label{prop:locreg}
    For any $\phi, \psi \in \PCreg(\tX,\gamma)$, we have
    $$\mathbf{1}_{\{\phi >\psi\}} \MA_{\poly}(\max( \phi, \psi)) = \mathbf{1}_{\{\phi >\psi\}} \MA_{\poly}(\phi).$$
\end{prop}
\begin{proof}
We argue in three steps, increasing the generality of the statement each time. We denote by $\overline{\tX}$ any compactification of $\tX$ as in Definition \ref{def:compactification}.

\textbf{Step 1:} assume that $\phi$, $\psi \in \PAPC(\tX,\gamma)$. Let $\Pi$ be a polyhedral structure on $\tX$ on which both $\phi$ and $\psi$ are defined, and pick a vertex $v \in \Pi^{(0)}$ with $\phi(v) >\psi(v)$. This implies that $\phi > \psi$ on a small neigbourhood $U$ of $v$. Let $\Pi'$ be a refinement of $\Pi$ such that $U \supseteq \Star_{\Pi'}(v)$. Since now $\phi > \psi$ on $\Star_{\Pi'}(v)$, the value $\MA_{\poly}( \max(\phi, \psi))(v)=(\max(\phi, \psi)) \cdot [\tX]^d \cdot v $ computed as an intersection product as in Definition \ref{def:intproduct}, using the subdivision $\Pi'$, clearly depends only on the restriction of $\max(\phi, \psi)$ to $\Star_{\Pi'}(v)$, and thus agrees with $\MA_{\poly}(\phi)(v)$.

\textbf{Step 2:} assume that $\phi \in \PAPC(\tX,\gamma)$ and $\psi \in \PCreg(\tX,\gamma)$. Let $\eps>0$, and let $\psi' \in \PAPC(\tX,\gamma)$ such that $\psi' \le \psi \le \psi' + \eps$ on $\tX$. 
Let $g$ be a continuous, compactly supported function on the open set $\{ \phi > \psi \}$. Write
\begin{align*}
\big\lvert \int_{\overline{\tX}} g \MA_{\poly}(\max(\phi, \psi)) - & \int_{\overline{\tX}} g \MA_{\poly}(\phi) \big\lvert \\ 
& \le \big\lvert \int_{\overline{\tX}} g \MA_{\poly}(\max(\phi, \psi))- \int_{\overline{\tX}} g \MA_{\poly}(\max(\phi, \psi')) \big\rvert \\
& \quad + \big\lvert \int_{\overline{\tX}} g \MA_{\poly}(\max(\phi, \psi')) - \int_{\overline{\tX}} g \MA_{\poly}(\phi) \big\lvert.
\end{align*}
Since $\psi' \le \psi$, we have that $\psi' < \phi$ on the support of $g$, so that Step 1 implies that the second summand of the right hand side vanishes; the first summand goes to zero when $\eps \rightarrow 0$ by continuity of $\PMA$ with respect to uniform convergence. This shows that $\MA_{\poly} (\max(\phi, \psi))$ agrees with $\MA_{\poly}(\phi)$ on $\{ \phi >\psi \}$, and concludes Step 2.

\textbf{Step 3:} let $\phi, \psi \in \PCreg(\tX,\gamma)$. Use $\phi' \in \PAPC(\tX,\gamma)$ such that $\phi' \le \phi \le \phi'+ \eps$ and argue exactly as in Step 2 to conclude.
\end{proof}

\begin{cor} \label{cor:locality}
Let $\phi, \psi \in \PCreg(\tX,\gamma)$ such that $\phi = \psi$ on an open subset $U \subset \tX$. Then the measures $\MA_{\poly}(\phi)$ and $\MA_{\poly}(\psi)$ agree on $U$.
\end{cor}
\begin{proof}
The proof follows \cite[Corollary 5.2]{BoucksomFavreJonsson2015}. Let $\eps>0$. Applying the previous proposition to $\phi+\eps$ and $\psi$ yields $\MA_{\poly}( \max(\phi + \eps, \psi))= \MA_{\poly}(\phi)$ on $U$, and letting $\eps \rightarrow 0$ yields $\MA_{\poly}( \max(\phi , \psi))= \MA_{\poly}(\phi)$ on $U$. Exchanging the roles of $\phi$ and $\psi$ concludes.
\end{proof}

\begin{cor}[Comparison principle] \label{cor:comparison}
For any $\phi, \psi \in \PCreg(\tX, \gamma)$, we have
$$\int_{\{\phi < \psi \}} \PMA(\psi) \le \int_{\{\phi < \psi \}} \PMA(\phi).$$
\end{cor}
\begin{proof}
The proof follows \cite[Corollary 5.3]{BoucksomFavreJonsson2015}. Let $\overline{\tX}$ be any compactification of $\tX$ as in Definition \ref{def:compactification}. Again, by the locality property in Proposition \ref{prop:locreg} we have that for all $\varepsilon > 0$ 
\begin{align*}
\deg(\gamma) &= \int_{\overline{\tX}} \MA_{\poly}\left(\max(\phi, \psi - \varepsilon ) \right) \\ &\geq  \int_{\{\phi < \psi - \varepsilon \}} \MA_{\poly}\left(\max (\phi, \psi - \varepsilon )\right) + \int_{\{\phi > \psi - \varepsilon \}} \MA_{\poly}\left(\max (\phi, \psi - \varepsilon )\right) \\
& = \int_{\{\phi < \psi - \varepsilon\}} \MA_{\poly}(\psi - \varepsilon ) + \int_{\{\phi > \psi - \varepsilon\}} \MA_{\poly}(\phi )\\
& = \int_{\{\phi < \psi - \varepsilon\}} \MA_{\poly}(\psi ) + \deg(\gamma) - \int_{\{\phi \leq \psi - \varepsilon\}} \MA_{\poly}(\phi).
\end{align*}
The result then follows by letting $\varepsilon \to 0$.
\end{proof}

We conclude this section with the following result, comparing the polyhedral Monge-Ampère measure with the usual real Monge-Ampère measure on top-dimensional faces:
\begin{prop} \label{prop:comp MA}
    Let $\phi \in \PCreg(\tX,\gamma)$ and $\sigma$ a top-dimensional face of $\Pi$. Then the equality of measures
    $$ \mathbf{1}_{\relint(\sigma)} \MA_{\poly}(\phi) = d! [\tX](\sigma) \; \RMA(\phi_{| \relint(\sigma)})$$
    holds.
\end{prop}
\begin{proof}
This follows from Proposition \ref{prop:MA-comp} applied to a sequence $(\phi_j)_j$ in $\PAPC(\gamma)$ converging uniformly to $\phi$, together with the weak continuity of the classical real Monge--Ampère measure with respect to uniform convergence \cite[Proposition 2.6]{Figalli}.
\end{proof}

\subsection{Polyhedral Monge--Ampère for $\PSH(\tX, \gamma)$}
We further extend the Monge--Ampère operator to $\PSH(\tX, \gamma)$. We assume $\gamma$ to be positive at infinity.

For any $\phi \in \PSH(\tX, \gamma)$ and any real number $t$, set $\phi^{\langle t \rangle} \coloneqq \max(\phi, \gamma-t)$. It follows from Proposition/Definition \ref{def:PC-reg} that $\phi^{\langle t \rangle}$ is an element of $\PCreg(\tX, \gamma)$.

\begin{defi}\label{def:MA-PSH}
The \emph{polyhedral Monge--Ampère measure of} $\phi \in \PSH(\tX, \gamma)$ is defined as follows. Let $g$ be a continuous, non-negative bounded function on $\tX$. Then we set 
$$\int_{\tX} g \MA_{\poly}(\phi) \coloneqq \lim_{t \to +\infty}\int_{\tX} \mathbf{1}_{\{\phi>\gamma-t\}}\, g \MA_{\poly}\left(\phi^{\langle t \rangle}\right),$$
and extend this functional to continuous bounded functions by linearity. This defines $\PMA(\phi)$ as a positive Radon measure on $\tX$.
\end{defi}
The above limit exists as the right hand side is increasing with respect to $t$: indeed, for any $s > t$, we have $\phi^{\langle t \rangle}=\max(\phi^{\langle s \rangle},\gamma-t)$ and $\{\phi > \gamma-t\}=\{\phi^{\langle s \rangle} > \gamma-t\}$, hence
\begin{alignat*}{2}
\int_{\tX} \mathbf{1}_{\{\phi>\gamma-t\}} g \MA_{\poly}(\phi^{\langle s \rangle}) 
& = \int_{\tX} \mathbf{1}_{\{\phi^{\langle s \rangle}>\gamma-t\}} g \MA_{\poly}(\phi^{\langle s \rangle}) &&\\
& = \int_{\tX} \mathbf{1}_{\{\phi^{\langle s \rangle}>\gamma-t\}} g \MA_{\poly}(\phi^{\langle t \rangle}) && \quad \text{by Proposition \ref{prop:locreg}}\\
& = \int_{\tX} \mathbf{1}_{\{\phi>\gamma-t\}} g \MA_{\poly}(\phi^{\langle t \rangle}),   &&
\end{alignat*}
which yields 
\[
\int_{\tX} \mathbf{1}_{\{\phi>\gamma-s\}} g \MA_{\poly}(\phi^{\langle s \rangle}) \ge \int_{\tX} \mathbf{1}_{\{\phi>\gamma-t\}} g \MA_{\poly}(\phi^{\langle t \rangle}).
\]
Moreover, this quantity is uniformly bounded by $\lVert g \rVert_{L^{\infty}} \deg(\gamma)$, and hence the limit is finite.
\begin{rem} \label{rem:MApsh}
For all $t \in \R$, the equality
$$\mathbf{1}_{\{ \phi > \gamma-t\} } \PMA(\phi^{\langle t \rangle}) = \mathbf{1}_{ \{\phi > \gamma-t\} } \PMA(\phi)$$
holds. Indeed, for any $g$ a continuous bounded function on $\tX$, it holds that 
$$\int_{ \{ \phi > \gamma -t \}} g \,\PMA(\phi) 
= \lim_{s \rightarrow \infty} \int_{\{ \phi > \gamma -t \}} g \,\PMA(\phi^{\langle s \rangle}) 
= \int_{ \{ \phi > \gamma -t \}} g \,\PMA(\phi^{\langle t \rangle}),$$
where the last equality follows from locality (Proposition \ref{prop:locreg}) applied $\phi^{\langle s \rangle}$ and $\phi^{\langle t \rangle}$ for $s \ge t$.
\end{rem}

\section{Energy and capacity}

Let $\tX$ be a balanced polyhedral space of dimension $d$ and let $\gamma$ be a PAPC function 
that is positive at infinity.
The goal of this section is to introduce the energy function $E$ on $\PSH(\tX, \gamma)$ (Definition \ref{defi:energy}), as well as the class of finite energy functions (Definition \ref{def:psh-finite-energy}). We then show various properties of the Monge--Ampère measures of functions with finite energy, and introduce a notion of capacity (Definition \ref{defi:capa}) that will be used in Section \ref{sec:variational app}.
\subsection{Energy for $\PCreg$}

\begin{defi} \label{defi:energy}
    Let $\phi \in \PCreg(\tX,\gamma)$. We define its energy as
    \[
    E(\phi) \coloneqq \frac{1}{d+1} \sum_{j=0}^d \int_{\tX}(\phi - \gamma)\MA_{\poly}(\underbrace{\phi, \dots, \phi}_{j\text{ times}},\underbrace{\gamma,\ldots, \gamma}_{(d-j) \text{ times}}).
    \]
\end{defi}

\begin{lemma} \label{lem:der energy}
    Let $\psi, \phi \in \PAPC(\tX,\gamma)$ and set $\phi_t \coloneqq (1-t) \phi+t \psi$ for $t \in [0,1]$. Then
    \begin{equation}\label{eq:der-energy}
    \frac{d}{dt}  E(\phi_t) = \int_{\tX} (\psi-\phi) \PMA(\phi_t).
    \end{equation}
\end{lemma}
\begin{proof}
Set $E(t)\coloneqq E(\phi_t)$; it is a polynomial of degree at most $(d+1)$ by multilinearity of $\PMA$. We will compute its derivative $E'(t)$. We first have
$$\frac{d}{dt} \MA_{\poly}(\underbrace{\phi_t, \dots, \phi_t}_{j\text{ times}},\underbrace{\gamma,\ldots, \gamma}_{(d-j) \text{ times}})=j \MA_{\poly}(\psi- \phi, \underbrace{\phi_t, \dots, \phi_t}_{(j-1)\text{ times}},\underbrace{\gamma,\ldots, \gamma}_{(d-j) \text{ times}}) .$$
As a consequence, we get
\begin{align*}
    E'(t) 
    & = \frac{1}{d+1} \sum_{j=0}^d \int_{\tX}(\psi - \phi) \MA_{\poly}(\underbrace{\phi_t, \dots, \phi_t}_{j\text{ times}},\underbrace{\gamma,\ldots, \gamma}_{(d-j) \text{ times}})\\
    & +\sum_{j=1}^{d} \frac{j}{d+1} \int_{\tX} (\phi_t -\gamma) \MA_{\poly}(\psi-\phi, \underbrace{\phi_t, \dots, \phi_t}_{j-1\text{ times}},\underbrace{\gamma,\ldots, \gamma}_{(d-j) \text{ times}}).
\end{align*}
Splitting the integral over the first argument in $\PMA$ and using Lemma \ref{lem:integral-formulae} twice, we have
\begin{align*}
\int_{\tX} (\phi_t -\gamma)   \MA_{\poly}(\psi-\phi, \underbrace{\phi_t, \dots, \phi_t}_{j-1\text{ times}},\underbrace{\gamma,\ldots, \gamma}_{(d-j) \text{ times}}) 
& = \int_{\tX}(\psi - \phi) \MA_{\poly}(\underbrace{\phi_t, \dots, \phi_t}_{j\text{ times}},\underbrace{\gamma,\ldots, \gamma}_{(d-j) \text{ times}}) \\
& - \int_{\tX} (\psi-\phi) \MA_{\poly}(\underbrace{\phi_t, \dots, \phi_t}_{j-1\text{ times}},\underbrace{\gamma,\ldots, \gamma}_{(d-j+1) \text{ times}}),
\end{align*}
and plugging this in the formula above for $E'(t)$ gives the result.
\end{proof}

\begin{prop}  \label{prop:energy2}
    Let $\phi, \psi \in \PCreg(\tX,\gamma)$. We have
$$E(\phi)-E(\psi) = \frac{1}{d+1} \sum_{j=0}^{d} \int_{\tX} (\phi-\psi) \MA_{\poly}(\underbrace{\phi, \dots, \phi}_{j\text{ times}},\underbrace{\psi,\ldots, \psi}_{(d-j) \text{ times}}).$$
\end{prop}
\begin{proof}
It follows from the proof of Theorem \ref{thm:PMA PCreg} that the map
$$\phi \mapsto \int_{\tX}(\phi - \gamma)\MA_{\poly}(\phi, \dots, \phi,\gamma,\ldots, \gamma)$$ 
is continuous along decreasing sequences, and similary with $\psi$ instead of $\gamma$. It follows that both sides of the equality are continuous with respect to decreasing convergence of $\PCreg$ functions, and we may thus assume that $\phi, \psi \in \PAPC(\tX,\gamma)$. For any $\phi, \psi \in \PAPC(\tX,\gamma)$, set
$$E_{\psi} (\phi) = \frac{1}{d+1} \sum_{j=0}^d \int_{\tX}(\phi - \psi)\MA_{\poly}(\underbrace{\phi, \dots, \phi}_{j\text{ times}},\underbrace{\psi,\ldots, \psi}_{(d-j) \text{ times}}).$$
Applying Lemma \ref{lem:der energy} to $\gamma = \psi$, we see that $E_{\gamma}(\phi_t)$ and $E_{\psi}(\phi_t)$ have the same derivative. The difference is thus constant in $t$, which grants
$$E_{\psi}(\phi) -E_\gamma(\phi) = E_{\psi}(\psi)-E_{\gamma}(\psi).$$
Since $E_{\psi}(\psi) =0$, this yields
$E_{\psi}(\phi) = E_{\gamma}(\phi)-E_{\gamma}(\psi),$
as was to be shown.
\end{proof}

\begin{prop} \label{prop:energy}
The energy functional $E \colon \PCreg(\tX, \gamma) \to \R$ is non-decreasing, concave, continuous for the topology of uniform convergence on $\PCreg(\tX, \gamma)$, and satisfies $E(\phi +c) = E(\phi) +c \deg(\gamma)$ for $c \in \R$. 
\end{prop}
\begin{proof}
The equality $E(\phi +c) = E(\phi) +c \deg(\gamma)$ follows from the invariance of $\PMA$ with respect to additive constants (see (2) in Theorem \ref{thm:PMA PCreg}) and that the total mass of $\PMA(\cdot)$ is $\deg(\gamma)$.

We prove that the energy functional is non-decreasing and concave on $\PAPC(\tX,\gamma)$.
Monotonicity follows from Lemma \ref{lem:der energy}, and (4) in Theorem \ref{thm:PMA PCreg}. For concavity, let $\phi$ and $\psi$ in $\PAPC(\tX,\gamma)$ and let $\phi_t = (1-t)\phi + t\psi$, we have to show that $t \mapsto E(\phi_t)$ is a concave function. Differentiating \eqref{eq:der-energy}, we obtain
\[
\frac{d^2}{dt^2}E(\phi_t) = d \int_{\tX} (\psi - \phi)\MA_{\poly}(\psi - \phi, \phi_t, \dotsc, \phi_t).
\]
As a result, the concavity of $E$ follows from the semi-negativity statement in Lemma \ref{lem:nega}. 

For the continuity, let $(\phi_n)_n$ be a sequence in $\PCreg(\tX, \gamma)$ converging to $\phi$ uniformly. Recall that for any $j$ the measure 
$$\mu_j(n) \coloneqq \PMA(\underbrace{\phi_n,\ldots, \phi_n}_{j\text{ times}},\underbrace{\gamma,\ldots, \gamma}_{(d-j) \text{ times}})$$ converges weakly to $\mu_j \coloneqq  \PMA(\underbrace{\phi, \ldots, \phi}_{j\text{ times}},\underbrace{\gamma,\ldots, \gamma}_{(d-j) \text{ times}})$. We have that
\begin{align*}
\bigg\lvert \int_{\tX} (\phi-\gamma) \mu_j -\int_{\tX} (\phi_n-\gamma) \mu_j(n) \bigg\rvert 
& = \bigg\lvert \int_{\tX} (\phi-\gamma) (\mu_j-\mu_j(n)) + \int_{\tX} (\phi-\phi_n) \mu_j(n) \bigg\rvert \\
& \le \lvert \int_{\tX} (\phi-\gamma) (\mu_j-\mu_j(n))  \rvert+ \deg(\gamma) \sup_{\tX} \lvert \phi-\phi_n \rvert
\end{align*}
and the left hand side goes to zero since $(\phi-\gamma)$ is bounded. Summing up over $j$ now yields the continuity of the energy functional. By continuity, we conclude that the energy functional is non-decreasing and concave also on $\PCreg(\tX,\gamma)$.
\end{proof} 

\subsection{Energy for $\PSH$}

\begin{defi} \label{defi:energy2}
    Let $\phi \in \PSH(\tX,\gamma)$. We define:
    $$E(\phi) \coloneqq \inf \left\{ E(\psi) \; | \; \psi \in \PAPC(\tX,\gamma), \psi \ge \phi \right\}. $$
\end{defi}

\begin{lemma} \label{lem:energy}
    The definition above is consistent with Definition \ref{defi:energy} for functions in $\PCreg(\tX, \gamma)$. Moreover, the functional $E \colon \PSH(\tX,\gamma) \to \R \cup \{-\infty\}$ satisfies $E(\phi+c) = E(\phi) +c \deg(\gamma)$ for $c \in \R$, and is non-decreasing and continuous along decreasing sequences.
\end{lemma}
\begin{proof}
First observe that the two definitions are clearly equivalent for $\phi \in \PAPC(\tX, \gamma)$ by monotonicity. Write $E$ for the energy functional as defined in Definition \ref{defi:energy}, and $\tilde{E}$ for the one in Definition \ref{defi:energy2}.

If $\phi \in \PCreg(\tX,\gamma)$ and $(\phi_j)_j$ is a sequence in $\PAPC(\tX, \gamma)$ decreasing to $\phi$, we claim that $E(\phi)$ is the limit of the decreasing sequence $E(\phi_j)$. Indeed, the $(\phi_j)_j$ converge uniformly to $\phi$ by Lemma \ref{lem:decreasing cv uniform cv}, and $E(\cdot)$ is continuous for the topology of uniform convergence by Proposition \ref{prop:energy}. This yields $E(\phi) \ge \tilde{E}(\phi)$. Conversely, if $\psi \in \PAPC(\tX,\gamma)$ such that $\psi \ge \phi$, then by Proposition \ref{prop:energy} $E(\psi) \ge E(\phi)$ and taking the inf over $\psi$ yields $\tilde{E}(\phi) \ge E(\phi)$.

Translation-invariance and monotonicity of $E$ are clear from the definition, so that we prove continuity along decreasing sequences. Let $(\phi_j)_j$ be a sequence in $\PSH(\tX,\gamma)$ decreasing pointwise to $\phi$, and set $\ell = \lim_j E(\phi_j)$. By monotonicity we have $E(\phi) \le \ell$. 
Let $\eps >0$. By Definition  \ref{defi:energy2}, there exists $\psi \in \PAPC(\tX,\gamma)$ such that $\psi \ge \phi$ and $E(\psi) \le E(\phi) + \eps$. Moreover, by Theorem \ref{thm:comp}, the set $U_{\eps} = \{\eta \in \PSH(\tX,\gamma) | \sup_{\tX} (\eta-\psi) < \eps \}$ is an open neighbourhood of $\psi$ in $\PSH(\tX,\gamma)$ for the topology of pointwise convergence, so that $\phi_j \in U_{\eps}$ for $j \gg 1$. It follows that $\phi_j \le \psi + \eps$ for $j \gg 1$, hence $E(\phi_j) \le E(\psi) + \eps \deg(\gamma)$ by monotonicity and translation-invariance. We infer $E(\phi_j) \le E(\phi) +(1+\deg(\gamma)) \eps$ for $j \gg 1$, whence $\ell \le E(\phi) +(1+\deg(\gamma)) \eps$, and letting $\eps \rightarrow 0$ concludes.
\end{proof}

\begin{defi}\label{def:psh-finite-energy} 
The subspace of $\PSH(\tX,\gamma)$ of functions $\phi$ such that $E(\phi) > - \infty$ is denoted by $\mathcal{E}^1(\tX, \gamma)$, and we say that such functions have \emph{finite energy}.
\end{defi}

\begin{lemma}\label{lem:total-mass-finite-energy}
    For any $\phi \in \mathcal{E}^1(\tX,\gamma)$ the measure $\MA_{\poly}(\phi)$ (Definition \ref{def:MA-PSH}) has total mass $\deg(\gamma)$.
\end{lemma}
\begin{proof}
For $t\in \R_{\geq 0}$, write $\phi^{<t>} \coloneqq \max(\phi, \gamma-t),$ and $\mu_t = \MA_{\poly}(\phi^{<t>})$. Since $\phi^{\langle t/2 \rangle} \ge \phi^{\langle t \rangle}$, Proposition \ref{prop:energy2} together with positivity of the Monge-Ampère measure gives:
$$E(\phi^{<t/2>})-E(\phi^{<t>}) \ge \frac{1}{d+1} \int_{\tX} (\phi^{<t/2>}-\phi^{<t>}) \mu_t.$$
 This yields
\begin{align*}
    E(\phi^{<t/2>})-E(\phi^{<t>}) 
    & \ge \frac{1}{d+1} \int_{0}^{t/2} \mu_t \{\phi^{<t/2>}-\phi^{<t>} \ge s \} ds \\
    & \ge \frac{1}{d+1} \int_{0}^{t/2} \mu_t \{\phi^{<t/2>}-\phi^{<t>} \ge t/2 \} ds \\
    & = \frac{t}{2(d+1)} \mu_t \{ \phi \le \gamma -t \}.
\end{align*}
Since $\phi$ has finite energy, the left-hand side goes to zero when $t \to + \infty$, so that $\mu_t \{ \phi \le \gamma-t \} = o(t^{-1}).$
We obtain that
$$
\deg(\gamma)=\int_{\tX} \mu_t =\mu_t\{ \phi \le \gamma-t \} + \mu_t \{ \phi > \gamma-t \} = o(t^{-1}) +\mu_t \{ \phi > \gamma-t \},
$$ hence taking the limit for $t \to +\infty$, by Definition \ref{def:MA-PSH} applied to $g \equiv 1$, we get that $\MA_{\poly}(\phi)$ has full mass $\deg(\gamma)$.
\end{proof}
\begin{prop}\label{prop:MAdecr}
Let $(\phi_j)_j$ be a sequence in $\PSH(\tX,\gamma)$, decreasing pointwise to $\phi \in \mathcal{E}^1(\tX,\gamma)$. Then $\phi_j \in \mathcal{E}^1(\tX,\gamma)$ for all $j$, and for any $v \in \mathcal{C}^0_b(\tX)$:
$$\lim_j \int_{\tX} v\, \PMA(\phi_j) = \int_{\tX} v \,\PMA(\phi).$$
\end{prop}
\begin{proof}
We argue as in \cite[Proposition 6.9]{BoucksomFavreJonsson2015}. Since $\phi_j \ge \phi$, Lemma \ref{lem:energy} grants $E(\phi_j) \ge E(\phi) >-\infty$, so that $\phi_j \in \mathcal{E}^1(\tX,\gamma)$ for all $j$.
Write $\phi^{\langle t \rangle} = \max ( \phi, \gamma -t )$ and similarly for $\phi_j^{\langle t \rangle}$. By the proof of Lemma \ref{lem:total-mass-finite-energy}, the estimate $$\MA_{\poly}(\phi_j^{<t>})\{ \phi_j \le \gamma-t \} = o(t^{-1})$$ holds, which implies that for all $j$ and any $v \in \mathcal{C}^0_b(\tX)$, 
$$\int_{\tX} v\, \PMA(\phi_j) = \lim_t \int_{\tX} v\, \PMA(\phi_j^{\langle t \rangle}),$$ and similarly for $\phi$. We claim that this convergence is uniform with respect to $j$. Granting the claim for now, the result follows from the convergence 
$$\lim_j \int_{\tX} v\, \PMA(\phi_j^{\langle t \rangle}) = \int_{\tX} v\, \PMA(\phi^{\langle t \rangle})$$ for all $t$, since $\phi_j^{\langle t \rangle}$ converges uniformly to $\phi^{\langle t \rangle}$ by Dini's lemma and Lemma \ref{lem:decreasing cv uniform cv}.

We now prove the claim. Up to shifting by constants, we may assume $\sup_{\tX} \lvert v \rvert =1$ and $\phi \le \gamma-1$, so that $E(\phi) \le 0$ by monotonicity of the energy, and moreover $\sup_{\tX} (\phi_j-\gamma) \le 0$ for all $j \gg 1$ by Theorem \ref{thm:comp}.
Fix $j \gg 1$ and set $\mu_t(j) \coloneqq \PMA(\phi_j^{\langle t \rangle})$. For $s \ge t$, $\mu_s(j)$ and $\mu_t(j)$ agree on $\{ \phi_j \ge \gamma -t \}$, hence we have
\begin{align*}
\big\lvert \int_{\tX} v \mu_t(j) - \int_{\tX} v \mu_s(j) \big\rvert 
& \le (\mu_t(j)+\mu_s(j))(\{ \phi_j < \gamma-t\})  \\
& \le \int \frac{\gamma-\phi_j^{\langle t \rangle}}{t} \mu_t(j)+ \int \frac{\gamma-\phi_j^{\langle s \rangle}}{t} \mu_s(j) \\
& \le \frac{d+1}{t} (\lvert E(\phi_j^{\langle t \rangle}) \rvert +\lvert E(\phi_j^{\langle s \rangle}) \rvert) \\
& \le \frac{2(d+1)}{t} \lvert E(\phi) \rvert,
\end{align*}
since $E(\phi) \le E(\phi_j) \le E(\gamma) = 0$ by monotonicity.  Letting $s$ go to infinity now shows that the weak convergence of $\mu_t(j)$ to $\PMA(\phi_j)$ is uniform with respect to $j$, and concludes the proof.
\end{proof}

\begin{prop}\label{prop:comparisonE} 
    For any $\phi, \psi \in \mathcal{E}^1(\tX,\gamma)$, the following hold:
    \begin{itemize}
        \item Locality: $\mathbf{1}_{\{\phi > \psi\}} \MA_{\poly} (\max (\phi, \psi)) =  \mathbf{1}_{\{\phi > \psi\}} \MA_{\poly}(\phi)$,
        \item Comparison principle: $\int_{\phi < \psi} \MA_{\poly}(\psi) \le \int_{\phi < \psi} \MA_{\poly}(\phi)$ .
    \end{itemize}
\end{prop}
\begin{proof}
We follow the proof of \cite[Proposition 6.11]{BoucksomFavreJonsson2015}, and set $u = \max \{ \phi, \psi \}$. By Remark \ref{rem:MApsh}, for any $t \in \R$, we have $\mathbf{1}_{\{ \phi > \gamma-t\} } \PMA(\phi^{\langle t \rangle}) = \mathbf{1}_{ \{\phi > \gamma-t\} } \PMA(\phi)$, hence
\begin{alignat*}{2}
\mathbf{1}_{\{ \phi^{\langle t \rangle} > \psi^{\langle t \rangle}\} } \PMA(\phi)
& =\mathbf{1}_{\{ \phi^{\langle t \rangle} > \psi^{\langle t \rangle}\} } \PMA(\phi^{\langle t \rangle})
&& \; \text{as $\{ \phi^{\langle t \rangle} > \psi^{\langle t \rangle}\} \subseteq \{ \phi > \gamma -t \}$} \\
& = \mathbf{1}_{\{ \phi^{\langle t \rangle} > \psi^{\langle t \rangle}\} } \PMA(u^{\langle t \rangle})
&& \; \text{by locality (Proposition \ref{cor:locality})}\\
& =\mathbf{1}_{\{ \phi^{\langle t \rangle} > \psi^{\langle t \rangle}\} } \PMA(u)
&& \; \text{as $\{ \phi^{\langle t \rangle} > \psi^{\langle t \rangle}\} \subseteq \{ u > \gamma -t \}$ and Remark \ref{rem:MApsh}.}
\end{alignat*}
Now we write
$$\mathbf{1}_{\{ \phi^{\langle t \rangle} > \psi^{\langle t \rangle}\} } \PMA(u) = \mathbf{1}_{\{ \phi > \gamma-t \geq \psi \}} \PMA(u)+\mathbf{1}_{\{\phi > \psi> \gamma-t\}} \PMA(u).$$
As $\mathbf{1}_{\{ \phi > \gamma-t \geq \psi \}}$ goes to zero pointwise on $\tX$, the first term of the right-hand side  converges weakly to zero by dominated convergence, while the second term converges weakly to $\mathbf{1}_{\{\phi > \psi\}} \PMA(u)$ also by dominated convergence. Since the left-hand side converges to $\mathbf{1}_{\{\phi>\psi\}} \PMA(\phi)$, we conclude the proof of locality. 

The comparison principle follows exactly as in the proof of Corollary \ref{cor:comparison}.
\end{proof}
\subsection{Capacity} 
In this section, let $\Pi$ be a simplicial polyhedral structure on $\tX$, and assume that $\gamma \in \PAPC^+(\Pi)$. 
The goal of this section is to establish the domination principle (Proposition \ref{prop:dom}), using the notion of \emph{capacity} of a Borel subset of $\tX$. 
\begin{defi} \label{defi:capa}
Let $E \subset \tX$ be a Borel set. Its capacity is the real number

\[\Capa(E) \coloneqq \sup \left\{ \int_E \MA_{\poly}(\phi)\; | \; \phi \in \PSH(\tX,\gamma), \gamma-1 \le \phi \le \gamma \right\}.
\]
\end{defi}
\begin{lemma} \label{lem:cap open}
    Every non-empty open subset in $\tX$ has positive capacity.
\end{lemma}

\begin{proof} 
Since the capacity is clearly increasing with respect to inclusions of Borel sets, it is enough to show that every non-empty open subset contains a point $x$ such that $\Capa\{x\}>0$. By the pure-dimensionality of $\tX$, it suffices to show that there exists a polyhedral structure $\Pi$ on $\tX$ such that any rational point $x \in \tX \cap N_\Q$ in the relative interior of a top-dimensional face of $\Pi$ has $\Capa\{x\}>0$.

To this purpose, let $\Pi$ be the simplicial polyhedral structure of $\tX$ such that $\gamma \in \PAPC^+(\Pi)$. Let $\sigma \in \Pi^{(d)}$ be a maximal face, and write $\sigma=\conv(p_0,\ldots,p_l)+\Cone(v_{l+1},\ldots,v_d)$. Let $x \in N_{\Q}$ be any point in the relative interior of $\sigma$. The refinement $\Pi_x$ of $\Pi$ is defined as the collection of
\begin{itemize}
    \item the polyhedra $\nu \in \Pi$ not containing $x$,
    \item the polyhedron 
    $
    \langle \tau,x \rangle \coloneqq \conv(x,p_{i_0},\ldots,p_{i_k})+\Cone(v_{i_{k+1}},\ldots,v_{i_j}),
    $ for any proper face $\tau=\conv(p_{i_0},\ldots,p_{i_k})+\Cone(v_{i_{k+1}},\ldots,v_{i_j}) \in \Pi$ of $\sigma$.
\end{itemize}
The polyhedral structure $\Pi_x$ is simplicial. Indeed, let $\tau$ be a proper face of $\sigma$ of dimension $j$. One one hand, $\tau$ doesn't contain $x$, hence $\langle \tau,x \rangle$ has dimension at least $j+1$; on the other hand, by definition, $\langle \tau,x \rangle$ has dimension at most $j+1$. 

Define $f \in \PA_b(\Pi_x)$ as the unique bounded piecewise affine function such that $f$ vanishes on any face of $\Pi$ not containing $x$, and takes value $-1$ at $x$. By adapting the proof of \cite[Proposition 4.17]{AminiPiquerez2020}, we show that $f$ is strictly convex relative to $\Pi$, i.e. for any face $\mu \in \Pi$ the restriction of $f$ to $(\Pi_x)_{|\mu}$ is strictly convex.

Let $\mu \in \Pi$. If $\mu$ does not contain $x$, then $(\Pi_x)_{|\mu}$ consists of $\mu$ and its faces, and $f$ is zero, hence convex on $(\Pi_x)_{|\mu}$. Assume now that $\mu=\sigma$. Let $\tau$ be any proper face of $\sigma$; we want to show that $f_{|\sigma}$ is strictly convex around $\tau$ and around $\langle \tau, x \rangle$. Let $\ell$ be an affine linear function on $\R^n \supseteq \tX$ which vanishes on $\tau$ and is strictly negative on $\sigma \setminus \tau$; we can assume that $\ell(x)=-1$. Now, by construction we have
\begin{itemize}
    \item $(f-\ell)_{|\langle \tau, x \rangle}=0$,
    \item if $y \in \sigma \setminus \langle \tau, x \rangle$, then there exists a face $\nu$ of $\sigma$ such that $y \in \langle \nu, x \rangle \setminus \langle \tau, x \rangle$ and we can write $y=z + \lambda x$, with $z \in \nu \setminus \tau$ and $\lambda \geq 0$, such that $$(f-\ell)(y)=\underbrace{(f-\ell)(z)}_\text{$>0$}+ \lambda \underbrace{(f-\ell)(x)}_\text{=0} >0.$$
\end{itemize}
This proves the strict convexity of $f$ around $\langle \tau, x \rangle$. Similarly, we have
\begin{itemize}
    \item $(f-2\ell)_{|\tau}=0$,
    \item if $y \in \sigma \setminus \tau$, then either there exists a face $\nu \neq \tau$ of $\sigma$ such that $y \in \langle \nu, x \rangle \setminus \tau$ and we can write $y=z + \lambda x$, with $z \in \nu \setminus \tau$ and $\lambda \geq 0$, such that $$(f-2\ell)(y)=\underbrace{(f-2\ell)(z)}_\text{$>0$}+ \lambda \underbrace{(f-2\ell)(x)}_\text{$>0$} >0;$$
    or we have that $y \in \langle \tau, x \rangle \setminus \tau$, hence we can write $y=z + \lambda x$, with $z \in \tau$ and $\lambda \geq 0$, which implies that $$(f-2\ell)(y)=\underbrace{(f-2\ell)(z)}_\text{$=0$}+ \lambda \underbrace{(f-2\ell)(x)}_\text{$>0$} >0.$$
\end{itemize}
This proves the strict convexity of $f$ around $\tau$. We conclude therefore that $f \in \PA(\Pi_x)$ is strictly convex relatively to $\Pi$.

For any sufficiently small $\varepsilon \in \Q_{>0}$, the function $\gamma_x \coloneqq \gamma + \varepsilon f$ is strictly convex on $\Pi_x$ by \cite[Proof of Proposition 4.15]{AminiPiquerez2020}; by Remark \ref{rem:PAPC+}, we have $$\PMA(\gamma_x)\{x\} >0$$ as $ (\gamma_x)^d \cdot [\tX]$ is a strictly positive $0$-cycle defined on $\Pi_x$ and $x \in \Pi_x^{(0)}$. Finally, the inequality $\gamma -1 \leq \gamma_x \leq \gamma$ holds by construction of $f$, hence $\gamma_x$ is a competitor in the supremum defining the capacity, and therefore $\Capa\{x\}>0$.
\end{proof}

\begin{lemma}\label{lem:cap}
    Let $\phi, \psi \in \mathcal{E}^1(\tX,\gamma)$ with $\phi \le \gamma$. Then
    \[
    \Capa \{ \psi < \phi \} \le t^{-d} \int_{\{ \psi<(1-t)\phi +t (\gamma+1) \}} \MA_{\poly}(\psi)
    \]
    for $0 < t < 1$.
\end{lemma}
\begin{proof}
We adapt the proof from \cite[Lemma 8.3]{BoucksomFavreJonsson2015}. Fix $u \in \PSH(\tX, \gamma)$ with $\gamma \leq u \leq \gamma +1$  
Furthermore, set $\phi_t \coloneqq (1-t)\phi + t u$. Since $\phi \leq \gamma \le u$ and $u \le (\gamma+1)$ we have
\[
\left\{\psi < \phi \right\} \subseteq \left\{\psi < \phi_t\right\} \subseteq \left\{\psi < (1-t)\phi + t (\gamma+1) \right\}.
\]
Now, from the additivity property (3) of Theorem \ref{thm:PMA PCreg} we get that $\MA_{\poly}(\phi_t) \geq t^d \,\PMA(u)$. Hence, 
\begin{eqnarray*}
t^d \int_{\psi < \phi} \MA_{\poly}(u) \leq \int_{\psi < \phi}\MA_{\poly} (\phi_t) \leq \int_{\psi < \phi_t} \PMA(\psi) \leq \int_{\psi < (1-t)\phi + t(\gamma+1)} \MA_{\poly}(\psi),
\end{eqnarray*}
where the second inequality follows from the comparison principle in Proposition \ref{prop:comparisonE}. Taking the supremum over all $u$ shows the result.
\end{proof}

We are now ready to prove the following: 
\begin{prop}[Domination principle]\label{prop:dom}
Let $\phi \in \PCreg(\tX, \gamma)$ and $\psi \in \mathcal{E}^1(\tX,\gamma)$, such that
\begin{itemize}
    \item $\nu \coloneqq \MA_{\poly}(\psi)$ is supported on a compact set $K \subset \tX$,
    \item the inequality $\phi \le \psi$ holds $\nu$-almost everywhere on $K$.
\end{itemize}
Then $\phi \le \psi$ on $\tX$.
\end{prop}
\begin{proof}
Up to shifting by a constant, we may assume $\gamma -c \le \phi \le \gamma$ for some $c \in \R$. Let $\eps >0$, and $t \ll 1$ such that $t(c+1) \le \eps/2$. Then we have
$$ \nu \{ \psi + \eps < (1-t) \phi +t (\gamma+1) \} \le \nu \{ \psi + \eps/2 < \phi \} =0$$ by the hypothesis on $\phi$ and $\psi$.
Lemma \ref{lem:cap} implies $\Capa \{ \psi + \eps < \phi \}  =0$. The latter is however an open subset of $\tX$ by continuity of $\PSH$ functions, and must therefore be empty by Lemma \ref{lem:cap open}. It follows that $\phi \ge \psi + \eps$, and this for any $\eps >0$, which concludes the proof.
\end{proof}

\subsection{Boundedness of solutions to the Monge--Ampère equation}

\begin{prop} \label{prop:boundedsol}
    Let $\phi \in \mathcal{E}^1(\tX,\gamma)$. If the measure $\MA_{\poly} (\phi)$ is supported on a compact subset $K \subset \tX$, then $\phi \in \PCreg(\tX, \gamma)$.
\end{prop}
\begin{proof}
We may assume $\sup_{\tX} (\phi-\gamma) =-1$. By assumption there exists a sequence $(\phi_j)_j$ in $\PAPC(\tX,\gamma)$ decreasing pointwise to $\phi$, we will show that the convergence is in fact uniform. 
\\By Theorem \ref{thm:comp} the sequence $(\sup (\phi_j - \gamma))_j$ converges to $-1$, so we may assume $\phi_j \le \gamma$ for all $j$. Let $\eps>0$, by Dini's lemma we have $\phi \le \phi_j + \eps$ on $K$ for $j \gg 1$, which by the domination principle (Proposition \ref{prop:dom}) implies $\phi \le \phi_j + \eps$ on $\tX$, and thus convergence is uniform.
\end{proof}

\section{The polyhedral Monge--Ampère equation} \label{sec:PMA}
Let $\tX \subseteq N_{\R}$ be a balanced polyhedral space of dimension $d$, and $\gamma$ a $\PAPC$ function on $\tX$ that is positive at infinity. In this section, we study polyhedral Monge--Amp\`ere equations and establish sufficient conditions for the existence of solutions using a variational approach.

\subsection{Regularization} \label{sec:reg}
\begin{defi} 
A continuous function $f \in \mathcal{C}^0(\tX)$ is said to be quasi-$\gamma$-$\PSH$, if there exists $\phi \in \PSH(\tX, \gamma)$ and $g \in \mathcal{C}^0_b(\tX)$ such that $f = \phi+g$. 
\end{defi}
\begin{defi}\label{defi:envelope}
Let $u$ be a continuous function on $\tX$. The envelope $P_{\gamma}(u)$ is defined as
$$P_{\gamma}(u) (x) \coloneqq \sup \{ \psi(x) \,| \,\psi \le (\gamma +u), \psi \in \PC(\tX), \psi \le \gamma+O(1) \},$$
with the convention that $P_{\gamma}(u) \equiv -\infty$ if the set of functions $\{ \psi \in \PC(\tX) | \psi \le \gamma+u, \psi \le \gamma +O(1) \}$ is empty.
\end{defi}
If $(\gamma+u)$ is quasi-$\gamma$-$\PSH$, then $P_{\gamma}(u) \not\equiv -\infty$. Indeed, let $\gamma+u = \phi+g$ with $\phi \in \PSH(\tX,\gamma)$ and $g \in \mathcal{C}^0_b(\tX)$, then $\psi \coloneqq \phi + \inf_{\tX} g$ is a competitor in the supremum defining the envelope.

The envelope satisfies the following basic properties.
\begin{prop} \label{prop:prop env}
Let notations as above and let $c \in \R$. The following hold true:
\begin{enumerate}
\item if $P_{\gamma}(u) \not\equiv - \infty$, then $P_\gamma(u) \in \PC(\tX)$,
\item $P_{\gamma}(u) \le \gamma+u$,
\item if $u$ is bounded, then $P_{\gamma}(u) \in \PC(\tX, \gamma)$,
\item if $(\gamma+u) \in \PSH(\tX,\gamma)$ or $(\gamma+u) \in \PC(\tX,\gamma)$, then $P_{\gamma}(u) = \gamma +u$,
\item if $u \le v$ then $P_{\gamma}(u) \le P_{\gamma}(v)$,
\item $P_{\gamma}(u+c) =P_{\gamma}(u) +c$,
\item $P_{\gamma}(u)$ is concave in both arguments in the following sense: 
$$P_{t\gamma+(1-t)\gamma'}(tu + (1-t)u') \geq tP_{\gamma}(u) + (1-t)P_{\gamma'}(u'),$$
\item $\sup_{\tX} \lvert P_{\gamma}(u) - P_{\gamma}(v) \lvert \le \sup_{\tX} \lvert u-v \rvert$.
\end{enumerate}
\end{prop}
\begin{proof}
If $ P_{\gamma}(u) \not\equiv - \infty$, then the subset $\{ \psi \in \PC(\tX) | \psi \le \gamma+u, \psi \le \gamma +O(1) \} \subset \PC(\tX)$ is non-empty and uniformly bounded from above, so that its pointwise sup defines a PC function by \cite[Theorem 6.24]{BBS}, and Part (1) follows. Part (2) is clear, and Part (3) follows from the inequalities:
$$\gamma+\inf_{\tX} u \le P_{\gamma}(u) \le \gamma+u \le \gamma+ \sup_{\tX} u.$$
Finally, (4)-(6) are clear from the definitions. 
Part (7) follows from the fact that if $\psi \in \PC(\tX)$, $\psi' \in \PC(\tX)$ with $\psi \leq \gamma + u$ and $\psi' \leq \gamma' + u'$, $\psi \le \gamma+O(1)$ and $\psi \le \gamma' \le O(1)$ then 
$$ t\psi + (1-t)\psi' \in \PSH(\tX, t\gamma + (1-t)\gamma')$$ and is dominated by $tu+ (1-t)u'$. 
Finally, (8) follows from applying (5) to the inequality $v + \inf (u-v) \le u \le v + \sup (u-v)$.
\end{proof}

\begin{theorem}\label{th:regularisation}
The following conditions are equivalent:
\begin{enumerate}
    \item for any bounded piecewise affine function $u$ on $\tX$, the envelope $P_\gamma(u)$ is in $\PCreg(\tX,\gamma)$;
    \item the subspace $\PAPC(\tX, \gamma) \subseteq \PC(\tX, \gamma)$ is dense for the topology of uniform convergence;
    \item $\PC(\tX,\gamma)=\PCreg(\tX,\gamma)$.
\end{enumerate}
\end{theorem}
\begin{proof}
We start by proving $(1) \implies (2)$. Let $\phi \in \PC(\tX, \gamma)$ and let $\Pi$ be a simplicial polyhedral structure where $\gamma$ is defined. By Lemma \ref{lem:ext} $\phi$ admits a continuous extension to $\overline{\tX}_\Pi$, and by Proposition \ref{prop:bounded-dense} there exists $(\phi_j)_j$ a sequence in $\PA(\tX,\gamma)$ converging uniformly to $\phi$. 
Set $\psi_j \coloneqq P_{\gamma}(\phi_j -\gamma)$; by $(1)$, $\psi_j$ is the uniform limit of functions $(\psi_{j,i})_i$ in $\PAPC(\tX,\gamma)$. For $j \gg 1$, we have $\lVert \phi_j -\phi \lVert_{\infty} \le \eps$, so the last point of Proposition \ref{prop:prop env} applied to $(\phi -\gamma)$ and $(\phi_j - \gamma)$ yields
$$\lVert \phi -\psi_j \rVert_{\infty} \le \eps,$$
as $P_{\gamma}(\phi-\gamma) = \phi$. It follows that $(\psi_j)_j$ is a sequence of functions converging uniformly to $\phi$, and we conclude by using a standard diagonal argument on $(\psi_{j,i})_{i,j}$.

The implication $(2) \implies (3)$ follows from Definition/Proposition \ref{def:PC-reg}, while $(3) \implies (1)$ follows from item $(3)$ in Proposition \ref{prop:prop env}.
\end{proof}

\begin{defi} \label{def: regularity property}
We say that $(\tX,\gamma)$ has the regularity property if one of the equivalent conditions in Theorem \ref{th:regularisation} holds.
\end{defi}

\begin{rem}\label{rem:regularity}
In the complete case $\tX=N_{\R}$, the regularity property holds for any $\gamma$. Indeed, in this case the $\gamma$-envelope of a function $u$ can be described as the convex support function of the convex hull of the epigraph of $u+\gamma$ (see \cite[p.\,36]{Rockafellar}). If $u$ is piecewise affine, then the epigraph is polyhedral, and so is its convex hull. Thus, the associated support function is again piecewise affine.
When $\tX$ is not complete, one may attempt to replace the usual convex hull by a polyhedral convex hull. However, it is unclear whether the polyhedral convex hull preserves polyhedrality in general. On the other hand, similar to \cite[Theorem 2.13]{BoteroGil2022}, one can show that conical $\PC$ functions on conical polyhedral spaces (i.\,e.\,\emph{tropical fans}), are regularizable. This amounts to the fact that by the conical property the $\PC$ functions are determined by their restrictions to the compact unit sphere. 
\end{rem}

\begin{cor} Assume that $(\tX,\gamma)$ has the regularity property. For any $u \in \mathcal{C}_b^0(\tX)$ we have
$$P_{\gamma}(u) (x) = \sup \left\{ \psi(x) \; |  \; \psi \le (\gamma +u), \psi \in \PAPC(\tX, \gamma) \right\}.$$
\end{cor}
\begin{proof}
The inequality $\geq$ is clear by definition. To prove the reverse inequality at $x \in \tX$, it is enough to show that for any $\eps >0$, there exists $\psi \in \PAPC(\tX,\gamma)$ such that $\psi \le (\gamma+u)$ and $\psi(x) \ge P_{\gamma}(u)(x) - \eps$.
By item (3) from Proposition \ref{prop:prop env}, $P_{\gamma}(u) \in \PC(\tX, \gamma)$, so that by Theorem \ref{th:regularisation}, there exists a sequence $(\psi_j)_j$ in $\PAPC(\tX,\gamma)$, that we can assume to be increasing, converging uniformly to $P_{\gamma}(u)$. As a result, $\psi_j(x) \ge P_{\gamma}(u)(x) - \eps$ work for $j$ sufficiently large, concluding the proof.
\end{proof}

We show that Theorem \ref{th:regularisation} implies that Properties (i) and (ii) from Proposition \ref{prop:gamma-psh} are in fact equivalent to being $\gamma$-$\PSH$.
\begin{cor} \label{cor:envpsh}Assume that $(\tX,\gamma)$ has the regularity property. A function $\phi$ lies in $\PSH(\tX, \gamma)$ if and only if $\phi \in \PC(\tX)$ and $\phi \leq \gamma + O(1)$. 
\\As a consequence, if $u$ is a continuous function on $\tX$ such that $(\gamma+u)$ is quasi-$\gamma$-$\PSH$, then $P_{\gamma}(u) \in \PSH(\tX, \gamma)$.
\end{cor}
\begin{proof}
The implication $\Rightarrow$ follows from Proposition \ref{prop:gamma-psh} so that we prove the reverse inclusion.
\\Let $\phi \in \PC(\tX)$ such that $\phi \leq \gamma + O(1)$, and set $\phi_j \coloneqq \max \{ \phi, \gamma -j \}$, which lie in $\PC(\tX,\gamma)$. By Theorem \ref{th:regularisation}, for all $j$, there exists a sequence $(\psi_{j, n})_n$ in $\PAPC(\tX,\gamma)$ converging uniformly to $\phi_j$, and we may assume it is decreasing with respect to $n$. The first claim then follows by using a standard diagonal argument (see e.\,g.\,\cite[Lemma~4.6]{BoucksomJonssonc}).

For the second claim, write $\gamma+u = \phi+g$ with $\phi \in \PSH(\tX,\gamma)$ and $g \in \mathcal{C}^0_b(\tX)$. The bound $P_{\gamma}(u) \le \gamma+u \le \phi+ \sup_{\tX} g$ shows that $P_{\gamma}(u) \le \gamma+O(1)$, so that the first claim and item (1) of Proposition \ref{prop:prop env} show that $P_{\gamma}(u) \in \PSH(\tX,\gamma)$.
\end{proof}

\subsection{Orthogonality}\label{sec:ortho}
\begin{defi} \label{defi:orthogonality property}
Assume that $(\tX,\gamma)$ has the regularity property. We say that $(\tX,\gamma)$ has the orthogonality property if, for any bounded piecewise affine function $u$ on $\tX$, the envelope $P_\gamma(u)  \in \PCreg(\tX,\gamma)$ satisfies
\begin{align*}\label{eq:ortho}\int_{\tX} (P_{\gamma}(u)-(\gamma+u)) \PMA(P_{\gamma}(u)) =0
\end{align*}
\end{defi}

\begin{theorem} \label{th:diffEP}
Assume that $(\tX,\gamma)$ has the regularity and orthogonality properties. Let $u, v \in \PA_b(\tX)$. Then $E \circ P_{\gamma}$ is differentiable at $u$ in the direction $v$, and
$$\left. \frac{d}{dt} \right|_{t=0} E(P_{\gamma}(u+tv)) = \int_{\tX} v \,\PMA(P_{\gamma}(u)).
$$

\end{theorem}
\begin{proof}
   Set $\mu \coloneqq \PMA(P_{\gamma}(u))$. We use the concavity properties of both the energy functional and of the envelope (see Propositions \ref{prop:energy} and \ref{prop:prop env}) and we argue as in the proof of \cite[Theorem~7.2]{BoucksomFavreJonsson2015} to obtain that
   $$\left. \frac{d}{dt} \right|_{t=0} E(P_{\gamma}(u+tv)) = \left. \frac{d}{dt} \right|_{t=0} \int_{\tX} (P_{\gamma}(u+tv)-\gamma) \mu.$$

    We are left to show that the right-hand derivative in the above equality agrees with $\int_{\tX} v \mu$, or equivalently
   $$\int_{\tX} (P_{\gamma}(u+tv)-P_\gamma(u)) \mu = t \int_{\tX} v \mu + o(t).$$
   Write $\Omega_t \coloneqq \{ P_{\gamma}(u+tv) < P_{\gamma}(u)+tv \}$, we claim that $\mu(\Omega_t) = O(t)$. Granting the claim for now, we observe that $\lvert (P_{\gamma}(u+tv) -P_{\gamma}(u)) \lvert \le t \sup_{\tX} \lvert v \rvert$ by Proposition \ref{prop:prop env}. Moreover, we have $P_{\gamma}(u+tv) \le (\gamma+u+tv) = P_{\gamma}(u)+tv$ $\mu$-almost everywhere by orthogonality. It follows that
   \begin{align*}
        \int_{\tX} (P_{\gamma}(u+tv)-P_\gamma(u)-tv) \mu 
        & = \int_{\Omega_t} (P_{\gamma}(u+tv)-P_\gamma(u)-tv) \mu\\ 
        & \le \mu(\Omega_t) \sup_{\tX} \lvert P_{\gamma}(u+tv)-P_\gamma(u)-tv\rvert \\ 
        & \le 2 \mu(\Omega_t) \cdot t \sup_{\tX} \lvert v \rvert = O(t^2),
   \end{align*}
   and the result follows. 
   
   We now prove the claim on $\mu(\Omega_t)$.
   To this end, let $\Pi$ be a polyhedral structure on $\tX$ such that $v \in \PA(\Pi)$. By Proposition \ref{prop:quasi-proj} there exists a quasi-projective refinement $\Pi'$ of $\Pi$, hence we fix $\gamma' \in \PAPC^+(\Pi')$. By Lemma \ref{lem:PA plus PAPC+}, there exists $m>0$ such that $m\,\gamma'+v \in \PAPC(\tX)$. Writing $w \coloneqq m\,\gamma'$ and $\Omega_t = \{ P_{\gamma}(u+tv)+tw < P_{\gamma}(u)+t(v+w) \}$, we apply the comparison principle (Corollary \ref{cor:comparison}) to these functions lying in $\PCreg(\tX,\gamma +tm\,\gamma')$ to get
   $$\int_{\Omega_t} \PMA(P_{\gamma}(u)+t(v+w)) \le \int_{\Omega_t} \PMA(P_{\gamma}(u+tv)+tw )).$$
   Expanding both Monge-Ampère measures by multilinearity now gives
   $$\int_{\Omega_t} \mu \le \int_{\Omega_t} \PMA(P_{\gamma}(u+tv)) +O(t),$$
   but the integral on the right-hand side vanishes by orthogonality, which concludes the proof.
\end{proof}

\begin{cor} \label{cor:diff}
Assume that $(\tX,\gamma)$ has the regularity and orthogonality properties. Let $u \in \mathcal{C}^0(\tX)$  such that $(\gamma+u) \in \mathcal{E}^1(\tX, \gamma)$. Then $E \circ P_{\gamma}$ is differentiable at $u$, and we have
$$\left. \frac{d}{dt} \right|_{t=0} E(P_{\gamma}(u+tv)) = \int_{\tX} v \,\PMA(P_{\gamma}(u)),
$$
for all $v \in \PA_b(\tX)$.
\end{cor}
\begin{proof}
Note that $(\gamma+u)+tv$ is in particular quasi-$\gamma$-PSH, so that the envelope $P_{\gamma}(u+tv)$ lies in $\PSH(\tX,\gamma)$ for all $t$ by Corollary \ref{cor:envpsh}. Moreover, since $P_{\gamma}(u+tv) \ge (\gamma+u) - \lvert t \rvert \sup_{\tX} \lvert v \rvert$ and $(\gamma+u)$ has finite energy, so does $P_{\gamma}(u+tv)$ by monotonicity of the energy.

We will reduce to the setting of Theorem \ref{th:diffEP}. Since $(\gamma +u) \in \PSH(\tX,\gamma)$, there exists a sequence $(u_j)$ in $\PA_b(\tX)$, decreasing pointwise to $u$, with $(\gamma+u_j) \in \PAPC(\tX,\gamma)$. Set $\psi_j \coloneqq P_\gamma(u_j +tv)$, it is a decreasing sequence in $\PCreg(\tX,\gamma)$ by the regularity property, and we claim that its pointwise limit $\psi \coloneqq \lim_j \psi_j$ is $P_\gamma(u+tv)$. Indeed, since $u_j +tv \ge u+tv$, we have $\psi_j \ge P_{\gamma}(u+tv)$ for all $j$ by Proposition \ref{prop:prop env}, and thus $\psi \geq P_\gamma(u+tv)$. For the reverse inequality, note that $\psi \le \psi_j \le (\gamma+u_j +tv)$ for all $j$, whence $\psi \le (\gamma+u +tv)$. Since $\psi$ is $\PSH(\tX,\gamma)$, we infer $\psi \le P_\gamma(u+tv)$, which proves the claim. Using Lemma \ref{lem:energy}, we then have
$$E(P_{\gamma}(u+tv)) = \lim_j E(P_{\gamma}(u_j+tv)).$$
Applying Theorem \ref{th:diffEP}, we have
$$E(P_{\gamma}(u_j+tv)) = E(u_j) + \int_0^t \Bigg( \int v \,\PMA(P_{\gamma}(u_j +s v)) \Bigg) ds.$$
As $j \to \infty$, the left-hand side goes to $E(P_{\gamma}(u+tv))$, while the right-hand side goes to
$$E(u) + \int_0^t \Bigg( \int v\, \PMA(P_\gamma(u +s v)) \Bigg) ds,$$
by Proposition \ref{prop:MAdecr}. Differentiating with respect to $t$ now concludes the proof.
\end{proof}

\subsection{The variational approach to the polyhedral Monge--Ampère equation}\label{sec:variational app}

Let $\mu$ be a compactly supported Radon measure on $\tX$, of total mass $\deg(\gamma)$. Consider the functional 
\begin{align*}
F_{\mu} \colon \PSH(\tX,\gamma) \to 
&\,  \R \cup \{- \infty \} \\
\phi \mapsto 
& \, F_{\mu} (\phi) = E(\phi)-\int_{\tX} \phi \,\mu.
\end{align*}
Since $\mu$ is compactly supported and $\PSH(\tX,\gamma)$ is endowed with the topology of uniform convergence on compact subsets, the functional $F_\mu$ is usc as $E$ is. By the compactness result in Theorem \ref{thm:comp}, the functional $F_{\mu}$ admits a maximizer $\phi \in \PSH(\tX,\gamma)$. 

\begin{prop} \label{prop:MA maxsol}
Assume that $(\tX,\gamma)$ has the regularity and orthogonality properties. If $\phi \in \PSH(\tX,\gamma)$ is a maximizer of $F_\mu$, then we have $\PMA(\phi) = \mu$. Moreover if $\gamma$ is strictly convex, then $\phi \in \PCreg(\tX,\gamma)$. 
\end{prop}
\begin{proof}
We start by observing that since $\phi$ maximizes $F_{\mu}$, we must have $E(\phi)> - \infty$, whence $\phi \in \mathcal{E}^1(\tX,\gamma)$. Let $v \in \PA_b(\tX)$, we will prove that $\int_{\tX} v \,\mu = \int_{\tX} v \PMA(\phi)$, which will imply the proposition by density of $\PA_b(\tX)$ in the space of continuous functions on a compactification of $\tX$ (Proposition \ref{prop:bounded-dense}). Set $\psi = \phi-\gamma$, and
    $$h(t)\coloneqq E(P_{\gamma}(\psi+tv)) -\int_{\tX}(\psi+tv)\, \mu.$$
By Corollary \ref{cor:diff}, $h$ is differentiable at $t=0$ and moreover we have that 
\[
h'(0) =\int_{\tX} v \,\PMA(\phi) - \int_{\tX} v \,\mu.
\]
Since $\phi$ maximizes $F_\mu$ it follows that $h(0)= F_{\mu}(\phi)$ is a local maximum of $h$, hence $h'(0)=0$. Finally, the maximizer $\phi$ lies in $\PCreg(\tX,\gamma)$ by Proposition \ref{prop:boundedsol}, if $\gamma$ is strictly convex. 
\end{proof}

\subsection{Dimension one}
\begin{prop} \label{prop:env PA} 
Assume that $\tX$ is one-dimensional. For any bounded $\PA$ function $u$ on $\tX$ defined on a polyhedral structure $\Pi$ where $\gamma$ is also defined, the envelope $P_{\gamma}(u)$ is in $\PAPC(\tX,\gamma)$ and is defined on $\Pi$.
\\In particular, the regularity property holds in dimension one.
\end{prop}
\begin{proof}
We prove that $P_{\gamma}(u)$ is affine on each $e$ of $\Pi$, we start with the case where $e$ is bounded. Up to adding a global affine function, we may assume that $P_{\gamma}(u) \le 0$ on $e$ and is zero on the vertices of $e$. We will show that $P_{\gamma}(u)$ has to be zero on the whole edge. Assume by contradiction that there exists $x \in \relint(e)$ such that $P_{\gamma}(u)(x) = - \eps <0$.
We claim that the function 
$$Q= \begin{cases}
    \max \{ P_{\gamma}(u), -\eps/2 \} & \text{on } e \\
    P_{\gamma}(u) & \text{otherwise}
\end{cases}
$$
lies in $\PC(\tX,\gamma)$. Since moreover $Q \le (\gamma+u)$, this will provide a contradiction as $Q(x) > P_{\gamma}(u)(x)$.
To prove the claim, let $U \subset \tX$ be the complement in $\tX$ of $\relint(e) \cap \{ P_{\gamma}(u) \ge -3\eps/4 \} $. Then $U \cup \relint(e)$ is an open cover of $\tX$, and $Q = P_{\gamma}(u)$ on $U$. As a result, $Q$ is $\PC$ by \cite[Proposition 5.13]{BBS}. 

We now let $e$ be an unbounded edge with vertex $v$, and assume again that $P_{\gamma}(u)(v)=0$. Assume by contradiction that the function $(P_{\gamma}(u)-\gamma)$ is not constant on $e$. Since it is a bounded convex function on $e$, it is non-positive, so that there exists $x \in \relint(e)$ with $P_{\gamma}(u)(x) = \gamma(x)-\eps$, with $\eps >0$. Similarly to the previous case, the function:
$$Q= \begin{cases}
    \max \{ P_{\gamma}(u), \gamma-\eps/2 \} & \text{on } e \\
    P_{\gamma}(u) & \text{otherwise}
\end{cases}
$$
lies in $\PC(\gamma)$, satisfies $Q \le \gamma+u$, and we have $Q(x) = \gamma(x)-\eps/2 >P_{\gamma}(u)(x)$, a contradiction.
\end{proof}

We introduce a Laplacian operator for suitable regular functions on $\tX$, and compare it with the polyhedral Monge--Ampère operator. It is an unbounded version of the Laplacian operator on bounded metric graphs introduced by Baker and Faber in \cite{baker-faber}. To this purpose, recall that given a polyhedral structure $\Pi$ on $\tX$, any edge $e$ of $\Pi$ carries an integral affine structure, and hence admits a natural coordinate $t_e$ (defined up to a sign, corresponding to orientation of the edge). We then equip $\tX$ with the measure $d\lambda$ such that for any edge $e$ of $\Pi$, $d \lambda_{|e} = [\tX](e) d\ell_e$, where $d \ell_e=d t_e$ is the Lebesgue measure on $e$ normalized by the integral affine structure. 

\begin{defi} \label{defi:PS}
A continuous function $f$ on $\tX$ is piecewise smooth if there exists a polyhedral structure $\Pi$ on $\tX$ such that $f_{|e}$ is $\mathcal{C}^2$ for any edge $e \in \Pi^{(1)}$; for any such $\Pi$, we say that $f$ \emph{is defined on} $\Pi$. The set of such functions is denoted by $\operatorname{PS}(\tX)$. 
\end{defi}

If $f \in \operatorname{PS}(\tX)$ is defined on $\Pi$, the second derivative 
$$f''_{| e} \colon e \to \R$$ defined as $\frac{d^2 f_{|e}}{d t_e^2}$ is a well-defined continuous function on $e$. Moreover, if $v \in \Pi^{(0)}$ and $e$ is an edge adjacent to $v$, define the outgoing slope of $f$ at $v$ in the direction $e$ as
$$s_{e/v}(f) \coloneqq \left. \frac{d}{dt} \right|_{t=0^+} f(v+tn_{e/v}),$$
where $n_{e/v}$ is the normal vector of $e$ relative to $v$; see Definition \ref{def:normal vectors}.

\begin{defi}\label{def:laplacian}
Let $f$ be a piecewise smooth function on $\tX$, defined on a polyhedral structure $\Pi$. The \emph{Laplacian} of $f$ is the measure on $\tX$ defined as
$$\Delta f \coloneqq \sum_{e \in \Pi^{(1)}} f_{| e}'' \, d\lambda_{|e} + \sum_{v \in \Pi^{(0)}} \sum_{\substack{e \succ v \\ e \in \Pi^{(1)}}} [\tX](e) s_{e/v}(f) \; \delta_v.$$
\end{defi}
\begin{rem}
The Laplacian $\Delta f$ depends only on the space $\tX$ and not on the choice of polyhedral structure $\Pi$. Moreover, it is a linear operator in the sense that $\Delta(af + bg) = a\Delta f + b \Delta g$ for any piecewise smooth functions $f$ and $g$ and any real numbers $a$ and $b$. 
\end{rem}
The Laplacian agrees with the measure $\PMA(f)$ whenever $f \in \PA(\tX)$, since $s_{e/v}(f) = f_e(n_{e/v})$ with the notation from Section \ref{sec:intersection}, and $f_{| e}''=0$. We claim that this holds more generally.

\begin{prop}\label{prop:laplacian}
For any $\phi \in \PC(\tX,\gamma) \cap \operatorname{PS}(\tX)$ the equality of measures
$$\PMA(\phi) = \Delta \phi$$
holds.
\end{prop}
\begin{proof}
Let $\phi \in \PC(\tX,\gamma) \cap \operatorname{PS}(\tX)$. By Proposition \ref{prop:bounded-dense} it suffices to show that we have
$$\int_{\tX} g \,\PMA(\phi)= \int_{\tX} g\, \Delta \phi$$
for any $g \in \PA_b(\tX)$.
By Proposition \ref{prop:env PA} and Theorem \ref{th:regularisation}, there exists a sequence $(\phi_j)_j \in \PAPC(\tX,\gamma)$ converging uniformly to $\phi$, and we have
$$\int_{\tX} g\, \PMA(\phi) 
= \lim_j \int_{\tX} g \,\PMA(\phi_j) 
=\int_{\tX} g \,\PMA(\gamma) +  \lim_j  \int_{\tX} (\phi_j-\gamma) \PMA(g) $$
by Lemma \ref{lem:integral-formulae}. As $\PMA$ agrees with the Laplacian on piecewise affine functions, we have
$$\int_{\tX} g \,\PMA(\phi) 
= \int_{\tX} g \,\Delta \gamma + \lim_j \int_{\tX} (\phi_j-\gamma) \Delta g 
= \int_{\tX} g \,\Delta \gamma + \int_{\tX} (\phi-\gamma) \Delta g$$
since the $(\phi_j)$ converge to $\phi$ uniformly. Since $g$ is bounded and piecewise affine, $g' \equiv 0$ outside of a compact subset. Moreover, $(\phi-\gamma)$ is a bounded convex function on each unbounded edge of a polyhedral structure $\Pi$ on $\tX$ where $\gamma$, $\phi$ (and $g$) are defined. Thus, $(\phi-\gamma)$ has non-positive and increasing derivative, which must then go to zero at infinity on each unbounded edge of $\Pi$. We apply Lemma \ref{lem:intlapl} to $g$ and $\phi-\gamma$, and conclude that
$$\int_{\tX} g \,\PMA(\phi) 
= \int_{\tX} g \,\Delta \gamma + \int_{\tX} g \,\Delta(\phi-\gamma) = \int_{\tX} g\, \Delta \phi.$$
\end{proof}

\begin{lemma} \label{lem:intlapl}
Let $f, g \in \operatorname{PS}(\tX) \cap \mathcal{C}^0_b(\tX)$ be defined on a polyhedral structure $\Pi$ on $\tX$. If $f'_{| e}$ and $g'_{| e}$ go to zero at infinity on each unbounded edge $e \in \Pi$, we have the equality
$$\int_{\tX} f \,\Delta g = \int_{\tX} g\, \Delta f.$$
\end{lemma}
\begin{proof}
Let $\Pi$ be a polyhedral structure on $\tX$ where both $f$ and $g$ are defined. For any edge $e\in \Pi^{(1)}$, by classical integration by parts, we have
$$\int_{e} f(t_e) g''(t_e) d\ell_e - \int_{e} f''(t_e) g(t_e) d \ell_e = \Bigg[ f(t_e) g'(t_e)-f'(t_e)g(t_e) \Bigg]_{e_0}^{e_1}$$
where $e_i$ are the endpoints of $e$, and $e_1=\infty$ if $e$ is unbounded. It follows that 
$$\int_{e} f(t_e) g''(t_e) d\ell_e + f(e_0) g'(e_0) - f(e_1)g'(e_1)= \int_{e} f''(t_e) g(t_e) d \ell_e + f'(e_0)g(e_0) - f'(e_1)g(e_1).$$
Multiplying the left-hand side by $[\tX](e)$ and summing over all the edges of $\Pi$ give 
\begin{align*}
   \sum_{e\in \Pi^{(1)}}&  [\tX](e)  \Big(\int_{e} f(t_e) g''(t_e) d\ell_e + f(e_0) g'(e_0) - f(e_1)g'(e_1) \Big) \\
   & = \sum_{e\in \Pi^{(1)}} \int_{e} f(t_e) g''(t_e) d\lambda_e + \sum_{e\in \Pi^{(1)}} [\tX](e) f(e_0) s_{e/e_0}(g)+ \sum_{e\in \Pi^{(1)}} [\tX](e) f(e_1)s_{e/e_1}(g) = \Delta g \,(f)
\end{align*}
as by assumption we have $s_{e/e_1}(g)=0$ for any unbounded edge. A similar computation gives that the right-hand side is equal to $\Delta f (g)$, which concludes the proof.
\end{proof}

As we shall see in Section \ref{sec:counter-example}, the orthogonality property does not hold for arbitrary balanced graphs. We will however show that it holds assuming a smoothness condition, which we now define.
\begin{defi}\label{def:smooth}
Let $\tX$ be one-dimensional and let $\Pi$ be a polyhedral structure on $\tX$ where $[\tX]$ is defined. We say that 
\begin{itemize}
    \item a vertex $v$ of $\Pi$ is \emph{smooth} if $\on{Span}_{\Z}\left(n_{\tau/v}\; | \; v \prec \tau\right)$ is a saturated lattice of rank $\on{val}(v)-1$ in $N $, where $\on{val}(v)$ is the number of one-dimensional faces adjacent to $v$. A vertex that is not smooth is called \emph{singular};
    \item the space $\tX$ is \emph{polyhedrally smooth} if, for any polyhedral structure $\Pi$ on $\tX$ where $[\tX]$ is defined, all vertices of $\Pi$ are smooth.
\end{itemize}
\end{defi}

\begin{rem}
The above definition is more general than the one in \cite[Definition 2.6]{Jell}, where the weights of $[\tX]$ are required to be all equal to one. Such polyhedral spaces are known as \emph{smooth} in the literature and can be thought as being locally (in a suitable sense) Bergman fans of matroids. Figure \ref{fig:smooth} shows a one-dimensional, polyhedrally smooth, balanced polyhedral space which is not smooth. Indeed, for the depicted polyhedral structure, one has to put weights $3$ on the unbounded faces, and $1$ on the bounded ones. 

Higher dimensional generalizations of the smoothness notion by Jell can be found for example in \cite[Definition~11]{smooth-higher}. It would be interesting to generalize \emph{polyhedrally smoothness} to higher dimensions and compare it to other smoothness conditions on tropical varieties, such as the notion of \emph{homologically smooth tropical fans} in \cite{hom-smooth}.  
\end{rem}
\begin{center}
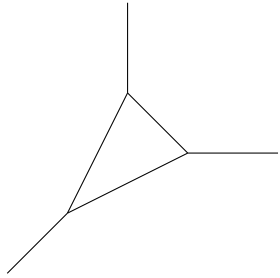
\begin{figure}[h]
\begin{tikzpicture}[scale=0.8]
\draw (-1, -1) -- (1, 0);
\draw (1, 0) -- (0, 1);
\draw (0,1) -- (-1,-1);
\draw (-1,-1)--(-2,-2);
\draw (1,0) -- (2.5,0);
\draw (0,1) -- (0,2.5);
\end{tikzpicture}\caption{A polyhedrally smooth polyhedral space which is not smooth}\label{fig:smooth}
\end{figure}
\end{center}

\begin{prop} \label{prop:dim1 ortho}
If a balanced polyhedral space is one-dimensional and polyhedrally smooth, then the orthogonality property holds.
\end{prop} 
\begin{proof} 
Let $u$ be a bounded $\PA$ function on $\tX$ defined on a polyhedral structure $\Pi$. By Proposition \ref{prop:env PA} $P_\gamma(u)$ is piecewise affine and defined on $\Pi$ as well, and the measure $\PMA(P_\gamma(u))$ is supported on the vertices of $\Pi$. 
Let $v$ be a vertex of $\Pi$ with $\PMA(P_\gamma(u))(v)>0$, and assume by contradiction that $P_\gamma(u)(v)=0$ and $\eps \coloneqq(\gamma+u)(v)>0$. 

Let $(n+1)$ be the valency of $\tX$ at $v$, and denote $\tau_0,\ldots,\tau_n$ the one-dimensional faces adjacent to $v$. By the polyhedrally smoothness assumption, $\tX$ is isomorphic to a tropical line in $\R^n$ in the neighbourhood of $v$; we label $e_i \coloneqq n_{\tau_i/v}$ the normal vectors at $v$, the balancing condition $a_i \coloneqq [\tX](e_i)$ along $e_i$, and the relation $\sum_{i=0}^n a_i e_i =0$ holds. We claim that, after adding an ambient affine function, the outgoing slope $P_\gamma(u)$ at $v$ are all strictly positive. Indeed, let $(s_0,\ldots,s_n)$ be the slopes of $P_\gamma(u)$ at $v$ along $e_i$; they satisfy the inequality $\sum_{i=0}^n a_i s_i = \PMA(P_\gamma(u))(v)>0$. We want to show that there exists a linear function $l$ such that $(a_0(s_0 + l_0),\ldots, a_n(s_n+ l_n))$ lies in $(\R_{>0})^{n+1}$, where $l_j\coloneqq l(e_j)$. Set
$$
l_j= \frac{\sum_{i=0}^{n}a_i s_i}{a_j(n+1)} -s_j,
$$
then clearly $a_j(l_j+s_j)>0$ for any $j$, and $l$ is a linear function on $\R^n$ as $\sum_{i=0}^{n}a_i l(e_i)=0$.

By the above claim, we can assume that $s_i >0$ for any $i=0,\ldots,n$. We define $Q$ to be the $\PA$ function on $\tX$ which takes value $\tfrac{\eps}{2}$ at $v$, and has slopes $\tilde{s_i}$ along $e_i$ near $v$, with $s_i >\tilde{s}_i>0$ sufficiently close to $s_i$ so that $Q \le (\gamma+u)$ and $Q$ agrees with $P_\gamma(u)$ at a point $w_i$ on $\tau_i$; after $w_i$, $Q$ is set to be equal to $P_\gamma(u)$. By construction, $Q$ is a competitor for the envelope of $u$, but $\tfrac{\eps}{2}=Q(v)>P_\gamma(u)(v)=0$, a contradiction.
\end{proof}

\begin{theorem} \label{theo:soldim1}
Let $\tX$ be a connected, one-dimensional, polyhedrally smooth, balanced polyhedral space. For any compactly supported Radon measure $\mu$ on $\tX$, there exists a unique (up to additive constant) function $\phi \in \PSH(\tX,\gamma)$ solving the Monge--Ampère equation $\PMA(\phi) = \mu.$ Moreover if $\gamma$ is strictly convex, then $\phi \in \PCreg(\tX,\gamma)$. 
\end{theorem}
\begin{proof}
The existence of a solution in $\PSH(\tX,\gamma)$ follows from Propositions \ref{prop:MA maxsol}, \ref{prop:env PA}, and \ref{prop:dim1 ortho}; moreover, if $\gamma$ is strictly convex, the solution is in $\PCreg(\tX,\gamma)=\PC(\tX,\gamma)$ by Proposition \ref{prop:boundedsol}. Since $\tX$ is one-dimensional, the operator $\PMA$ is linear, thus to prove uniqueness it is enough to show that if $f \in \mathcal{C}^0_b(\tX)$ satisfies $\PMA(f) =0$, then $f$ is constant.

 Let $\Gamma$ be a polyhedral structure on $\tX$ and $\PMA(f) =0$ for some $f \in \mathcal{C}^0_b(\tX)$. Then, on the interior of any edge of $\Gamma$, we have $f'' \equiv 0$ (in the sense of distributions) by comparison with the real Monge--Ampère measure (Proposition \ref{prop:comp MA}), so that $f$ is affine on any edge of $\Gamma$. It follows that $f \in \PA_b(\Gamma)$ . Let $p$ be any point where $f$ achieves its minimum $m$. Since $p$ is a minimum of $f$, each outgoing slope of $f$ at $p$ must be nonnegative, but since $\PMA(f)(p) =0$ the outgoing slopes of $f$ at $p$ add up to zero, and hence are all zero. It follows that $f$ is constant on a neighbourhood of $p$, hence $\{ f= m \}$ is both open and closed in $\tX$. Since $\tX$ is connected, we conclude that $f$ must be constant, as was to be shown.
\end{proof}

\subsection{Counterexamples to the orthogonality property}\label{sec:counter-example} 

We give counterexamples to the orthogonality property (Definition \ref{defi:orthogonality property}) in the one-dimensional case when the polyhedrally smoothness assumption doesn't hold (see Definition \ref{def:smooth} and Proposition \ref{prop:dim1 ortho}), and in higher dimension in the case of a Bergman fan of a matroid when $\gamma$ is not strictly convex.

\begin{example}\label{ex:smoothness-necessary} Let $\tX \subset \R^2$ be the union of the $x$-axis and the $y$-axis. Let $\gamma$ be the restriction to $\tX$ of the function $|x|+ |y|$. For $\eps>0$, set 
$$
\phi(x,y)= \begin{cases}
    \lvert x \rvert + \eps & \text{on the $x$-axis}\\
    \eps - \lvert y \rvert & \text{for $-\eps \leq y \leq \eps$} \\
    \lvert y - \eps \rvert & \text{otherwise}.
\end{cases}
$$
Because of the $y$-axis the envelope of $u \coloneqq \phi-\gamma$ must satisfy $P_\gamma(u)(0,0) \le 0$, from which we can easily check that $P_\gamma(u)$ is $\lvert x \rvert$ on the $x$-axis, and $P_\gamma(u)= \max \{ 0, \lvert y \rvert -\eps \}$ on the $y$-axis. As a result $\PMA(P_\gamma(u)))(0,0)=2 \delta_{(0,0)}$, while $0=P_\gamma(u)(0,0) < \phi(0,0)=\eps$, thus the orthogonality property doesn't hold for $u$. 
\end{example}

\begin{example}\label{exa:counter1} Let $M$ be the uniform matroid $U_{3,4}$ of rank $3$ with ground set $\{1,2,3,4\}$. The Bergman fan $\Sigma_M \subseteq \R^3$ of $M$, in the sense of Ardila and Klivans \cite{zbMATH02244752}, is a refinement of the 2-skeleton of the fan $\Sigma_{\PP^3}$ of $\PP^3$. In particular, let $e_i$ for $i=1,2,3$ and $e_4 \coloneqq -\sum_i e_i$ be the generators of the rays of $\Sigma_{\PP^3}$, and denote by $\sigma_{ij} \coloneqq \R_{\geq 0}e_i + \R_{\geq 0}e_j$. The fan $\Sigma_M$ refines $\Sigma_{\PP^3}$ by adding the ray $e_{ij}=e_i+e_j$ in each $\sigma_{ij}$. The space $\tX \coloneqq |\Sigma_M|$ is balanced by assigning weight one to all maximal dimensional faces of $\Sigma_M$. 

Let $\gamma$ be the piecewise linear $\PC$ function on $\tX$ such that $\gamma$ takes value $1$ at any primitive generator of a one-dimensional face of $\Sigma_M$, except for $e_{12}$ and $e_{34}$ where $\gamma(e_{12})=\gamma(e_{34})=0$. 

Let $\sigma'_{12}$ be the simplicial refinement of $\sigma_{12} \simeq (\R_{\geq 0})^2$ given by adding the vertices $(0, \eps)$, $(\eps, 0)$, $(\eps, \eps)$, and the half-lines $(0, \eps) +\R_{\ge 0} (1,0)$, $(0,0) + \R_{\ge 0} (1,1)$, and $(\eps,0) +\R_{\ge 0} (0,1)$. We define the function $g$ to be affine on each face of $\sigma'_{12}$, with values $\eps$ at $(0,0)$, $2 \eps$ at $(0, \eps)$ and $(\eps, 0)$, 0 at $(\eps, \eps)$; the slope of $g$ along vertical and horizontal unbounded face is $1$, while the slope along the $(1,1)$-direction ray is zero. On $\sigma_{34}$, the function $g$ is defined similarly; on every other face of $\Sigma_M$, set $g=\gamma+\eps$. In particular, $g$ is continuous on $\tX$, $\gamma \leq g$, and $u \coloneqq g-\gamma \in \PA_b(\tX)$.

\begin{claim}
$P_\gamma(u)=\gamma$.
\end{claim}
\begin{proof}
Since $\gamma \le g$ and $\gamma$ is PC, we have $P_\gamma(u) \ge \gamma$ by definition of the envelope.

Since $P_\gamma(u) \le g$, we get $P_\gamma(g) \le 0$ on the unbounded faces $(\eps, \eps) + \R_{\ge 0} (1,1)$ both in $\sigma_{12}$ and in $\sigma_{34}$. Since $\gamma \leq P_\gamma(u)$, we get $P_\gamma(u)=0$ on the line $\R e_{12} = \R e_{34}$, and in particular $P_\gamma(u) (0,0,0) =0$. 
It follows that on $\R_{\ge 0} e_1$, $P_\gamma(u)$ is a convex function, vanishing at $0$ and such that $x \le P_\gamma(u)(x) \le x+ \eps$, where $x$ is the standard coordinate on $\R_{\ge 0} e_1$. This yields $P_\gamma(u)(x) = x = \gamma(x)$ on $\R_{\ge 0} e_1$, and similarly on $\R_{\ge 0} e_2$.

We now prove $P_\gamma(u)=\gamma$ on $\sigma_{12}$. Indeed, the equality holds on the boundary $\partial \sigma_{12}$ of $\sigma_{12}$ and the half line $\ell = \R_{\ge} (1,1)$. Since any point in $\relint(\sigma_{12})$ lies in a segment with one endpoint in $\partial \sigma$ and one in $\ell$, the convexity inequality applied on that segment then yields $P_\gamma(u) \le \gamma$, since $\gamma$ is affine on the two cones delimited by $\ell$. Similarly, we get $P_\gamma(u)=\gamma$ on $\sigma_{34}$.

Consider now the cone $\sigma_{23}$. The equality $P_\gamma(u)=\gamma$ holds on the boundary of $\sigma_{23}$, so arguing as above we need to show it on the half-line $\ell'=\R_{\geq 0}e_{23}$. Writing $y$ the integral coordinate on $\ell'$, we have once again that $\gamma(y)= y \le P_\gamma(u)(y) \le y + \eps = g(y)$ while $P_\gamma(u)(0)=0$, so that $P_\gamma(u)(y) = y$ on the diagonal, and we get $P_\gamma(u)=\gamma$ on $\sigma_{23}$. The same holds for the other 2-dimensional cones of $\Sigma_{\PP^3}$, and this concludes the proof.
\end{proof}

As $\PMA(\gamma)(0,0)=4 \delta_{(0,0,0)}$, while $0=\gamma(0,0,0) < g(0,0,0)=\eps$, the orthogonality property doesn't hold for $u$. 
\end{example}
\begin{rem} \label{rem:counterexpos}
The function $\gamma$ in Example \ref{exa:counter1} is identically zero on the line generated by $e_{12}$, so that $\gamma$ is not strictly convex on $\Sigma_M$ (or any refinement). This example shows that some strict positivity condition should be imposed on $\gamma$ in order to expect orthogonality to hold.
\end{rem}
\begin{rem}\label{rem:matroids}
Example \ref{exa:counter1} can be modified to obtain a conical example, thereby producing a counterexample to the orthogonality property for the Bergman fan of a matroid. Since Bergman fans of matroids are among the simplest and most fundamental examples of polyhedral spaces, this naturally leads to the question of determining for which classes of matroids the associated Bergman fan satisfies the orthogonality property.
\end{rem}

\section{Connection with the non-archimedean Monge-Ampère equation}\label{sec:NA}
The main result of this section is Proposition \ref{prop:solMANA}, in which we relate the solution of the polyhedral Monge--Amp\`ere equation to that of the corresponding non-archimedean Monge--Amp\`ere equation in the case where $\tX$ is the tropicalization of a toric degeneration of Calabi--Yau complete intersections. To this end, we first briefly recall the relevant notation and some results on non-archimedean Monge--Amp\`ere equations and toric degenerations of complete intersections.

\subsection{Non-archimedean approach to the SYZ conjecture}
In this section, we briefly recall the non-archimedean approach to the SYZ conjecture, and refer the reader to \cite{LiNA} for more details.

Let $\pi \colon X \to \mathbb{D}^*$ be a projective meromorphic degeneration of $d$-dimensional Calabi--Yau manifolds, over the complex punctured disc $ \mathbb{D}^*$. Setting $K=\C((t))$, the space $X$ can be viewed as an algebraic variety over $K$. One can then associate to it its Berkovich analytification $X^{\an}$ together with a canonical simplicial subset $\Sk(X) \subset X^{\an}$, known as the \emph{essential skeleton} of $X$. For background on Berkovich non-archimedean geometry, we refer to \cite{Berkovich1990}, and for the construction and properties of the essential skeleton to \cite{MustataNicaise, NicaiseXu}.

Fix a relatively ample line bundle $L$ on $X$. By Yau’s solution of the Calabi conjecture \cite{yau}, each fiber $X_t=\pi^{-1}(t)$ carries a unique Ricci-flat K\"ahler metric in the class $c_1(L_{|X_t})$, obtained as the solution of a complex Monge--Ampère equation. When the family is maximally degenerate, i.e. when $\dim_{\R}\Sk(X)=d$, it is expected that the asymptotic behaviour of these metrics can be described via non-archimedean geometry, as conjectured by Kontsevich and Soibelman \cite{KontsevichSoibelman}.

More precisely, the volume forms associated to the Ricci-flat K\"ahler metrics converge, in a suitable sense, to a limit measure $\mu_0$ on $X^{\an}$ \cite[Theorem A]{BoucksomJonsson2017}, which is a Lebesgue measure supported on the essential skeleton $\Sk(X)$ and of total mass $(L^d)$.
It is proved in \cite{BoucksomFavreJonsson2015} that there exists a unique semi-positive metric $\Psi$ on the analytification of $L$ solving the non-archimedean Monge--Ampère equation
$\MA_{\NA}(\Psi)=\mu_0.$
This metric is expected to encode the asymptotic behaviour of the Ricci-flat K\"ahler metrics on the fibers $X_t$ as $t\to 0$.

Following \cite[Definition 3.3]{LiNA}, we say that $(X,L)$ satisfies the \emph{comparison property} if there exists a model $(\cX,\mathscr{L})$ of $(X,L)$ over $\C[[t]]$ such that, writing the solution as $\Psi=\phi_{\mathscr{L}}+\phi,$
the function $\phi$ satisfies 
\begin{equation}\label{eq:comparison}
\phi=\phi\circ\rho_{\cX}
\end{equation}
over the interior of the maximal faces of $\Sk(\cX)$. Here $\Sk(\cX)$ denotes the skeleton of the model $\cX$, and $\rho_{\cX}$ is the associated Berkovich retraction; we refer to \cite[\S 3.1]{NicaiseXu} for precise definitions.

Assuming the comparison property, the function $\phi$ can be related to the solution of a real Monge--Ampère equation on the essential skeleton $\Sk(X)$. Building on this relation, it is shown in \cite{LiNA} that the comparison property implies a metric version of the SYZ conjecture from mirror symmetry \cite{StromingerYauZaslow}; see \cite[Theorem 1.3]{LiNA} for a precise statement. In particular, this result reduces a problem in differential geometry to the analysis of a non-archimedean Monge–Ampère equation, thereby placing it within the framework of non-archimedean pluripotential theory.

\subsection{Toric degenerations of Calabi--Yau complete intersections}
Let $k$ be an algebraically closed field of characteristic zero, $K = k((t))$, and $R=k[[t]]$.
Let $X/K$ be a toric degeneration of a Calabi--Yau complete intersection as defined in \cite{Gross2005}; in particular, $X$ is a maximally degenerate family in the sense of the previous section. We will use the notation from \cite{Yamamoto2021,YamaSYZ} and refer to these for the precise definition of the degeneration.

Writing $d = \dim_K (X)$, the family $X$ is naturally embedded in a proper toric variety $X_{\Sigma'}$ of dimension $d+r$ for some $r \ge 1$, and is polarized by the ample line bundle $L$ associated to a convex conewise-linear function $h$ on $\Sigma'$. We write $\mathbb{T}$ for the $K$-torus of the ambient toric variety, and denote by $M$ and $N$ the associated character and cocharacter lattices, respectively. We also denote by $\trop \colon \mathbb{T}^{\an} \to N_{\R}$ the tropicalization map, given by minus the $\log$ of the modulus of the coordinates.

The tropicalization $\tX \coloneqq \trop(X^{\an} \cap \mathbb{T}^{\an}) \subset N_{\R}$ of $X$ is a $d$-dimensional balanced polyhedral space in the sense of Definition \ref{def:balanced-poly}; see for instance \cite[Theorems 3.3.5 and 3.4.14]{MaclaganSturmfels2015}. We endow it with the polyhedral structure $\Pi \coloneqq \mathscr{P}_{X(f_1,\ldots,f_r)}$ constructed in \cite[Section 5.2]{Yamamoto2021}. 
Let $\overline{\tX} \coloneqq \overline{\tX}_{\Sigma'}$ be the compactification of $\tX$ with respect to fan $\Sigma'$ (see Definition \ref{def:compactification}), that is denoted by $\trop(X)$ in \emph{loc.\,cit.}
We then denote by $\overline{\trop} \colon X^{\an} \to \overline{\tX}$ the extended tropicalization map (see \cite[§2.3]{YamaSYZ}), whose restriction to $\tX$ agrees with $\trop$.

Additionally, let $B\coloneqq \Sk(\Pi)$ be the union of the bounded faces of $\Pi$; this is denoted by $B^{\check{h}}_{\nabla}$ in \cite{Yamamoto2021}, see \cite[Proposition 5.7]{Yamamoto2021}. It follows from \cite[Theorem 1.1]{YamaSYZ} that the restriction of the tropicalization map to $\Sk(X)$ induces a homeomorphism between $\Sk(X) \subset X^{\an}$ and $B \subset \tX$. Recall from the previous section that $\Sk(X)$ carries a natural Lebesgue measure $\mu_0$ of total mass $(L^d)$, which we view as a measure on $\tX$ (supported on $B$) via the above homeomorphism.

Let $\phi_0$ be the semi-positive model metric on $(X^{\an}, L^{\an})$ induced by the restriction of the toric metric associated to the piecewise linear function $h$ (see \cite[Proposition 4.3.10]{BurgosPhilipponSombra}). By the proof of \cite[Lemma 4.3]{GotoYamamoto}, there exists a toric $R$-model $(X_{\tilde{\Sigma}'},  \mathscr{L}_0)$ of $(X_{\Sigma'} , L)$ such that writing $\mathcal{X}$ for the closure of $X$ in $X_{\tilde{\Sigma}'}$, the metric $\phi_0$ coincides with the model metric associated to the restriction of $\mathscr{L}_0$ to X. We also endow $\tX$ with the $\PAPC$ function $\gamma \coloneqq h|_{\tX}$.

We will furthermore use the minimal snc $R$-model $\cX$ of $X$ constructed in \cite[Proposition 5.13]{YamaSYZ} by repeatedly blowing-up $\mathcal{X}$, and denote by $\cX_s \coloneqq \cX \times_R k$ its special fiber. We endow it with the line bundle $\mathscr{L}$ obtained by pulling-back $\mathscr{L}_0$ via the blowup map; in particular, it holds that $\cL_{|X} =L$ and $\phi_0= \phi_{\mathscr{L}}$.
Writing $\Sk(\cX)$ its Berkovich skeleton, by \cite[Theorem 1.1]{YamaSYZ} we have that $\Sk(\cX) = \Sk(X)$ and the map
$\trop \colon \Sk(\cX) \to B$
is an integral piecewise affine isomorphism. Finally, the associated Berkovich retraction is denoted $\rho_{\cX} \colon X^{\an} \to \Sk(\cX)$.
 
\begin{prop}\label{prop:solMANA}
Let notations as above. Assume that 
\begin{itemize}
    \item $(L^d)=\deg(\gamma)$
    \item there exists $\phi \in \PCreg(\tX, \gamma)$ such that  $\PMA(\phi) = \mu_0$,
    \item the function $\phi_{\NA}\coloneqq (\phi-\gamma) \circ \overline{\trop}$ is such that the metric $(\phi_0 + \phi_{\NA})$ on $L^{\an}$ is semi-positive. 
\end{itemize}
Then the metric $\Psi \coloneqq \phi_0 +\phi_{\NA}$ is the (unique) solution to the non-archimedean Monge--Ampère equation
$$\NAMA(\Psi) = \mu_0$$
on $X^{\an}$.
Moreover, when $k= \C$, the comparison property \eqref{eq:comparison} holds for the degeneration $X$.
\end{prop}

\begin{proof}
By \cite[Lemma 4.2]{GotoYamamoto}, the equality $\trop = \trop \circ \rho_{\cX}$ holds over the interior of the maximal faces of $\Sk(\cX)$, and thus so does $\phi_{\NA} = \phi_{\NA} \circ \rho_{\cX}$. Fixing such a maximal face $\tau$, it follows from \cite[Theorem 1.1]{Vilsmeier2021} that the measure $\NAMA(\Psi)$ agrees with $d!\, \RMA(\phi_{\relint(\tau)})$ on $\rho_{\cX}^{-1}(\relint(\tau))$, where $\Psi \coloneqq \phi_0 +\phi_{\NA}$ and $\RMA(\phi_{\relint(\tau)})$ is viewed as a measure on $\rho_{\cX}^{-1}(\relint(\tau))$ by pushforward via the inclusion $\relint(\tau) \subset \rho_{\cX}^{-1}(\relint(\tau))$; note that, since $k$ is algebraically closed, the factor $[\tilde{K}(S):\tilde{K}]$ appearing in \emph{loc.\,cit.\,}is equal to $1$. 

On the other hand, Proposition \ref{prop:comp MA} and Lemma \ref{lem:weightone} below give the equalities 
$$ \mathbf{1}_{\relint(\tau)}  \PMA(\phi) =d! \,([\tX](\tau))\MA_{\R}(\phi_{| \relint(\tau)}) = d! \MA_{\R}(\phi_{|\relint(\tau)}),$$ so that $\NAMA(\Psi)$ agrees with $\mu_0$ away from the preimage by $\rho_{\cX}$ of the lower-dimensional faces of $\Sk(\cX)$. As both $\NAMA(\Psi)$ and $\mu_0$ have total mass $(L^d)$ and the measure of lower-dimensional faces of $\Sk(\cX)$ with respect to $\mu_0$ is zero, it follows that they agree everywhere. This proves that $\Psi$ is the unique solution to the non-archimedean Monge--Ampère equation $\NAMA(\Psi) = \mu_0$.

Finally, we have shown above that $\phi_{\NA} = \phi_{\NA} \circ \rho_{\cX}$ holds over the interior of maximal faces of $\Sk(\cX)$, so that the comparison property \eqref{eq:comparison} holds for the model $(\cX, \mathscr{L})$ as $\phi_0 = \phi_{\mathscr{L}}$.
\end{proof}

\begin{lemma}\label{lem:weightone}
Let $\tau$ be a maximal face of $B \subset \tX$. Then $[\tX](\tau) =1$.
\end{lemma}
\begin{proof}
By \cite[Section 5.5]{YamaSYZ}, the polyhedral structure $\trop(\Sk(\cX))$ is a refinement of $\Sk(\Pi)$. As a result, it suffices to show that for any maximal face $\sigma \in \Sk(\cX)$, the weight $[\tX](\trop(\sigma))$ is equal to $1$. To this end, we rely on the description of the weights provided in \cite{CLD,GublerCLD}, which we now recall.

Let $\mathbb{T}'$ be a split $K$-torus of dimension $d$ with cocharacter lattice $N'$, together with a surjective morphism $q \colon \mathbb{T} \to \mathbb{T}'$ such that the corresponding linear map $F \colon N_{\R} \to N'_{\R}$ is injective when restricted to $\trop(\sigma)$. It follows from \cite[Proposition 14.2.2]{CLD} that the restriction of $q^{\an}$ to $\trop^{-1}(\trop(\relint(\sigma))) =\rho_{\cX}^{-1}(\relint(\sigma)) \subset X^{\an} \cap \mathbb{T}^{\an}$ is flat and finite, hence we denote by $d(\sigma)$ its degree. 
By \cite[Proposition 7.11]{GublerCLD}, the weight $[\tX](\trop(\sigma))$ can be computed using the following formula from \cite[Definition 7.5]{GublerCLD}:
\begin{equation}\label{eq:weight}
[\tX](\trop(\sigma)) =d(\sigma) \cdot [M_{\trop(\sigma)}:M']^{-1},
\end{equation}
where $M_{\trop(\sigma)}$ is the dual lattice defined in Definition \ref{defi:polyhedral}, $M'$ is the character lattice of $\mathbb{T}'$, and $[M_{\trop(\sigma)}:M']$ its index in $M_{\trop(\sigma)}$. 
We will now construct suitable $\mathbb{T}'$ and $q$ as above, such that $d(\sigma) = [M_{\trop(\sigma)}:M']=1$.

Let $\sigma$ be a maximal face of $\Sk(\cX)$, and let $p \in \cX_s$ the corresponding closed point in the special fiber. Let $D_1,\ldots,D_{d+1}$ be the irreducible components of $\cX_s$ containing $p$. By the proof of \cite[Lemma 4.2]{GotoYamamoto}, there exists an open neighbourhood $\mathscr{U}$ of $p$ in $\cX$ given as follows:
$$\cU = \Spec R \left[z^{m_l} \colon 1 \le l \le d+r+1 \right] \bigg/ \left( t - \prod_{l=1}^{d+1} z^{m_l}, -z^{m_{d+1+i}} +f_i \prod_{l = l_i}^{d+1} z^{m_l} (1 \le i \le r) \right),$$
where $\left\{z^{m_l}\,|\, 1 \le l \le d+r+1\right\}$ is a basis of $M\oplus \Z$, each $l_i \le d+1$, $f_i \in k[z^{m_l}\,|\, 1 \le l \le d+r+1]$, and in particular for $l =1,\ldots,d+1$, the monomial $z^{m_l}$ is an equation for the irreducible component $D_l \subset \cX_s$ on $\cU$. Write $U \subset \mathbb{T}$ for the generic fiber of $\cU \cap \mathbb{T}_R$; this embedding into $\mathbb{T}$ is the restriction to $U$ of the embedding of $X \cap \mathbb{T}$ inside $\mathbb{\T}$,
as the model $\cX$ is constructed by repeatedly blowing-up the closure of $X$ in a toric $R$-model, along closed subvarieties of the special fiber; see 
\cite[§ 5.3]{YamaSYZ}.

Define $M' \subset M \oplus \mathbb{Z}$ to be the sublattice generated by $m_1,\ldots,m_{d+1}$, and the toric $R$-scheme
$$\mathscr{T}' \coloneqq \Spec R\left[M'\right] \bigg/ \left( \prod_{i=1}^{d+1} z^{m_{i}} =t \right),$$
whose generic fiber $\mathbb{T}'$ is a $d$-dimensional $K$-torus; we denote its cocharacter lattice by $N'$. The inclusion $M' \subset M \oplus \Z$ gives rise to a surjective homomorphism of $K$-tori $q \colon \mathbb{T} \to \mathbb{T}'$, and let $F \colon N_{\R} \to N'_{\R}$ be the induced linear projection. 
We claim that the restriction of $F$ to $\trop(\sigma)$ is injective. Indeed, for $x \in U^{\an}$ and $m \in M$, we have that $\langle m, \trop(x) \rangle = -\log \lvert z^{m} \rvert_x$, and by \cite[(4.16), proof  of Lemma 4.2]{GotoYamamoto}, the equality
$$\log \lvert z^{m_{d+1+i}} \rvert_x = \sum_{l = l_i}^{d+1} \log \lvert z^{m_{l}} \rvert_x   $$
holds for all $i \ge 1$ when $x \in \sigma$. It follows that the point $\trop(x) \in N_{\R}$ is uniquely determined by the $\langle m, \trop(x) \rangle$ for $m \in M'$, which proves the claim.

We will now compute $[\tX](\trop(\sigma))$ using  \eqref{eq:weight}. Define a morphism $f\colon \cU\to \mathscr{T}'$ at the level of rings, by sending the monomial $\chi^{m_i} \in K[M']$ to $z^{m_i}$. Its restriction to the generic fiber coincides with the restriction $q_{| U}$.
\\Since $\cU$ is snc at $p$ and $z^{m_i}$ is an equation for $D_i$ at $p$, we infer that $f$ induces an isomorphism between the formal schemes $\widehat{\cU}_{/p}$ and $\widehat{\mathscr{T}}_{/ f(p)}$ 
-- here $\widehat{\cU}_{/p}$ means the formal completion of $\cU$ along the closed subscheme $p$. Passing to the generic fiber in the sense of Berkovich, we get that the restriction of $f^{\an}$ to $\rho_{\cX}^{-1}(\relint(\sigma)) $ is an isomorphism onto its image inside $(\mathbb{T}')^{\an}$, so that the degree $d(\sigma)$ of $q$ must be $1$. Since $X^{\an}$ is algebraic, the multiplicity $[\tX]\left(\trop(\sigma)\right)$ must be an integer by \cite[Proposition 7.11]{GublerCLD}, which forces $[M_{\trop(\sigma)}:M']=1$ and thus $[\tX](\trop(\sigma))=1$, as was to be shown.
\end{proof}

\subsection{Example} \label{ex:elliptic}
We review \cite[Example 2.34]{Li2024} from the perspective of this article. 

Let $Z_{\Delta}$ be the toric Fano surface obtained as blowup of $\Pro^1 \times \Pro^1$ at two toric points, and associated to the polytope $\Delta = \conv(\pm (1, -1) ; \pm (1, 0) ; \pm (0,1))$. We denote by $\Sigma \subset N_{\R} = \R^2$ the normal fan of $\Delta$; see Figure \ref{fig:exa}.
Consider a one-parameter family $X$ of Calabi--Yau hypersurfaces in $Z_\Delta$ degenerating to the toric boundary of $Z_\Delta$. In particular, this can be regarded as a family of elliptic curves over $\Spec \mathbb{C}((t))$. We denote by $\tX$ its tropicalization: it consists of the boundary of the dual polytope $\Delta^{\vee}$ and six unbounded edges; we denote by $\Pi$ this polyhedral structure on $\tX$.
\begin{center}
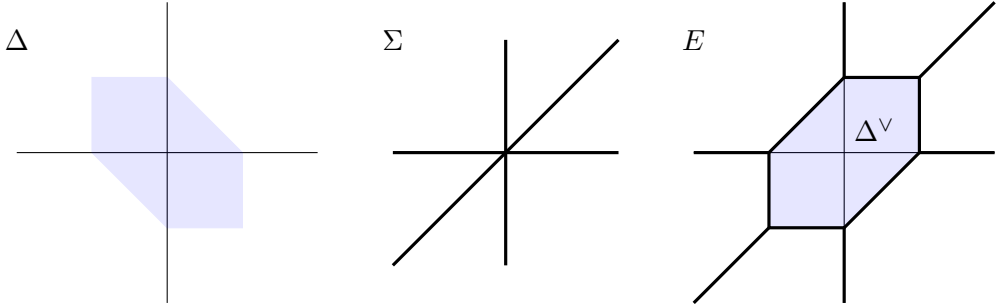
\begin{figure}[h]
\begin{tikzpicture}
\draw (-2, 0) -- (2, 0);
\draw (0, -2) -- (0, 2);
\filldraw[blue, opacity=0.1] (-1, 0) to (0,-1) to (1, -1) to (1, 0) to (0,1) to (-1, 1) to cycle;
\draw (-2, 1.5) node{$\Delta$};
\draw[very thick] (3,0) to (6,0);
\draw[very thick] (4.5,-1.5) to (4.5,1.5);
\draw[very thick] (3,-1.5) to (6,1.5);
\draw (3, 1.5) node{$\Sigma$};
\draw (7, 0) -- (11, 0);
\draw (9, -2) -- (9, 2);
\filldraw[blue, opacity=0.1] (8, 0) to (8,-1) to (9, -1) to (10, 0) to (10,1) to (9, 1) to cycle;
\draw[very thick] (8, 0) to (8,-1) to (9, -1) to (10, 0) to (10,1) to (9, 1) to cycle;
\draw (7, 1.5) node{$E$};
\draw(9.4,0.3) node{$\Delta^{\vee}$};
\draw[very thick] (10,0) to (11,0);
\draw[very thick] (10,1) to (11,2);
\draw[very thick] (9,1) to (9,2);
\draw[very thick] (8,0) to (7,0);
\draw[very thick] (8,-1) to (7,-2);
\draw[very thick] (9,-1) to (9,-2);
\end{tikzpicture}\caption{The polytope $\Delta$, its normal fan $\Sigma$ and the tropicalization $\tX$}\label{fig:exa}
\end{figure}
\end{center}
Write $D$ the toric boundary of $\Pro^1 \times \Pro^1$, $\pi \colon Z_{\Delta} \to \Pro^1 \times \Pro^1$ the blowup, and $E_1$, $E_2$ the exceptional divisors. For $0<\eps <2$, set 
$$D_{\eps} \coloneqq \pi^* D- \eps(E_1+E_2).$$
As $D_{\eps}$ is an ample toric divisor on $Z_{\Delta}$, it is associated to a strictly convex conewise-affine function on the fan of $Z_{\Delta}$, which we denote by $\gamma_{\eps}$. We also denote its restriction to $\tX$ by $\gamma_{\eps}$, which is described explicitly as follows: $\gamma_{\eps}$ is a function in $\PA(\Pi)$, with values at the vertices $\gamma_{\eps}((\pm 1, 0)) = \gamma_{\eps} ((0, \pm 1 )) =1$, $\gamma_{\eps}(\pm 1, \pm 1)) =(2-\eps)$, and $\gamma_{\eps}$ is affine with outgoing slope $(2-\eps) \geq 0$ on the rays starting at $\pm(1,1)$, and affine with outgoing slope $1$ on the four other rays.

Let $d\ell$ be the Lebesgue measure on $\partial \Delta^{\vee}$, normalized such that each edge has mass $\frac{4-\eps}{3}$. The total mass is then $ d \ell(\tX) =8-2 \eps= \deg_{\tX}(\gamma_{\eps})= \deg_X( \mathcal{O}_X(D_{\eps}))$. 
In \cite{Li2024}, the author uses an approach inspired from optimal transport to solve a real Monge--Ampère equation on the boundary of $\Delta^{\vee}$, and then describes the solution to the non-archimedean Monge--Ampère equation $\NAMA(\cdot)=d \ell$. It is shown that this approach does not work for every choice of polarization; on the other hand we are able to prove that the solution to the polyhedral Monge--Ampère equation always exists and recovers the solution to the non-archimedean Monge--Ampère equation. In particular, this implies the comparison property as in Proposition \ref{prop:solMANA}; see Proposition \ref{prop:exdim1} below.

\begin{lemma} \label{lem:equdim1}
Let $L_{\eps}$ be the restriction to $X$ of the toric line bundle $\mathcal{O}_{Z_{\Delta}}(D_{\eps})$, and let $\theta_{\eps}$ be the restriction to $L^{\an}$ of the toric metric induced by $\gamma_{\eps}$. Under the identification $\Sk(X) = \Sk(\Pi)$, the following equality of measures holds:
$$\MA_{\NA}(\theta_{\eps}) = \PMA(\gamma_{\eps}).$$
\end{lemma}
\begin{proof}
The measure $\PMA(\gamma_{\eps})$ is straightforward to compute: it is the atomic measure putting mass $(2-\eps)$ at $(\pm 1, 0)$ and $(0, \pm 1)$, and mass $\eps$ at $\pm (1,1)$.

Let $\cX$ be the closure of $X$ in the $R$-model $Z_{\Delta}$: it is a minimal snc model, whose special fiber is the toric boundary of $Z_{\Delta, k}$, and such that the dual graph of the special fiber is $\Sk(X)$. Writing $\cL_{\eps}$ the canonical model of $L_{\eps}$, it holds that $\theta_{\eps}$ is the model metric associated to $\cL_{\eps}$. As a result, by definition of the non-archimedean Monge--Ampère measure, we need to show that for any irreducible component $C \subset \cX_s$ corresponding to a vertex $v \in \Sk(X)$, the following equality holds:
$$\cL_{\eps} \cdot C =  \PMA(\gamma_{\eps})(v).$$ The left-hand side agrees with the degree $\deg_C(\cL_{\eps})$, which by flatness can be computed on the generic fiber $Z_{\Delta}$, hence agrees with $D_{\eps} \cdot \tilde{C}$, where $\tilde{C}/K$ is the corresponding toric curve in $Z_{\Delta}$. If $C$ is the strict transform of a boundary component of $\Pro^1 \times \Pro^1$, we then have $D_{\eps} \cdot \tilde{C} = 2-\eps$, and if $C$ is an exceptional divisor we have $\tilde{C} \cdot D_{\eps} =\eps$, which agrees with the formula for $\PMA(\gamma_{\eps})$.
\end{proof}

Consider the solution $\phi$ to the polyhedral Monge--Ampère equation
$$\PMA(\phi) = d \ell$$
with $\phi \in \PC^{\on{reg}}(\tX, \gamma_{\eps})$. We give an explicit description of $\phi$; we begin by normalizing the solution such that $ \phi(1,0) = 0$. 
By Proposition \ref{prop:comp MA}, on each edge of $\tX$, the function $\phi$ solves also the real Monge--Ampère equation with respect to $d \ell$, thus $\phi$ is affine on the unbounded faces of $\tX$ and quadratic on each bounded face - in particular $\phi$ is piecewise smooth on $\Pi$ in the sense of Definition \ref{defi:PS}. On an edge $e$ of $\partial \Delta^{\vee}$, the solution $\phi$ is then of the form
$$\phi_e(t) = \frac{(4-\eps)t^2}{6} + a_e t +b_e$$
where $t$ is the integral coordinate of $e$. Moreover, by Proposition \ref{prop:laplacian}, the value $\PMA(\phi)(v)$ at a vertex $v \in \Delta^{\vee}$ is the sum of the outgoing slopes of $\phi$ along the edges emanating from $v$. This implies that at each vertex of $\tX$, the sum of the outgoing slopes of $\phi$ is zero, and by explicit computation, we get  
$$\phi(\pm 1, 0) = \phi(0, \pm 1) =0,$$
$$\phi(\pm (1,1))=\frac{1-\eps}{3}.$$
Indeed, we define on the edge from $(0,1)$ to $(-1,0)$ with integral coordinate $t$, a function $$ \phi_1(t)= \frac{(4-\eps)(t^2-t)}{6}$$ (in particular is invariant under $t \mapsto (1-t)$), and on the edge $\{ (1, t) | t \in [0,1] \}$, we define $$ \phi_2(t)= \frac{(4-\eps)t^2}{6}- \frac{1}{3}\left(1+\frac{\eps}{2}\right)t;$$ the function on other edges is defined similarly using symmetries of $\tX$. We claim that the function $\phi$ defined to restrict to $\phi_1$ or $\phi_2$, and extended to unbounded edges such that it is affine with same slope as $\gamma$, is a solution to $\PMA(\phi) = d \ell$.

It is clear that $\PMA(\phi) =d \ell$ on each edge, so it remains to check that the sum of outgoing slopes at each vertex is zero, by symmetry we only treat $v = (0,1)$ and $w=(1,1)$.\\
At $v$, the outgoing slope in the unbounded direction is the one of $\gamma$, hence $1$. In the direction of $w$ it is $\phi_2'(0)=-( \frac{1}{3}+\frac{\eps}{6})$, and in the direction of $(-1, 0)$ it is $-\phi'_1(1) = -\frac{4-\eps}{6}$, these 3 slopes do sum up to zero.
\\At $w$, the slope in the unbounded direction is again the one of $\gamma$, i.e. $(2-\eps)$. In the $v$ direction it is $-\phi_2'(1)$, and also in the direction of $(1,0)$ by the symmetries. But $\phi_2'(1) =1-\frac{\eps}{2}$, so that indeed $2 \phi_2'(1) = (2-\eps)$ and the outgoing slopes at $w$ add up to zero. This proves $\PMA(\phi) = d \ell$.

\begin{prop} \label{prop:exdim1}
Let $\phi_{\NA}\coloneqq (\phi-\gamma_\eps) \circ \overline{\trop}$. The metric $\Psi = \theta_{\eps} + \phi_{\NA}$ is the solution to the non-archimedean Monge--Ampère equation
$$\MA_{\NA}(\Psi) = d \ell$$
on $X^{\an}$, and the comparison property holds for $(X, \mathcal{O}_X(D_{\eps}))$.
\end{prop}
\begin{proof}
By Proposition \ref{prop:solMANA} and Lemma \ref{lem:equdim1}, it suffices to show that the metric $(\theta_{\eps} +\phi_{\NA})$ on $(\mathcal{O}_X(D_{\eps}))^{\an}$ is semi-positive.

Set $\psi \coloneqq \phi-\gamma_\eps$. On every unbounded edge of $\Pi$, $\psi$ is a continuous convex function, and by Proposition \ref{prop:comp MA} the second derivative $\psi''$ vanishes. As a result, it must be constant on every unbounded face of $\Pi$, so that $\psi = \psi \circ r_{\Pi}$. Now let $(\phi_j)_j$ be a sequence in  $\PAPC(\tX,\gamma_\eps)$ converging to $\phi$ uniformly. Setting $\psi_j \coloneqq  (\phi_j-\gamma_\eps) \circ r_{\Pi}$, we have that $(\gamma_\eps+\psi_j) \in \PAPC(\tX,\gamma_\eps)$ and $\psi_j$ converges uniformly to $\psi$ on $\tX$. As a result, it suffices to show that $(\theta_{\eps}+\psi_j \circ \overline{\trop})$ is a semi-positive model metric for all $j$. 

By \cite[Lemma 5.25]{YamaSYZ} (note that the blowups performed to construct $\cX$ are the identity in this example) and under the identification $\lvert \Sk(X) \rvert = \lvert \Sk(\Pi) \rvert$, the Berkovich retraction $\rho_{\cX}$ agrees with the composition $r_{\Pi} \circ \overline{\trop}$, so that $\psi_j \circ \overline{\trop} = \psi_j \circ \rho_{\cX}$. It then follows from \cite[Theorem 3.2.10 (ii)]{Thuillier} and \cite[Theorem 3.28]{Wanner} that $\NAMA (\psi_j \circ \overline{\trop}) = i_*\Delta ((\psi_{j})_{| \Sk(\Pi)})$, where here the Laplacian is defined on the finite graph $\Sk(\Pi)$ (and does not see the unbounded edges), and $i \colon \Sk(X) \to X^{\an}$ is the embedding. Since $\psi_j$ is constant along the unbounded edges, we have $\Delta ((\psi_{j})_{| \Sk(\Pi)}) = \PMA(\psi_j)$ by Proposition \ref{prop:laplacian}. It follows from Lemma \ref{lem:equdim1} that
$$\NAMA(\theta_{\eps})+ \NAMA (\psi_j \circ \overline{\trop})=\PMA(\gamma_\eps+\psi_j) \ge 0,$$
so that $(\theta_{\eps}+\psi_j \circ \overline{\trop})$ is semi-positive, which concludes.
\end{proof}

\vspace{0.5cm}
\footnotesize \printbibliography
\end{document}

%% file: macros.tex
\newcommand{\N}{\ensuremath{\mathbb{N}}}
\newcommand{\Z}{\ensuremath{\mathbb{Z}}}

\newcommand{\Q}{\ensuremath{\mathbb{Q}}}

\newcommand{\R}{\ensuremath{\mathbb{R}}}
\newcommand{\C}{\ensuremath{\mathbb{C}}}

\newcommand{\PP}{\ensuremath{\mathbb{P}}}

\newcommand{\Pro}{\ensuremath{\mathbb{P}}}

\newcommand{\cX}{\ensuremath{\mathscr{X}}}

\newcommand{\cL}{\ensuremath{\mathscr{L}}}

\newcommand{\cU}{\ensuremath{\mathscr{U}}}

\newcommand{\eps}{\ensuremath{\varepsilon}}

\renewcommand{\R}{\ensuremath{\mathbb{R}}}
\renewcommand{\C}{\ensuremath{\mathbb{C}}}
\renewcommand{\Pro}{\ensuremath{\mathbb{P}}}

\renewcommand{\cU}{\ensuremath{\mathscr{U}}}

\newcommand{\Spec}{\ensuremath{\mathrm{Spec}\,}}

\newcommand{\dist}{\mathrm{dist}}

\newcommand{\Hom}{\mathrm{Hom}}

\newcommand{\an}{\mathrm{an}}

\newcommand{\NAMA}{\ensuremath{\mathrm{MA_{NA}}}}
\newcommand{\RMA}{\ensuremath{\mathrm{MA_\mathbb{R}}}}

\newcommand{\PMA}{\ensuremath{\mathrm{MA_{poly}}}}

\newcommand{\Star}{\mathrm{Star}}

\newcommand{\trop}{\operatorname{trop}}

\newcommand{\conv}{\operatorname{conv}}

\newcommand{\MA}{\operatorname{MA}}

\newcommand{\PSH}{\operatorname{PSH}}
\newcommand{\tX}{\mathbf{X}}

\newcommand{\Sk}{\mathrm{Sk}}

\newcommand{\T}{\mathbb{T}}
\DeclareMathOperator{\Ker}{Ker}

\DeclareMathOperator{\PA}{PA}
\DeclareMathOperator{\Aff}{Aff}
\DeclareMathOperator{\Lip}{Lip}
\DeclareMathOperator{\diam}{diam}
\DeclareMathOperator{\poly}{poly}
\DeclareMathOperator{\relint}{relint}
\DeclareMathOperator{\Cone}{Cone}

\DeclareMathOperator{\PC}{PC}
\DeclareMathOperator{\PCC}{PCC}

\DeclareMathOperator{\PCreg}{PC^{\mathrm{reg}}}

\DeclareMathOperator{\Vol}{Vol}

\DeclareMathOperator{\PAPC}{PAPC}

\DeclareMathOperator{\Capa}{Cap}
\DeclareMathOperator{\rec}{rec}

\DeclareMathOperator{\NA}{NA}
\newcommand{\on}[1]{\operatorname{#1}}

\newcommand{\ca}[1]{{\mathcal{#1}}}

%% file: References.bib
@article{BoucksomJonssonc,
     author = {S\'ebastien Boucksom and Mattias Jonsson},
     title = {Global pluripotential theory over a trivially valued field},
     journal = {Annales de la Facult\'e des sciences de Toulouse : Math\'ematiques},
     pages = {647--836},
     publisher = {Universit\'e Paul Sabatier, Toulouse},
     volume = {Ser. 6, 31},
     number = {3},
     year = {2022},
     doi = {10.5802/afst.1705},
     language = {en},
     url = {https://afst.centre-mersenne.org/articles/10.5802/afst.1705/}
}

@article{BEGZ,
author = {S{\'e}bastien Boucksom and Philippe Eyssidieux and Vincent Guedj and Ahmed Zeriahi},
title = {{Monge--Ampère equations in big cohomology classes}},
volume = {205},
journal = {Acta Mathematica},
number = {2},
publisher = {Institut Mittag-Leffler},
pages = {199 -- 262},
year = {2010},
doi = {10.1007/s11511-010-0054-7},
URL = {https://doi.org/10.1007/s11511-010-0054-7}
}

@book{ewald,
 author = {Ewald, G{\"u}nter},
 title = {Combinatorial convexity and algebraic geometry},
 fseries = {Graduate Texts in Mathematics},
 series = {Grad. Texts Math.},
 issn = {0072-5285},
 volume = {168},
 isbn = {0-387-94755-8},
 year = {1996},
 publisher = {New York, NY: Springer},
 language = {English},
 zbMATH = {960150},
 Zbl = {0869.52001}
}

@article{AminiPiquerez2020,
  author      = {Omid Amini and Matthieu Piquerez},
  date        = {2020},
  title       = {Hodge theory for tropical varieties},
  eprint      = {2007.07826},
  eprintclass = {math.AG},
  eprinttype  = {arXiv},
}

@article{AHK,
 author = {Adiprasito, Karim and Huh, June and Katz, Eric},
 title = {Hodge theory for combinatorial geometries},
 fjournal = {Annals of Mathematics. Second Series},
 journal = {Ann. Math. (2)},
 issn = {0003-486X},
 volume = {188},
 number = {2},
 pages = {381--452},
 year = {2018},
 language = {English},
 doi = {10.4007/annals.2018.188.2.1},
 zbMATH = {6921184},
 Zbl = {1442.14194}
}

@article{BGJM,
    author = {Burgos Gil, J. and Gubler, W. and Jell, Philipp and Martin, F.},
    title = "{Differentiability of non-archimedean volumes and non-archimedean Monge--Ampère equations, with an appendix by Lazarsfeld}",
    journal = {Algebraic Geometry},
    volume = {7},
    number = {(2)},
    pages = {113-152},
    year = {2020},
    doi = {10.14231/AG-2020-005},
}

@article{LiSurvey,
AUTHOR  =   {Li, Y.},
TITLE   =   {Survey on the metric {S}{Y}{Z} conjecture and non-archimedean geometry},
YEAR    =   {2022},
eprint =   {2204.11363},
archivePrefix={arXiv},
primaryClass={math.AG}
}

@Article{GotoYamamoto,
      title={Toric degenerations of Calabi--Yau complete intersections and metric SYZ conjecture}, 
      author={Keita Goto and Yuto Yamamoto},
      year={2024},
      eprint={2407.09133},
      archivePrefix={arXiv},
      primaryClass={math.AG},
      url={https://arxiv.org/abs/2407.09133}, 
}

@Article{GathmannKerberMarkwig2009,
  author       = {Gathmann, Andreas and Kerber, Michael and Markwig, Hannah},
  date         = {2009},
  journaltitle = {Compos. Math.},
  title        = {Tropical fans and the moduli spaces of tropical curves},
  doi          = {10.1112/S0010437X08003837},
  issn         = {0010-437X},
  number       = {1},
  pages        = {173--195},
  url          = {https://doi.org/10.1112/S0010437X08003837},
  volume       = {145},
  fjournal     = {Compositio Mathematica},
  mrclass      = {14N35 (14N10)},
  mrnumber     = {2480499},
  mrreviewer   = {Yunfeng Jiang},
}

@Article{KT,
    TITLE   =   {Non-archimedean {K}ähler geometry},
    AUTHOR  =   {Kontsevich, M. and Tschinkel, Y.},
    YEAR    =   {2002},
    NOTE    =   {Unpublished note}
}

@Article{APW1,
    TITLE   =   {Mixed differential calculus and Monge--Ampère equations on tropical varieties},
    AUTHOR  =   {Amini, O. and Piquerez, M. and Wu, L.},
    YEAR    =   {2026},
    NOTE    =   {In preparation}
}

@Article{APW2,
    TITLE   =   {On solutions to tropical Monge--Ampère equations},
    AUTHOR  =   {Amini, O. and Piquerez, M. and Wu, L.},
    YEAR    =   {2026},
    NOTE    =   {In preparation}
}

@InProceedings{GublerCLD,
author="Gubler, Walter",
editor="Baker, Matthew
and Payne, Sam",
title="Forms and Currents on the Analytification of an Algebraic Variety (After Chambert-Loir and Ducros)",
booktitle="Nonarchimedean and Tropical Geometry",
year="2016",
publisher="Springer International Publishing",
address="Cham",
pages="1--30",
isbn="978-3-319-30945-3"
}

@Article{CLD,
    TITLE   =   {Formes différentielles réelles et courants sur les espaces de {B}erkovich},
    AUTHOR  =   {Chambert-Loir, A. and Ducros, A.},
    YEAR    =   {2025},
    eprint={1204.6277},
      archivePrefix={arXiv},
      primaryClass={math.AG},
      url={https://arxiv.org/abs/1204.6277}, 
    }

@article{BHJ17,
 author = {Boucksom, S{\'e}bastien and Hisamoto, Tomoyuki and Jonsson, Mattias},
 title = {Uniform {{\(K\)}}-stability, {Duistermaat}-{Heckman} measures and singularities of pairs},
 fjournal = {Annales de l'Institut Fourier},
 journal = {Ann. Inst. Fourier},
 issn = {0373-0956},
 volume = {67},
 number = {2},
 pages = {743--841},
 year = {2017},
 language = {English},
 doi = {10.5802/aif.3096},
 zbMATH = {6821960},
 Zbl = {1391.14090}
}

@article{BHJ19,
 author = {Boucksom, S{\'e}bastien and Hisamoto, Tomoyuki and Jonsson, Mattias},
 title = {Uniform {K}-stability and asymptotics of energy functionals in {K{\"a}hler} geometry},
 fjournal = {Journal of the European Mathematical Society (JEMS)},
 journal = {J. Eur. Math. Soc. (JEMS)},
 issn = {1435-9855},
 volume = {21},
 number = {9},
 pages = {2905--2944},
 year = {2019},
 language = {English},
 doi = {10.4171/JEMS/894},
 zbMATH = {7117723},
 Zbl = {1478.53115}
}

@article{BBGZ,
     author = {Berman, Robert J. and Boucksom, S\'ebastien and Guedj, Vincent and Zeriahi, Ahmed},
     title = {A variational approach to complex {Monge--Amp\`ere} equations},
     journal = {Publications Math\'ematiques de l'IH\'ES},
     pages = {179--245},
     year = {2013},
     publisher = {Springer-Verlag},
     volume = {117},
     doi = {10.1007/s10240-012-0046-6},
     zbl = {1277.32049},
     language = {en},
     url = {https://www.numdam.org/articles/10.1007/s10240-012-0046-6/}
}

@article{BGK,
 author = {Burgos Gil, J. and Gubler, W. and Künnemann, K.},
 title = {A tropical formula for non-Archimedean local heights},
 year = {2025},
 eprint =   {2512.07431},
 eprintclass    =   {math.AG},
 eprinttype =   {arXiv}
}

@book{gutierrez,
 author = {Guti{\'e}rrez, Cristian E.},
 title = {The {Monge}--{Amp{\`e}re} equation},
 edition = {2nd edition},
 fseries = {Progress in Nonlinear Differential Equations and Their Applications},
 series = {Prog. Nonlinear Differ. Equ. Appl.},
 issn = {1421-1750},
 volume = {89},
 isbn = {978-3-319-43372-1; 978-3-319-43374-5},
 year = {2016},
 publisher = {Basel: Birkh{\"a}user/Springer},
 language = {English},
 doi = {10.1007/978-3-319-43374-5},
 zbMATH = {6610115},
 Zbl = {1356.35004}
}

@Book{CoxLittleSchenck2011,
  author     = {Cox, David A. and Little, John B. and Schenck, Henry K.},
  title      = {Toric varieties},
  year       = {2011},
  volume     = {124},
  series     = {Graduate Studies in Mathematics},
  publisher  = {American Mathematical Society, Providence, RI},
  isbn       = {978-0-8218-4819-7},
  pages      = {xxiv+841},
  doi        = {10.1090/gsm/124},
  url        = {http://dx.doi.org/10.1090/gsm/124},
  mrclass    = {14M25 (05A15 05E45 52B12)},
  mrnumber   = {2810322},
  mrreviewer = {Ivan V. Arzhantsev},
}

@incollection{baker-faber,
 author = {Baker, Matthew and Faber, Xander},
 title = {Metrized graphs, {Laplacian} operators, and electrical networks},
 booktitle = {Quantum graphs and their applications. Proceedings of an AMS-IMS-SIAM joint summer research conference on quantum graphs and their applications, Snowbird, UT, USA, June 19--23, 2005},
 isbn = {0-8218-3765-6},
 pages = {15--33},
 year = {2006},
 publisher = {Providence, RI: American Mathematical Society (AMS)},
 language = {English},
 zbMATH = {5082564},
 Zbl = {1114.94025}
}

@Article{BedfordTaylor1982,
  author       = {Bedford, Eric and Taylor, B. A.},
  date         = {1982},
  journaltitle = {Acta Mathematica},
  title        = {A new capacity for plurisubharmonic functions},
  doi          = {10.1007/bf02392348},
  issn         = {0001-5962},
  number       = {0},
  pages        = {1--40},
  volume       = {149},
  publisher    = {International Press of Boston},
}

@Book{Fulton1993,
  author     = {Fulton, W.},
  title      = {Introduction to toric varieties},
  year       = {1993},
  volume     = {131},
  series     = {Annals of Mathematics Studies},
  note       = {The William H. Roever Lectures in Geometry},
  publisher  = {Princeton University Press, Princeton, NJ},
  isbn       = {0-691-00049-2},
  pages      = {xii+157},
  doi        = {10.1515/9781400882526},
  url        = {http://dx.doi.org/10.1515/9781400882526},
  mrclass    = {14M25 (14-02 14J30)},
  mrreviewer = {T. Oda},
}

@Book{Berkovich1990,
  author     = {Berkovich, V. G.},
  title      = {Spectral theory and analytic geometry over non-{A}rchimedean fields},
  year       = {1990},
  volume     = {33},
  series     = {Mathematical Surveys and Monographs},
  publisher  = {American Mathematical Society, Providence, RI},
  isbn       = {0-8218-1534-2},
  pages      = {x+169},
  mrclass    = {32P05 (32C15 32C37 46S10 47S10)},
  mrreviewer = {W. Bartenwerfer},
}

@Article{MustataNicaise,
  author     = {Mustata, M. and {Nicaise}, J.},
  title      = {Weight functions on non-{A}rchimedean analytic spaces and the {K}ontsevich-{S}oibelman skeleton},
  doi        = {10.14231/AG-2015-016},
  issn       = {2214-2584},
  number     = {3},
  pages      = {365--404},
  url        = {http://dx.doi.org/10.14231/AG-2015-016},
  volume     = {2},
  fjournal   = {Algebraic Geometry},
  journal    = {Algebr. Geom.},
  mrclass    = {14G22 (13A18 14F17)},
  mrreviewer = {Alessandra Bertapelle},
  year       = {2015},
}

@Article{NicaiseXu,
  author   = {{Nicaise}, J. and {Xu}, C.},
  title    = {The essential skeleton of a degeneration of algebraic varieties},
  year     = {2016},
  volume   = {138},
  number   = {6},
  pages    = {1645--1667},
  issn     = {0002-9327},
  doi      = {10.1353/ajm.2016.0049},
  url      = {http://dx.doi.org/10.1353/ajm.2016.0049},
  fjournal = {American Journal of Mathematics},
  journal  = {Amer. J. Math.},
  mrclass  = {14E30 (14G22 14J32)},
}

@Article{StromingerYauZaslow,
  author    = {Strominger, A. and Yau, {S. T.} and Zaslow, E.},
  title     = {Mirror symmetry is T-duality},
  year      = {1996},
  volume    = {479},
  number    = {1-2},
  month     = {11},
  pages     = {243--259},
  issn      = {0550-3213},
  journal   = {Nuclear Physics B},
  publisher = {Elsevier},
}

@InBook{KontsevichSoibelman,
  author    = {Kontsevich, M. and Soibelman, Y.},
  title     = {Affine Structures and Non-Archimedean Analytic Spaces},
  booktitle = {The Unity of Mathematics: In Honor of the Ninetieth Birthday of I.M. Gelfand},
  year      = {2006},
  editor    = {Etingof, Pavel and Retakh, Vladimir and Singer, I. M.},
  publisher = {Birkh{\"a}user Boston},
  pages     = {321--385},
  address   = {Boston, MA},
}

@Article{GublerJellKuennemannEtAl,
  author     = {{Gubler}, W. and {Jell}, P. and {Kuennemann}, K. and {Martin}, F.},
  title      = {{Continuity of Plurisubharmonic Envelopes in Non-Archimedean Geometry and Test Ideals}},
  issn       = {0373-0956},
  note       = {With an appendix by Jos\'{e} Ignacio Burgos Gil and Mart\'{\i}n Sombra},
  number     = {5},
  pages      = {2331--2376},
  url        = {http://aif.cedram.org/item?id=AIF_2019__69_5_2331_0},
  volume     = {69},
  fjournal   = {Universit\'{e} de Grenoble. Annales de l'Institut Fourier},
  journal    = {Ann. Inst. Fourier (Grenoble)},
  mrclass    = {32P05 (14G22 14M25 32U05)},
  mrnumber   = {4018262},
  mrreviewer = {Mattias Jonsson},
  year       = {2019},
}

@InCollection{OssermanRabinoff2013,
  author     = {Osserman, Brian and Rabinoff, Joseph},
  booktitle  = {Tropical and non-{A}rchimedean geometry},
  date       = {2013},
  title      = {Lifting nonproper tropical intersections},
  doi        = {10.1090/conm/605/12110},
  isbn       = {978-1-4704-1021-6},
  pages      = {15--44},
  publisher  = {Amer. Math. Soc., Providence, RI},
  series     = {Contemp. Math.},
  url        = {https://doi.org/10.1090/conm/605/12110},
  volume     = {605},
  mrclass    = {14T05 (14C17 14G20 14M25)},
  mrnumber   = {3204266},
  mrreviewer = {Joaquim\ Ro\'{e}},
}

@Article{Paynea,
  author     = {Payne, Sam},
  title      = {Analytification is the limit of all tropicalizations},
  journal    = {Math. Res. Lett.},
  year       = {2009},
  volume     = {16},
  number     = {3},
  pages      = {543--556},
  issn       = {1073-2780},
  url        = {https://doi.org/10.4310/MRL.2009.v16.n3.a13},
  fjournal   = {Mathematical Research Letters},
  mrclass    = {14T05 (14G22)},
  mrnumber   = {2511632},
  mrreviewer = {Joaquim Ro\'e},
}

@article{hom-smooth,
 author = {Amini, Omid and Piquerez, Matthieu},
 title = {Homological smoothness and {Deligne} resolution for tropical fans},
 year = {2024},
 eprint      = {2405.05718},
  eprintclass = {math.AG},
  eprinttype  = {arXiv},
}

@Article{non-arch-stab,
  author        = {{Boucksom}, Sébastian and {Jonsson}, Mattias},
  title         = {Measures of finite energy in pluripotential theory: A synthetic approach.},
  year          = {2023},
  eprint        = {2307.01697},
  archiveprefix = {arXiv},
  primaryclass  = {math.AG},
}

@Article{BoucksomJonsson2017,
  author     = {{Boucksom}, S. and {Jonsson}, M.},
  title      = {Tropical and non-{A}rchimedean limits of degenerating families of volume forms},
  doi        = {10.5802/jep.39},
  issn       = {2429-7100},
  pages      = {87--139},
  url        = {https://doi.org/10.5802/jep.39},
  volume     = {4},
  fjournal   = {Journal de l'\'Ecole polytechnique. Math\'ematiques},
  journal    = {J. \'Ec. polytech. Math.},
  mrclass    = {32Q25 (14G22 14J32 14T05 32P05 53C23)},
  mrnumber   = {3611100},
  mrreviewer = {Dan Abramovich},
  year       = {2017},
}

@Article{YamaSYZ,
      title={Non-archimedean SYZ fibrations via tropical contractions}, 
      author={Yuto Yamamoto},
      year={2024},
      eprint={2404.04972},
      archivePrefix={arXiv},
      primaryClass={math.AG},
      url={https://arxiv.org/abs/2404.04972}, 
}

@article{zbMATH02244752,
 author = {Ardila, Federico and Klivans, Caroline J.},
 title = {The {Bergman} complex of a matroid and phylogenetic trees},
 fjournal = {Journal of Combinatorial Theory. Series B},
 journal = {J. Comb. Theory, Ser. B},
 issn = {0095-8956},
 volume = {96},
 number = {1},
 pages = {38--49},
 year = {2006},
 language = {English},
 doi = {10.1016/j.jctb.2005.06.004},
 zbMATH = {2244752},
 Zbl = {1082.05021}
}

@article{Liu,
 author = {Liu, Yifeng},
 title = {A non-archimedean analogue of the {Calabi}-{Yau} theorem for totally degenerate abelian varieties},
 fjournal = {Journal of Differential Geometry},
 journal = {J. Differ. Geom.},
 issn = {0022-040X},
 volume = {89},
 number = {1},
 pages = {87--110},
 year = {2011},
 language = {English},
 doi = {10.4310/jdg/1324476752},
 zbMATH = {6024983},
 Zbl = {1254.14026}
}

@book{Figalli,
 author = {Figalli, Alessio},
 title = {The {Monge}--{Amp{\`e}re} equation and its applications},
 fseries = {Zurich Lectures in Advanced Mathematics},
 series = {Zur. Lect. Adv. Math.},
 isbn = {978-3-03719-170-5; 978-3-03719-670-0},
 year = {2017},
 publisher = {Z{\"u}rich: European Mathematical Society (EMS)},
 language = {English},
 doi = {10.4171/170},
 zbMATH = {6669881},
 Zbl = {1435.35003}
}

@article{smooth-higher,
 author = {Itenberg, Ilia and Katzarkov, Ludmil and Mikhalkin, Grigory and Zharkov, Ilia},
 title = {Tropical homology},
 fjournal = {Mathematische Annalen},
 journal = {Math. Ann.},
 issn = {0025-5831},
 volume = {374},
 number = {1-2},
 pages = {963--1006},
 year = {2019},
 language = {English},
 doi = {10.1007/s00208-018-1685-9},
 zbMATH = {7070539},
 Zbl = {1460.14146}
}

@article{Jell,
 author = {Jell, Philipp},
 title = {Constructing smooth and fully faithful tropicalizations for {Mumford} curves},
 fjournal = {Selecta Mathematica. New Series},
 journal = {Sel. Math., New Ser.},
 issn = {1022-1824},
 volume = {26},
 number = {4},
 pages = {23},
 note = {Id/No 60},
 year = {2020},
 language = {English},
 doi = {10.1007/s00029-020-00586-2},
 zbMATH = {7239259},
 Zbl = {1440.14283}
}

@article{yau,
 author = {Yau, Shing-Tung},
 title = {On the {Ricci} curvature of a compact {K{\"a}hler} manifold and the complex {Monge}--{Amp{\`e}re} equation. {I}},
 fjournal = {Communications on Pure and Applied Mathematics},
 journal = {Commun. Pure Appl. Math.},
 issn = {0010-3640},
 volume = {31},
 pages = {339--411},
 year = {1978},
 language = {English},
 doi = {10.1002/cpa.3160310304},
 zbMATH = {3576316},
 Zbl = {0369.53059}
}

@Book{MaclaganSturmfels2015,
  author     = {Maclagan, Diane and Sturmfels, Bernd},
  title      = {Introduction to tropical geometry},
  isbn       = {978-0-8218-5198-2},
  pages      = {xii+363},
  publisher  = {American Mathematical Society, Providence, RI},
  series     = {Graduate Studies in Mathematics},
  volume     = {161},
  mrclass    = {14T05 (05B35 14M25 15A80 52B70)},
  mrnumber   = {3287221},
  mrreviewer = {Patrick Popescu-Pampu},
  year       = {2015},
}

@article {FS,
    AUTHOR = {Fulton, W. and Sturmfels, B.},
     TITLE = {Intersection theory on toric varieties},
   JOURNAL = {	Topology},
      VOLUME = {36},
      YEAR = {1997},
    NUMBER = {2},
     PAGES = {335--353},
}

@article{matroids-lagrangian,
 author = {Ardila, Federico and Denham, Graham and Huh, June},
 title = {Lagrangian geometry of matroids},
 fjournal = {Journal of the American Mathematical Society},
 journal = {J. Am. Math. Soc.},
 issn = {0894-0347},
 volume = {36},
 number = {3},
 pages = {727--794},
 year = {2023},
 language = {English},
 doi = {10.1090/jams/1009},
 zbMATH = {7682634},
 Zbl = {1512.05068}
}

@Article{Gross2005,
  author       = {Gross, Mark},
  date         = {2005},
  journaltitle = {Math. Ann.},
  title        = {Toric degenerations and {B}atyrev-{B}orisov duality},
  doi          = {10.1007/s00208-005-0686-7},
  issn         = {0025-5831},
  number       = {3},
  pages        = {645--688},
  url          = {https://doi.org/10.1007/s00208-005-0686-7},
  volume       = {333},
  fjournal     = {Mathematische Annalen},
  mrclass      = {14J32 (14M25)},
  mrnumber     = {2198802},
  mrreviewer   = {Diego Matessi},
}

@Article{BoucksomFavreJonsson2015,
  author       = {Boucksom, S\'{e}bastien and Favre, Charles and Jonsson, Mattias},
  date         = {2015},
  journaltitle = {J. Amer. Math. Soc.},
  title        = {Solution to a non-{A}rchimedean {M}onge-{A}mp\`ere equation},
  doi          = {10.1090/S0894-0347-2014-00806-7},
  issn         = {0894-0347},
  number       = {3},
  pages        = {617--667},
  url          = {https://doi.org/10.1090/S0894-0347-2014-00806-7},
  volume       = {28},
  fjournal     = {Journal of the American Mathematical Society},
  mrclass      = {32P05 (32W20)},
  mrnumber     = {3327532},
  mrreviewer   = {Christian Lehn},
}

@article{AndreassonHultgren,
      title={Solvability of Monge--Amp\`ere equations and tropical affine structures on reflexive polytopes}, 
      author={Rolf Andreasson and Jakob Hultgren},
      year={2023},
      eprint={2303.05276},
      archivePrefix={arXiv},
      primaryClass={math.DG},
      url={https://arxiv.org/abs/2303.05276}, 
}

@article{LiNA,
author = {Yang Li},
title = {{Metric SYZ conjecture and non-Archimedean geometry}},
volume = {172},
journal = {Duke Mathematical Journal},
number = {17},
publisher = {Duke University Press},
pages = {3227 -- 3255},
year = {2023},
doi = {10.1215/00127094-2022-0099},
URL = {https://doi.org/10.1215/00127094-2022-0099}
}

@article{Wanner,
author = {Wanner, V.},
title = {{Comparison of two notions of subharmonicity on non-archimedean curves}},
volume = {293},
journal = {Math. Zeitschrift},
issue = {1},
pages = {432-1823},
year = {2019},
doi = {0.1007/s00209-018-2205-z},
}

@Article{Li2024,
  author       = {Li, Yang},
  date         = {2024},
  journaltitle = {Cambridge Journal of Mathematics},
  title        = {Metric SYZ conjecture for certain toric Fano hypersurfaces},
  doi          = {10.4310/cjm.2024.v12.n1.a3},
  issn         = {2168-0949},
  number       = {1},
  pages        = {223--252},
  volume       = {12},
  publisher    = {International Press of Boston},
}

@Book{BurgosPhilipponSombra,
  author       = {José Ignacio {Burgos Gil} and Patrice Philippon and Martin Sombra},
  date         = {2014},
  title = {Arithmetic geometry of toric varieties. Metrics, measures and heights.},
  publisher = {Astérisque, no. 360},
  url       = {http://numdam.org/item/AST_2014__360__R1_0/},
}

@article{BurgosSombra2010,
author = {Burgos Gil, José and Sombra, Martín},
year = {2010},
month = {08},
pages = {},
title = {When do the Recession Cones of a Polyhedral Complex Form a Fan?},
volume = {46},
journal = {Discrete and Computational Geometry},
doi = {10.1007/s00454-010-9318-4},
}

@Article{BurgosGilGublerJellKuennemann2021,
      title={Pluripotential theory for tropical toric varieties and non-archimedean Monge--Amp\'ere equations}, 
      author={José Ignacio {Burgos Gil} and Walter Gubler and Philipp Jell and Klaus Künnemann},
      year={2021},
      eprint={2102.07392},
      archivePrefix={arXiv},
      primaryClass={math.AG},
      url={https://arxiv.org/abs/2102.07392}, 
}

@Article{Vilsmeier2021,
  author       = {Vilsmeier, Christian},
  date         = {2021},
  journaltitle = {Math. Z.},
  title        = {A comparison of the real and non-archimedean {M}onge-{A}mp\`ere operator},
  doi          = {10.1007/s00209-020-02527-3},
  issn         = {0025-5874},
  number       = {1-2},
  pages        = {633--668},
  url          = {https://doi.org/10.1007/s00209-020-02527-3},
  volume       = {297},
  fjournal     = {Mathematische Zeitschrift},
  mrclass      = {32P05 (14G22 14T20 32W20)},
  mrnumber     = {4204708},
}

@article{mihatsch,
 author = {Mihatsch, Andreas},
 title = {On tropical intersection theory},
 fjournal = {Selecta Mathematica. New Series},
 journal = {Sel. Math., New Ser.},
 issn = {1022-1824},
 volume = {29},
 number = {2},
 pages = {33},
 note = {Id/No 17},
 year = {2023},
 language = {English},
 doi = {10.1007/s00029-022-00818-7},
 zbMATH = {7679754},
 Zbl = {1527.14119}
}

@Article{HultgrenJonssonMazzonMcCleerey2022,
      title={Tropical and non-Archimedean Monge--Ampère equations for a class of Calabi-Yau hypersurfaces}, 
      author={Hultgren, Jakob and Jonsson, Mattias and Mazzon, Enrica and McCleerey, Nicholas},
      year={2022},
      eprint={2208.13697},
      archivePrefix={arXiv},
      primaryClass={math.AG},
      url={https://arxiv.org/abs/2208.13697}, 
}

@phdthesis{Thuillier,
  TITLE = {{Th{\'e}orie du potentiel sur les courbes en g{\'e}om{\'e}trie analytique non archim{\'e}dienne. Applications {\`a} la th{\'e}orie d'Arakelov}},
  AUTHOR = {Thuillier, Amaury},
  URL = {https://theses.hal.science/tel-00010990},
  SCHOOL = {{Universit{\'e} Rennes 1}},
  YEAR = {2005},
  TYPE = {Thèses},
  HAL_ID = {tel-00010990},
  HAL_VERSION = {v1},
}

@article{Yamamoto2021,
  author    = {Yamamoto, Yuto},
  date      = {2021},
  title     = {Tropical contractions to integral affine manifolds with singularities},
  eprint       = {2105.10141},
  eprintclass       = {math.AG},
  eprinttype = {arXiv}
}

@Article{BoteroGil2022,
  author       = {Botero, Ana Mar\'{\i}a and Burgos Gil, Jos\'{e} Ignacio},
  date         = {2022},
  journaltitle = {Math. Z.},
  title        = {Toroidal b-divisors and {M}onge-{A}mp\`ere measures},
  doi          = {10.1007/s00209-021-02789-5},
  issn         = {0025-5874},
  number       = {1},
  pages        = {579--637},
  url          = {https://doi.org/10.1007/s00209-021-02789-5},
  volume       = {300},
  fjournal     = {Mathematische Zeitschrift},
  mrclass      = {14C17 (14T15 26B25 32W20)},
  mrnumber     = {4359537},
  mrreviewer   = {Xia Liao},
}

@book{Rockafellar,
 author = {Rockafellar, R. Tyrrell},
 title = {Convex analysis},
 fseries = {Princeton Mathematical Series},
 series = {Princeton Math. Ser.},
 volume = {28},
 year = {1970},
 publisher = {Princeton University Press, Princeton, NJ},
 language = {English},
 doi = {10.1515/9781400873173},
 zbMATH = {3307135},
 Zbl = {0193.18401}
}

@article{DoVu,
 author = {Do, Duc Thai and Vu, Duc-Viet},
 title = {Complex {Monge}--{Amp{\`e}re} equations with solutions in finite energy classes},
 fjournal = {Mathematical Research Letters},
 journal = {Math. Res. Lett.},
 issn = {1073-2780},
 volume = {29},
 number = {6},
 pages = {1659--1683},
 year = {2022},
 language = {English},
 doi = {10.4310/MRL.2022.v29.n6.a2},
 zbMATH = {7682672},
 Zbl = {1523.32064}
}

@article{NGL,
 author = {Di Nezza, Eleonora and Guedj, Vincent and Lu, Chinh H.},
 title = {Finite entropy vs finite energy},
 fjournal = {Commentarii Mathematici Helvetici},
 journal = {Comment. Math. Helv.},
 issn = {0010-2571},
 volume = {96},
 number = {2},
 pages = {389--419},
 year = {2021},
 doi = {10.4171/CMH/515},
 zbMATH = {7367627},
 Zbl = {1473.32014}
}

@article{yuan-zhang,
 author = {Yuan, Xinyi and Zhang, Shou-Wu},
 title = {The arithmetic {Hodge} index theorem for adelic line bundles},
 fjournal = {Mathematische Annalen},
 journal = {Math. Ann.},
 issn = {0025-5831},
 volume = {367},
 number = {3-4},
 pages = {1123--1171},
 year = {2017},
 doi = {10.1007/s00208-016-1414-1},
 zbMATH = {6706027},
 Zbl = {1372.14017}
}

@Article{BK,
      title={On the height of the universal abelian variety}, 
      author={Burgos Gil, Jos{\'e} Ignacio and Kramer, J{\"u}rg},
      year={2024},
      eprint={2403.11745},
      archivePrefix={arXiv},
      primaryClass={math.NT},
      url={https://arxiv.org/abs/2403.11745}, 
}

@article{GY,
 author = {Cai, Yulin and Gubler, Walter},
 title = {Abstract divisorial spaces and arithmetic intersection numbers},
 date = {2024},
  eprint =   {2409.00611},
  eprintclass = {math.AG},
  eprinttype  = {arXiv}
}

@article{BBS,
 author = {Botero, Ana Mar{\'{\i}}a and Burgos Gil, Jos{\'e} Ignacio and Sombra, Mart{\'{\i}}n},
 title = {Convex analysis on polyhedral spaces},
 fjournal = {Mathematische Zeitschrift},
 journal = {Math. Z.},
 issn = {0025-5874},
 volume = {301},
 number = {2},
 pages = {1631--1674},
 year = {2022},
 doi = {10.1007/s00209-021-02869-6},
 zbMATH = {7525051},
 Zbl = {1535.26009}
}

@Article{BoucksomFavreJonsson,
  author       = {Boucksom, S. and Favre, Charles and Jonsson, Mattias},
  date         = {2016},
  journaltitle = {Journal of Algebraic Geometry},
  title        = {Singular semipositive metrics in non-{A}rchimedean geometry},
  doi          = {10.1090/jag/656},
  issn         = {1056-3911},
  number       = {1},
  pages        = {77--139},
  url          = {https://doi.org/10.1090/jag/656},
  volume       = {25},
  fjournal     = {Journal of Algebraic Geometry},
  journal      = {J. Algebraic Geom.},
  mrclass      = {14G22},
  mrnumber     = {3419957},
  mrreviewer   = {Marco A. Garuti},
  year         = {2016},
}

@article{delta-forms,
 author = {Gubler, Walter and K{\"u}nnemann, Klaus},
 title = {A tropical approach to nonarchimedean {Arakelov} geometry},
 fjournal = {Algebra \& Number Theory},
 journal = {Algebra Number Theory},
 issn = {1937-0652},
 volume = {11},
 number = {1},
 pages = {77--180},
 year = {2017},
 language = {English},
 doi = {10.2140/ant.2017.11.77},
 zbMATH = {6679113},
 Zbl = {1386.14096}
}

@Article{AllermannRau2010,
  author       = {Allermann, Lars and Rau, Johannes},
  date         = {2010},
  journaltitle = {Mathematische Zeitschrift},
  title        = {First steps in tropical intersection theory},
  doi          = {10.1007/s00209-009-0483-1},
  issn         = {1432-1823},
  number       = {3},
  pages        = {633--670},
  url          = {https://doi.org/10.1007/s00209-009-0483-1},
  volume       = {264},
  refid        = {Allermann2010},
}
